\documentclass{amsart}

\usepackage{amsmath,amssymb}
\usepackage{graphicx}
\usepackage{rotating,subfigure}
\usepackage{xcolor}
\usepackage{star}

\usepackage[colorlinks,
linkcolor=blue,
anchorcolor=blue,
citecolor=blue
]{hyperref}

\newcommand {\cov}{\textnormal {cov}}
\newcommand {\var}{\textnormal {var}}
\newcommand \ri{\mathrm{i}}
\newcommand \fd{\mathfrak{d}}

\begin{document}



\title[Spectral analysis of GL from high-dimensional noisy data]{Impact of signal-to-noise ratio and bandwidth on graph Laplacian spectrum from high-dimensional noisy point cloud} 

\author{Xiucai Ding}

\address{Department of Statistics, University of California, Davis, CA, USA} 
\email{xcading@ucdavis.edu}.

\author{Hau-Tieng Wu}
\address{Department of Mathematics and Department of Statistical Science, Duke University, Durham, NC, USA}
\email{hauwu@math.duke.edu}

\maketitle

\begin{abstract}
We systematically study the spectrum of kernel-based graph Laplacian (GL) constructed from high-dimensional and noisy random point cloud in the nonnull setup. The problem is motived by studying the model when the clean signal is sampled from a manifold that is embedded in a low-dimensional Euclidean subspace, and corrupted by high-dimensional noise. We quantify how the signal and noise interact over different regions of signal-to-noise ratio (SNR), and report the resulting peculiar spectral behavior of GL. In addition, we explore the impact of chosen kernel bandwidth on the spectrum of GL over different regions of SNR, which lead to an adaptive choice of kernel bandwidth that coincides with the common practice in real data. This result paves the way to a theoretical understanding of how practitioners apply GL when the dataset is noisy.
\end{abstract}

\section{Introduction}

Spectral algorithms are popular and have been widely applied in unsupervised machine learning, like eigenmap \cite{Belkin_Niyogi:2003}, locally linear embedding (LLE) \cite{Roweis_Saul:2000}, diffusion map (DM) \cite{Coifman_Lafon:2006}, to name but a few. A common ground of these algorithms is graph Laplacian (GL) and its spectral decomposition that have been extensively discussed in the spectral graph theory \cite{MR1421568}. 
Up to now, due to its wide practical applications, there have been rich theoretical results discussing the spectral behavior of GL under the manifold model when the point cloud is clean. For example, the pointwise convergence \cite{Belkin_Niyogi:2003,MR2332434, gine2006empirical,SINGER2006128}, the $L^2$ convergence without rate \cite{belkin2007convergence,MR2396807,Singer_Wu:2012}, the $L^2$ convergence with rate \cite{MR4130541}, the $L^\infty$ convergence with rate \cite{dunson2019diffusion}, etc. We refer readers to the cited literature therein for more information. Also see, for example \cite{koltchinskii2000random,JMLR:v7:braun06a}, for more relevant results. 
However, to our knowledge, limited results are known about the GL spectrum when the dataset is contaminated by noise \cite{Steinerburger2016,elkaroui20102,el2016graph,shen2020scalability}, particularly when the noise is high-dimensional \cite{elkaroui20102,el2016graph,shen2020scalability}. 
Specifically, in the high-dimensional setup, \cite{elkaroui20102,el2016graph} report controls on the operator norm between the clean GL and noisy GL in some specific signal-to-noise (SNR) regions. However, a finer and more complete description of the spectrum is still lacking. The main focus of this work is extending existing results so that the spectral behavior of GL is better depicted when the dataset is contaminated by high-dimensional noise.

GL is not uniquely defined, and there are several possible constructions from a given point cloud. For example, it can be constructed by the idea of local barycentric coordinates like that in LLE \cite{Wu_Wu:2017} or by taking landmarks like that in Roseland \cite{shen2020scalability}, the kernel can be asymmetric \cite{Coifman_Lafon:2006}, the metric can be non-Euclidean \cite{talmon2013empirical}, etc. In this paper, we focus ourselves on a specific setup; that is, GL is constructed from a random point cloud $\mathcal{X}:=\{\xb_i\}_{i=1}^n\subset \mathbb{R}^p$ via a symmetric kernel with the usual Euclidean metric.
To be more specific, from $\mathcal{X}$, construct the {\em affinity matrix} (or {\em kernel matrix}) $\Wb\in \mathbb{R}^{n\times n}$ by  
\begin{equation}\label{eq_defnw}
\Wb(i,j)=\exp\left(-\upsilon\frac{\|\xb_i-\xb_j \|^2}{h}\right), \ 1 \leq i, j \leq n, 
\end{equation}
where $\upsilon>0$ is the chosen parameter and $h \equiv h(n)$ is 
the chosen {\em bandwidth}. In other words, we focus on kernels of the exponential type, that is, $f(x) = \exp(-\upsilon x)$, to simplify the discussion. 
Then, define the {\em transition matrix} 
\begin{equation}\label{eq_noramlizedmatrix}
\Ab=\Db^{-1}\Wb,
\end{equation}  
where $\Db$ is the {\em degree matrix}, which is a diagonal matrix with diagonal entries defined as 
\begin{equation}\label{eq_degreematrix}
\Db(i,i)=\sum_{j=1}^n \Wb(i,j), \ i=1,2,\cdots, n. 
\end{equation}
Note that $\Ab$ is row-stochastic. The (normalized) GL is defined as 
\[
\Lb:=\frac{1}{h}(\mathbf{I}-\Ab)\,. 
\]
Since $\Lb$ and $\Ab$ are related by an {isotropic} shift and {universal} scaling, we focus on studying the spectral distributions of $\Wb$ and $\Ab$ in the rest of the paper. With the eigendecomposition of $\Ab$, we could proceed with several data analysis missions. For example, spectral clustering is carried out by combining top few non-trivial eigenvectors and $k$-mean \cite{MR2409803}, embedding datasets to a low-dimensional Euclidean space by eigenvectors and/or eigenvalues reduces the dataset dimension \cite{Belkin_Niyogi:2003,Coifman_Lafon:2006}, etc. Thus, the spectral behavior of $\Ab$ is the key to fully understanding these algorithms.

\subsection{Mathematical setup}\label{sec:Intro math setup}
We now specify the high-dimensional noisy model and the problem we are concerned with in this work.
Suppose $\mathcal{X}$ is independent and identically {distributed} (i.i.d.)  following a sub-Gaussian random vector 
\begin{equation}\label{eq_modelsteptwo} 
\xb_i=\zb_i+\yb_i, \ 1 \leq i \leq n\,,
\end{equation}
where $\zb_i$ is a random vector with mean 0 and $\cov(\zb_i)=\text{diag}\{\lambda_1,\cdots, \lambda_d,0,\cdots,0\}$, $\lambda_1\geq\lambda_2\geq\ldots\geq\lambda_d>0$, and $\mathbf{y}_i$ is sub-Gaussian with independent entries,
\begin{equation}\label{eq_yassum1}
\mathbb{E}(\mathbf{y}_i)=\mathbf{0} \ \mbox{ and } \ \cov(\yb_i)=\mathbf{I}_p\,.
\end{equation}
We also assume that 
$\zb_i$ and $\yb_i$ are independent.
As a result, $\xb_i$ is a sub-Gaussian random vector with mean 0 and covariance
\begin{equation}\label{eq_sigdefn}
\Sigma=\text{diag}\{\lambda_1+1,\cdots, \lambda_d+1,1,\cdots, 1\}\,,
\end{equation}
where $d\in \mathbb{N}$.  In this model, $\zb_i$ and $\yb_i$ represent the signal and noise part of the point cloud respectively.
To simplify the discussion, we also assume that $\mathbf{z}_{i1},\ldots,\mathbf{z}_{id}$ are continuous random variables.
Denote $\mathcal{Y}:=\{\mathbf{y}_i\}_{i=1}^n \in \mathbb{R}^p$ and $\mathcal{Z}:=\{\zb_i\}_{i=1}^n$. 

We adopt the {high-dimensional setting} and assume that 
\begin{equation}\label{eq_ratio}
\gamma \leq c_n:=\frac{n}{p} \leq \gamma^{-1}
\end{equation}
for some constant $0<\gamma \leq 1$; in other words, we focus on the {\em large $p$ and large $n$} setup. We also assume that $d$ is independent of $n$. Note that  \eqref{eq_modelsteptwo} is closely related to the commonly applied {\em spiked covariance model} \cite{johnstone2001}.

Next, we consider the following setup to control the signal strength: 
\begin{equation}\label{eq_lambdadefinition}
  \lambda_i \asymp n^{\alpha_i}, \ 0 \leq \alpha_i < \infty  \ \text{are some constants}.  
\end{equation}
Thus, $\lambda_i$ represents the {\em signal strength} in the $i$-th component of the signal. On the other hand, the {\em total noise} in the dataset $\mathcal{X}$ is $\texttt{tr}(\cov(\yb_i))=p$, and we call $\lambda_i/p$ the {\em signal to noise ratio} (SNR) associated with the $i$-th component of the signal.

We denote the matrix  $\Wb_1 \in \mathbb{R}^{n \times n}$ such that 
\begin{equation}\label{eq_singlematrix}
\Wb_1(i,j)=\exp\left(-\upsilon\frac{\| \zb_i-\zb_j \|_2^2}{h}\right)\,.
\end{equation}
Note that compared with $\Wb$ defined in \eqref{eq_defnw}, $\Wb_1$ is the affinity matrix constructed from the clean signal $\mathcal{Z}$. 
Moreover, let $\Wb_y$ be the affinity matrix associated with $\mathcal{Y},$ which represents the noise part. Their associated transition matrices $\Ab_1$ and $\Ab_y$ are defined similarly as in (\ref{eq_noramlizedmatrix}). 
{We will show in part (1) of Theorem \ref{thm_informativeregion}, part (2) of Theorem \ref{thm_adaptivechoiceofc} and Corollary \ref{coro_adaptivechoiceofc}  that under \eqref{eq_modelsteptwo}, when the SNR is above some threshold and the bandwidth is properly chosen, we can separate the signal and noise parts in the sense that $\Wb$ is close to
\begin{equation}\label{eq_separationsinglenoise}
\Wb_1 \circ \Wb_y 
\end{equation}
 with high probability, where $\circ$ stands for the Hadamard product. The closeness is quantified using the normalized operator norm; that is, $n^{-1}\| \Wb-\Wb_1 \circ \Wb_y\|=o(1)$ with high probability. In fact, using the definition of $\Wb_c$ below in (\ref{eq_defncrossingterm}), we immediately see that in general we have $\Wb=\Wb_1 \circ \Wb_y \circ \Wb_c.$ When the SNR is relatively large and the bandwidth is chosen properly, we show that with high probability $\max_{i,j}|\Wb_c(i,j)-1|=o(1)$, and Lemma \ref{lem_hardamardproductbound} leads to the claim.  }

The main problem we ask in this work is studying the relationship between the spectra of $\Wb_1$ and $\Wb$ under different SNR regions and bandwidths; that is, how do noise and chosen bandwidth impact the spectrum of the affinity matrix and how does the noisy spectrum deviate from that of the clean affinity matrix. By combining  existing understandings of $\Wb_1$ under suitable models, like manifolds \cite{MR4130541,dunson2019diffusion}, we obtain a finer understanding of the commonly applied GL.

\subsection{Relationship with the manifold model}\label{section:intro manifold model}

We claim that this seemingly trivial spiked covariance model \eqref{eq_modelsteptwo} overlaps with the commonly considered nonlinear manifold model.
Consider the case that the manifold is embedded into a subspace of fixed dimension in $\mathbb{R}^p$. With some mild assumptions about the manifold, and the fact that the Euclidean distance of $\mathbb{R}^p$ is invariant to rotation, the noisy manifold model can be studied by the spiked covariance model satisfying \eqref{eq_modelsteptwo}-\eqref{eq_lambdadefinition}. We defer details to Section \ref{sec_reducedproblem} to avoid distraction. 
We shall however emphasize that it does not mean that we could understand the manifold structure by studying the spiked covariance model. The problem we study in this paper is the relationship between the noisy and clean affinity matrices; that is, how does the noise impact the GL. The problem of exploring the manifold structure from a {\em clean} dataset via the GL is a different one. See, for example, \cite{gine2006empirical,MR4130541,dunson2019diffusion}. Therefore, by studying the relationship between the noisy GL and clean GL via studying the model where the data is sampled from a sum of two sub-Gaussian random variables, we could further explore the manifold structure from {\em noisy} dataset.

\subsection{Some related works} 
The focus of the current paper is to study the GL spectrum under the {nonnull case for the point cloud $\mathcal{X}=\{\xb_i\}_{i=1}^n$ described in (\ref{eq_modelsteptwo})}, and connect it to the GL spectrum under the null case; that is, when the GL is constructed from {the pure noise point cloud $\mathcal{Y}=\{\yb_i\}_{i=1}^n$}. The eigenvalues of $\Wb$, a high-dimensional Euclidean distance kernel random matrix, in the null case have been studied in several works. 
In the {pioneering} work \cite{elkaroui2010}, the author studied the spectrum of $\Wb$ under the null case assuming the bandwidth $h=p$, and  
showed that the complicated kernel matrix can be well approximated by a low rank perturbed Gram matrix; see \cite[Theorem 2.2]{elkaroui2010} or  \eqref{eq_kdtau} below for more details. 
It was concluded that studying $\Wb$ in the null case is closely related to studying the principal component analysis (PCA) of the dataset with a low rank perturbation. The results in \cite{elkaroui2010} were extended to more general kernels beyond the exponential kernel function, and $\mathcal{Y}$ can be anisotropic with some moment assumptions; for example, more general kernels with Gaussian noise was reported in \cite{cheng2013spectrum}, the Gaussian assumption was removed in \cite{do2013spectrum,MR3044473}, the convergence rates of individual eigenvalues of $\Wb$ were provided in \cite{DW1}, etc.

However, much less is known for the nonnull case (\ref{eq_modelsteptwo}). To our knowledge, the most relevant works are \cite{elkaroui20102,el2016graph}. By assuming $h=p$ and $\mathbb{E}\| \zb_i\|_2^2 \asymp p$, in \cite{elkaroui20102} the author showed that $n^{-1} \Wb$ could be well approximated by $n^{-1}\exp(-2 \upsilon) \Wb_1$, which connected the noisy observation and the clean signal. In fact, the results of \cite[Theorem 2.1]{elkaroui20102} were established for a class of smooth kernels and the noise could be anisotropic. In a recent work \cite{AR}, the authors established the concentration inequality of $\| \Wb-\mathbb{E} \Wb \|$ and used it to study spectral clustering.

We mention that other types of kernel random matrices related to (\ref{eq_defnw}) have also been studied in the literature. For example, the inner-product type kernel random matrices of the form $f(\xb_i^\top \xb_j),$ where $f$ is some kernel function, has drawn attention among researchers. In \cite{elkaroui2010}, the author showed that under some regularity assumptions, the inner-product kernel matrices could also be approximated by a Gram matrix with low rank perturbations, which had been generalized in \cite{cheng2013spectrum,do2013spectrum,MR3044473,Fan2018,9022455,2015arXiv150405880P}. 
The other example is the Euclidean random matrices of the form $F(\xb_i-\xb_j)$ arising from physics, where $F$ is some measurable function. Especially, the empirical spectral distribution has been extensively studied in, for example, \cite{MR1724455, MR2462254,MR3338323}, among others.     

Finally, our present work is also in the line of research regarding the robustness of GL. There have been efforts in different settings along this direction. For example, the perturbation of the eigenvectors of GL was studied in \cite{MEYER2014326}, the consistency and robustness of spectral clustering were analyzed in \cite{AR,MR2396807,10.1145/3097983.3098156}, and the robustness of DM was studied in \cite{Steinerburger2016,shen2020scalability}, among others.

\subsection{An overview of our results}

Our main contribution is a systematic treatment of the spectrum of GL constructed from a high-dimensional random point cloud corrupted by high-dimensional noise; that is, we consider the nonnull setup (\ref{eq_modelsteptwo}). We establish a connection between the spectra of noisy and clean affinity matrices with different choices of bandwidth and different SNRs by extensively expanding the results reported in \cite{elkaroui20102} with tools from random matrix theory. 
More specifically, we allow the signal strength to diverge with $n$ so that the relative strength of signal and noise is captured, and characterize the spectral distribution of the noisy affinity matrix by studying how the signal and noise interact and how different bandwidths impact the spectra of $\Wb$ and $\Ab.$ Motivated by our theoretical results, we propose an adaptive bandwidth selection algorithm with theoretical support. The proposed method utilizes some certain quantile of pairwise distance as in (\ref{eq_bandwithchoice}), where the quantile can be selected using our proposed Algorithm \ref{alg:choice}. We provide detailed results when $d=1$ in (\ref{eq_modelsteptwo}), and discuss how to extend the results to $d>1$ and why some cases are challenging when $d>1$.
Our result, when combined with existing manifold learning results like \cite{gine2006empirical,MR4130541,dunson2019diffusion} and others, pave the way to a better understanding of how GL-based algorithms behave in practice, and a theoretical support for the commonly applied bandwidth scheme, when $\{\zb_i\}_{i=1}^n$ is distributed on a low-dimensional, compact and smooth manifold.

We now provide a heuristic explanation of our results assuming $d=1$ and when $h=p$. In Section \ref{section main result}, when $\alpha<1,$ we show that the noisy kernel affinity matrix $\Wb$ provides limited useful information for the clean affinity matrix $\Wb_1$ in (\ref{eq_singlematrix}). Specifically, similar to the null setting, $\Wb$ can be well approximated by a Gram matrix with a finite rank perturbation; see Section \ref{sec_result_fixedepsilon} for more details.
When $\alpha \geq 1,$ $\Wb$ becomes closer to $\Wb_1$ with some universal scaling and isotropic shift (c.f. (\ref{eq_wba1})). Note that some related results have been established for $\alpha=1$ in \cite{elkaroui20102}. 
We show that the convergence rate is adaptive to $\alpha$, and provide {a quantification of how eigenvalues of the noisy affinity matrix $\Wb$ converge}.
We mention that {when $\alpha>1$,} the classic bandwidth choice $h=p$ needs to be modified to reflect the underlying signal structure. Specifically, {if $h=p$ and $\alpha>1$, we have} $\Wb \approx \mathbf{I}$. These results are stated in Theorem \ref{thm_informativeregion}.  Similar results hold for the transition matrix $\Ab$, and they are reported in Section \ref{sec_transitionmatrixab}.

Motivated by the fact that $\Wb$ becomes trivial when $\alpha>1$ is large, in Section \ref{section_result_2ndchoice}, we focus on the case $h=p+\lambda.$ When $\alpha<1,$ the results are analogous to the setting $h=p.$ When $\alpha \geq 1,$ the spectrum of $\Wb$ will be dramatically different. {Specifically, when $\alpha=1$, $\Wb$ is close to $\Wb_1$ with some universal scaling and isotropic shift (c.f. (\ref{eq_wa2})), and when $\alpha>1,$ $\Wb$ is close to $\Wb_1$ without any scaling and shift.} 
Moreover, besides the top $\log(n)$ eigenvalues of $\Wb$, the other eigenvalues are trivial; see Theorem \ref{thm_adaptivechoiceofc} for more details.

In practice, $\lambda$ is unknown and not easy to estimate especially when $\{\zb_i\}$ are sampled from a nonlinear geometric object. {For practical implementation, we propose a bandwidth selection algorithm in Section \ref{sec_bandwidthselectionalgo}. 
It turns out that the proposed bandwidth selection algorithm can bypass the challenge of estimating $\lambda$, and the result coincides with that determined by the ad hoc bandwidth selection method commonly applied by practioners. See \cite{shnitzer2019recovering,lin2021wave} for example, among many others. Note that the bandwidth issue is also discussed in \cite{AR}. To our knowledge, our result is the first step toward a theoretical understanding of that common practice.}

Finally, we point out technical ingredients of our proof. We focus on the discussion of the bandwidth $h=p$ and take $d=1$ for an example.  
First, when $\lambda$ is bounded, i.e., $\alpha=0,$ the spectrum of the $\Wb$ has been studied in \cite{elkaroui2010} using an entrywise expansion to the order of two (with a third order error). When $0<\alpha<0.5,$ the above strategy can be adapted straightforwardly with some modifications. When $0.5 \leq \alpha<1,$ we need to expand the entries to a higher order depending on $\alpha$. With this expansion, we show that except for the Gram matrix, the other parts are either fixed rank or negligible; see Sections \ref{sec_sub_subsuper} and \ref{sec_sub_slowly} for more details. Second, when $\alpha \geq 1,$ the entrywise expansion strategy fails and we need to conduct a more careful analysis. Particularly, thanks to the Hadamard product representation of $\Wb$ in (\ref{eq_separationsinglenoise}), we need to analyze the spectral norm of $\Wb_1$ carefully; see Lemma \ref{Lemma: W1 bound}, which has its own interest. Third, to explore the number of informative eigenfunctions, we need to investigate the individual eigenvalues of $\Wb_1$ using Mehler's formula (c.f. (\ref{eq_melherformula})); see Section  \ref{sec_sub_fast} for more details.

The paper is organized in the following way. The main results are described in Section \ref{section main result} for the bandwidth $h \asymp p$ and in Section \ref{section_result_2ndchoice} for the bandwidth $h \asymp (\lambda+p)$ and an adaptive bandwidth selection algorithm. The numerical studies are reported in Section \ref{section numerical studies}. The paper ends with the discussion and conclusion in Section \ref{sec: discussion and conclusion}. The background and necessary results for the proof are listed in Section \ref{section preliminary results} and the proofs of the main results are given in Sections \ref{sec_proofs2} and \ref{sec_proofs3}. 

\vspace{3pt}

\noindent {\bf Conventions.} We systematically use the following notations. For a smooth function $f(x)$,  denote $f^{(k)}(x), k=0,1,2,\cdots,$ to be the $k$-th  derivative of $f(x).$ For two sequences $a_n$ and $b_n$ depending on $n,$ the notation $a_n=O(b_n)$ means that $|a_n| \leq C|b_n|$ for some constant $C>0,$ and $a_n=o(b_n)$ means that $|a_n| \leq c_n |b_n|$ for some positive sequence $c_n \downarrow 0$ as $n \rightarrow \infty.$ We also use the notation $a_n \asymp b_n$ if $a_n=O(b_n)$ and $b_n=O(a_n).$

\section{Main results (I): classic bandwidth choice $h \asymp p$}\label{section main result}

In this section, we state the main results regarding the spectra of $\Wb$ and $\Ab$ when the bandwidth satisfies $h \asymp p.$ For definiteness, without loss of generality, we assume that $h=p.$  Such a choice has appeared in many works in kernel methods and manifold learning, e.g., \cite{elkaroui2010,DW1,do2013spectrum,Liao,el2015graph,el2016graph}. 
Since our focus is the nonlinear interaction of the signal and noise in the kernel method, in what follows, we focus on reporting the results for $d=1$ and omit the subscripts in (\ref{eq_lambdadefinition}).  We refer the readers to Remark  \ref{rmk_multipled} {and Section \ref{sec_generalizationrevision2}} below for a discussion on the setting when $d>1.$

\subsection{Some definitions}
We start with introducing the {notion of} {\em stochastic domination} \cite{MR3119922}. It makes precise statements of the form ``$\mathsf{X}^{(n)}$ is bounded with high probability by $\mathsf{Y}^{(n)}$ up to small powers of $n$''. 

\begin{definition} [Stochastic domination]\label{defn_stochasdomi} Let
 \begin{align*}
 \mathsf{X}&\,=\big\{\mathsf{X}^{(n)}(u):  n \in \mathbb{N}, \ u \in \mathsf{U}^{(n)}\big\}, \ \\  \mathsf{Y}&\,=\big\{\mathsf{Y}^{(n)}(u):  n \in \mathbb{N}, \ u \in \mathsf{U}^{(n)}\big\},
 \end{align*}
be two families of nonnegative random variables, where $\mathsf{U}^{(n)}$ is a possibly $n$-dependent parameter set. We say that $\mathsf{X}$ is {\em stochastically dominated} by $\mathsf{Y}$, uniformly in the parameter $u$, if for all small $\upsilon>0$ and large $ D>0$, there exists $n_0(\upsilon, D)\in \mathbb{N}$ so that we have 
\begin{equation*}
\sup_{u \in \mathsf{U}^{(n)}} \mathbb{P} \Big( \mathsf{X}^{(n)}(u)>n^{\upsilon}\mathsf{Y}^{(n)}(u) \Big) \leq n^{- D},
\end{equation*}   
for a sufficiently large $n \geq  n_0(\upsilon, D)$. We interchangeably use the notation $\mathsf{X}=O_{\prec}(\mathsf{Y})$ {or $\mathsf{X} \prec \mathsf{Y}$}  if $\mathsf{X}$ is stochastically dominated by $\mathsf{Y}$, uniformly in $u$, when there is no danger of confusion.
In addition, we say that an $n$-dependent event $\Omega \equiv \Omega(n)$ holds {\em with high probability} if for a $D>1$, there exists $n_0=n_0(D)>0$ so that
\begin{equation*}
\mathbb{P}(\Omega) \geq 1-n^{-D},
\end{equation*}
when $n \geq n_0.$ 

\end{definition}

For any constant $c \in \mathbb{R},$ denote $T_c$ to be the shifting operator that shifts a probability measure $\mu$ defined on $\mathbb{R}$ by $c$; that is
\begin{equation}\label{eq_shiftoperator}
T_c \mu(I):=\mu(I-c),
\end{equation} 
where $I \subset \mathbb{R}$ is a measurable subset. Till the end of the paper, for any symmetric $n \times n$ matrix $\Bb,$ the eigenvalues are ordered in the decreasing fashion so that $\lambda_1(\Bb) \geq \lambda_2(\Bb) \geq \cdots \geq \lambda_n(\Bb).$

\begin{definition}\label{Definition typical locations}
For a given probability measure $\mu$ and $n\in \mathbb{N}$, define the $j$-th \emph{typical location} of  $\mu$ as $\gamma_\mu(j)$; that is,  
\begin{equation}\label{eq_typical}
\int_{\gamma_\mu(j)}^{\infty} \mu(\mathrm{d} x)=\frac{j}{n}\,,
\end{equation}
where $j=1,\ldots,n$. $\gamma_\mu(j)$ is also called the $j/n$-quantile of $\mu$.
\end{definition}

Let $\Yb \in \mathbb{R}^{p \times n}$ be the data matrix associated with the point cloud $\mathcal{Y}$; that is, the $i$-th column $\Yb$ is $\yb_i$. Denote the empirical spectral distribution (ESD) of $\Qb=\frac{\sigma^2}{p}\Yb^\top \Yb$ as 
\begin{equation*}
\mu_{\Qb}(x)=\frac{1}{n} \sum_{i=1}^n \mathbf{1}_{\{\lambda_i(\Qb)\leq x\}}, \ x \in \mathbb{R}. 
\end{equation*}
Here, $\sigma>0$ stands for the standard deviation of the scaled noise $\sigma \yb_i$. It is well-known that in the high-dimensional region (\ref{eq_ratio}), $\mu_{\Qb}$ has the same asymptotic properties \cite{Knowles2017} as the so-called Marchenko-Pastur (MP) law \cite{MP}, denoted as $\mu_{c_n,\sigma^2}$, satisfying
\begin{equation}\label{eq_mp}
\mu_{c_n,\sigma^2}(I)=
(1-c_n)_{+} \chi_{I}(0)+\zeta_{c_n,\sigma^2}(I)\,,
\end{equation}
where  $\chi_I$ is the indicator function and $(a)_+:=0$ when $a\leq 0$ and $(a)_+:=a$ when $a>0$, 
\begin{equation}
\mathrm{d} \zeta_{c_n,\sigma^2}(x)=\frac{1}{2\pi \sigma^2} \frac{\sqrt{(\lambda_+-x)(x-\lambda_-)}}{c_nx} \mathrm{d} x\,, 
\end{equation}
$\lambda_{+}=(1+ \sigma^2 \sqrt{c_n})^2$ and $\lambda_{-}=(1- \sigma^2 \sqrt{c_n})^2$.

Denote 
\begin{equation}\label{eq_defntau}
\tau \equiv \tau (\lambda):=2 \left( \frac{\lambda}{p}+1 \right),
\end{equation}
and for any kernel function $f(x)$, define
\begin{equation}\label{eq_varsigmalambda}
\varsigma \equiv \varsigma(\lambda): =f(0)+2f'(\tau)-f(\tau)\,. 
\end{equation}
Till the end of this paper, we will use $\tau$ and $\varsigma$ for simplicity, unless we need to specify the values of $\tau$ and $\varsigma$ at certain points of $\lambda$. As mentioned below \eqref{eq_defnw}, we focus on $f(x)=\exp(-\upsilon x)$ throughout the paper unless otherwise specified. Recall the shift operator defined in (\ref{eq_shiftoperator}). Denote 
\begin{equation}\label{eq_nudefinition}
\nu_\lambda:=T_{\varsigma(\lambda)} \mu_{c_n,-2f'(\tau(\lambda))}, 
\end{equation}
where $\mu_{c_n,-2f'(\tau(\lambda))}$ is the MP law defined in (\ref{eq_mp}) by replacing $\sigma^2$ with $-2f'(\tau(\lambda))$.

\subsection{Spectrum of kernel affinity matrices: low signal-to-noise region $0 \leq \alpha <1$}\label{sec_result_fixedepsilon}
In this subsection, we present the results {when the SNR is small in the sense that $0 \leq  \alpha<1$ in Theorems \ref{thm_affinity matrix} and \ref{lem_affinity_slowly}. In this setting, there does not exist a natural connection between the spectrum of $\Wb$ and the signal part $\Wb_1$ in (\ref{eq_singlematrix}). More specifically, even though the spectrum of $\Wb$ can be studied, $\| \Wb-\Wb_1 \|$ is not close to zero. }

In the first result, we consider the case when $\lambda$ is bounded from above; that is, $\alpha=0$. 
In this bounded region, as in the null case studied in \cite{elkaroui2010, DW1} (see Lemma \ref{lem_karoui} for more details),  the spectrum is {governed by} the MP law, except for a few outlying eigenvalues. These eigenvalues either come from the kernel function expansion that will be detailed in the proof, or the Gram matrix {$\mathbf{Q}_x=\frac{1}{p}\Xb^\top \Xb,$ where $\Xb \in \mathbb{R}^{p \times n}$ is  the data matrix associated with the noisy observation $\xb_i$ defined in (\ref{eq_modelsteptwo})} when the signal is above some threshold. This result is not surprising, since the signal is asymptotically negligible compared with the noise.

\begin{theorem}[Bounded region]\label{thm_affinity matrix} Suppose \eqref{eq_defnw} and \eqref{eq_modelsteptwo}-\eqref{eq_lambdadefinition} hold true, $d=1$ and $h=p$. Moreover, we assume that $\lambda$ is a fixed constant. Set
\begin{equation*}
\mathsf{S}:=
\begin{cases}
3, & \lambda \leq \sqrt{c_n};  \\
4, & \lambda>\sqrt{c_n}. 
\end{cases}
\end{equation*}
For any given small $\epsilon>0,$ when $n$ is sufficiently large, for some constant $C>0,$ with probability at least $1-O(n^{-1/2})$, we have
\begin{equation}\label{eq_formrigidity}
\left| \lambda_{i}(\Wb)-\gamma_{\nu_0}(i) \right| \leq Cn^{-1/4}, \  \mathsf{S} < i \leq  (1-\epsilon) n, 
\end{equation}
where $\nu_0$ is defined in (\ref{eq_nudefinition}). 
\end{theorem}

\begin{remark}\label{rmk_generalkernel}
In Theorem \ref{thm_affinity matrix}, we focus on reporting the bulk eigenvalues of $\Wb.$ In this case, the outlying eigenvalues are mainly from the kernel function expansion, which we call the ``kernel effect'' hereafter, and {the resulting} Gram matrix. For example, as we will see in the proofs in Section \ref{sec_proofs2},  we have that $\lambda_1(\Wb)=n\exp(-\tau \upsilon)+o_{\prec}(1)$ and $\lambda_2(\Wb)=\| \mathrm{Sh}_1(\tau)+\mathrm{Sh}_2(\tau) \|+o_{\prec}(1),$ where $\mathrm{Sh}_1(\tau)$ and $\mathrm{Sh}_2(\tau)$ are defined in (\ref{eq_sho}). 
Moreover, we mention that Theorem \ref{thm_affinity matrix} holds for a more general kernel function as in \cite{elkaroui2010, DW1}; that is, \eqref{eq_defnw} is replaced by
\begin{equation}\label{eq_defnwQ}
\Wb(i,j)=f\left(\frac{\|\xb_i-\xb_j \|^2}{h}\right), \ 1 \leq i, j \leq n, 
\end{equation}  
where $f \in C^3(\mathbb{R})$ is monotonically decreasing, bounded and $f(2)>0$. Moreover, we remark that in (\ref{eq_formrigidity}) we can relax $\epsilon>0$ to $\epsilon=0$ with an additional assumption that $|c_n-1| \geq \tau$ for some constant $\tau>0,$ which is a standard assumption in the random matrix theory literature guaranteeing that the smallest eigenvalue of the Gram matrix is bounded from below; for instance, see the monograph \cite{erdos2017dynamical}. Finally, we mention that in the pure noise setting, i.e., $\lambda=0,$ the asymptotics of the spectrum of $\Wb$ (equivalently, $\Wb_y$ in this setting) has been established in \cite{DW1} using a different approach, where the results hold for $\epsilon=0$ regardless of the ratio of $n/p.$ However,  in \cite{DW1}, the bound is weaker than what we show here; that is, the rate is $n^{-1/9}$ when $i$ is bounded, and the rate $n^{-1/4}$ only appears when $i \asymp n.$ 
\end{remark}

In the second result, we study the case when the signal strength $\lambda$ diverges with $n$ but slowly; that is, $\lambda=\lambda(n)  \asymp n^\alpha$, where $0<\alpha < 1.$   In this region, since $\alpha<1$, the signal is still weaker than the noise, and again it is asymptotically negligible. Thus, it is not surprising to see that the noise information dominates and the spectral distribution is almost like the MP law, except for the first few eigenvalues. Again, like in the bounded region, these finite number of outlying eigenvalues come from the interaction of the nonlinear kernel and the Gram matrix.


\begin{theorem}[Slowly divergent region]\label{lem_affinity_slowly}
Suppose \eqref{eq_defnw} and \eqref{eq_modelsteptwo}-\eqref{eq_lambdadefinition} hold true, $d=1$ and $h=p$. For any given small $\epsilon>0,$ when $n$ is sufficiently large,  we have that {the following holds} with probability at least $1-O(n^{-1/2})$: 
\begin{enumerate}
\item[(1)] When $0<\alpha < {0.5-\epsilon},$ for some constant $C>0,$ we have that:  
\begin{align}\label{eq_formrigidity22}
& | \lambda_{i}(\Wb)  -\gamma_{\nu_0}(i) | \leq C  \max\left\{n^{-1/4}, n^{\epsilon}\frac{\lambda}{\sqrt{n}} \right\},  \ 4<i \leq (1-\epsilon) n.  
\end{align}

\item[(2)] When ${0.5-\epsilon} \leq \alpha <1,$ denote
\begin{equation}\label{eq_defnd}
\fd \equiv \fd(\alpha):=\left\lceil \frac{1}{1-\alpha} \right\rceil+1. 
\end{equation} 
For some constant $C>0,$ there exists some integer $\mathsf{K}$ satisfying
\begin{equation*}
4 \leq \mathsf{K} \leq C4^\fd,
\end{equation*} 
so that with high probability, for all $\mathsf{K} < i \leq (1-\epsilon) n,$ we have that 
\begin{align}\label{eq_formrigidity2}
| \lambda_{i}(\Wb)-&\gamma_{\nu_0}(i) | \leq  
C\max\left\{p^{\mathcal{B}(\alpha) }, \frac{\lambda}{p} \right \},
\end{align}
where $\mathcal{B}(\alpha)<0$ is defined as 
\begin{equation}\label{eq_defnmathcalbalpha}
\mathcal{B}(\alpha)={(\alpha-1)\left\lceil \frac{1}{1-\alpha} \right\rceil +\alpha}\,.
\end{equation}
\end{enumerate}
\end{theorem}

As in Theorem \ref{thm_affinity matrix}, since the outlying eigenvalues are impacted by both the kernel effect and signals, we focus on reporting the bulk eigenvalues. Similar to the discussion in Remark \ref{rmk_generalkernel}, the outlying eigenvalues can be figured out from the proof in Section \ref{sec_proofs2}. Moreover, the number of the outlying eigenvalues is adaptive to $\alpha.$ As can be seen from Theorem \ref{lem_affinity_slowly}, {for any small $\epsilon>0$, when $0<\alpha<0.5-\epsilon,$} the results are similar to (2) of Theorem \ref{thm_affinity matrix} except for the convergence rates. When {$\alpha \geq 0.5-\epsilon, $} there will be more, but finite, outlying eigenvalues, which comes from a high order expansion in the proof. Finally, we mention that Theorem \ref{lem_affinity_slowly} holds for a more general kernel function like that in \eqref{eq_defnwQ}.

We remark that when  $\alpha<1$ and $\lambda>\sqrt{c_n},$ the non-bulk eigenvalues (i.e., outlying eigenvalues) ``seem'' to be useful for understanding the signal. 
For example, we may potentially utilize the number of outliers to determine if the signal strength is stronger than $\sqrt{c_n}$. 
However, to our best knowledge, there exists limited literature on utilizing this information via GL since $\mathbf{W}$ ($\mathbf{A}$ respectively) is not close to $\mathbf{W}_1$ ($\mathbf{A}_1$ respectively) in this region, except some relevant work in \cite{el2015graph} when $1/2< \alpha<1$. 
Recall that some of the non-bulk eigenvalues are generated by the kernel effect instead of the signals. As commented in Remark \ref{rmk_generalkernel}, while it is true that when $\lambda$ is bounded the kernel effect can be quantified, when $\lambda$ diverges, we loss the capability to distinguish these outliers. 
Especially, when $1/2 \leq \alpha<1,$ although we show in (2) of Theorem \ref{lem_affinity_slowly} that the number of non-bulk eigenvalues is finite and depends on the signal strength, we do not know the exact number of outliers and the relation between the signals and outliers. Therefore, it is challenging to recover the signals using this information. For example, when the signal is supported on a linear subspace so that there are multiple spikes, to our knowledge it is challenging to estimate the dimension via GL.

We shall emphasize that the GL approach is very different from PCA. Note that for the $d$-spiked model, only the $d$ largest eigenvalues could possibly detach from the bulk and can be located when PCA is applied. However, for the GL approach, even a single spike can lead to multiple outlying eigenvalues as in (2) of Theorem \ref{lem_affinity_slowly}. This shows that even for a simple $1$-dim linear manifold that is realized as a $1$-dim {\em linear} subspace in $\mathbb{R}^p$, the nonlinear method via GL is very different from the linear method via PCA. More details on this aspect can be found in Section \ref{sec: discussion and conclusion}. 
The problem becomes more challenging if the subspace is nonlinear, even under the assumption that the manifold model can be reduced to this low rank spiked model in Section \ref{sec_reducedproblem} that we focus on in this paper.

Finally, we point out that when $1/2 < \alpha <1,$ although it is challenging to obtain precise information for its clean signal counterpart from a direct application of the standard GL via $\Wb$ (c.f. Theorem \ref{lem_affinity_slowly}) or $\Ab$ (c.f. Corollary \ref{thm_normailizedaffinitymatrix}), we could consider a variant of GL via the transition matrix $\Ab$ by zeroing-out the diagonal elements of $\Wb$ proposed in \cite{el2016graph}. It has been shown in \cite{el2016graph} that the zeroing-out strategy could help the analysis of noisy datasets. We refer the readers to Appendix \ref{sec_generalizationrevision1} for more details.

\begin{remark}
When $\alpha<1,$ although it is still unclear how to use the bulk of eigenvalues from the noisy observation to extract information for the clean signal counterpart, it is a starting point for understanding the existence of a strong signal (i.e., $\alpha \geq 1$). For example, when $\alpha<1,$ Theorems \ref{thm_affinity matrix} and \ref{lem_affinity_slowly} demonstrate that most of the eigenvalues follow the MP law so that except for a finite number of outliers, two consecutive eigenvalues should be close to each other by a distance of $o(1)$. Under the alternative (i.e., $\alpha \geq 1$), since the noisy GL is close to the clean GL as shown later in Section \ref{sec_informativeregion} below, two consecutive eigenvalues can be separated by a distance of constant order.
\end{remark}

\subsection{Spectrum of kernel affinity matrices: high signal-to-noise region $\alpha \geq 1$}\label{sec_informativeregion} In this subsection, we present the results such that the spectra of $\Wb$ and $\Wb_1$ in (\ref{eq_singlematrix}) can be connected after properly scaling when the SNR is relatively large, i.e., $\alpha \geq 1$. 
We first prepare some notations. Denote 
\begin{align}\label{eq_wba1}
\Wb_{a_1}:=\exp(-2\upsilon) \Wb_1+(1-\exp(-2\upsilon))\mathbf{I}_n.
\end{align}
Clearly, $\Wb_{a_1}$ is closely related to $\Wb_1$ via  a scaling and an isotropic shift and $\Wb_1$ contains only the signal information. 
On the other hand, note that $\Wb_{a_1}=[\exp(-2\upsilon) \boldsymbol{1}\boldsymbol{1}^\top+(1-\exp(-2\upsilon))\mathbf{I}_n]\circ \Wb_1$, where the matrix $\exp(-2\upsilon) \boldsymbol{1}\boldsymbol{1}^\top+(1-\exp(-2\upsilon))\mathbf{I}_n$ comes from the first order Taylor expansion of $\Wb_y$. We introduce another affinity matrix that will be used when $\alpha$ is too large so that the bandwidth $h \asymp p$ is relatively small compared with the signal strength. Denote $\Wb_{c} \in \mathbb{R}^{n \times n}$
\begin{equation}\label{eq_defncrossingterm}
\Wb_{c}(i,j)=\exp\left(-2\upsilon \frac{(\zb_i-\zb_j)^\top(\yb_i-\yb_j) }{h} \right),
\end{equation}
and 
\begin{equation}\label{eq_defntildewa1}
\widetilde{\Wb}_{a_1}:=\Wb_{a_1} \circ \Wb_c.
\end{equation}
Note that $\widetilde{\Wb}_{a_1}$ differs from $\Wb_{a_1}$ by the matrix $\Wb_c.$ It will be used when $\alpha \geq 2.$

Our main result for this SNR region is Theorem \ref{thm_informativeregion} below. For some constant $C \equiv C(\alpha)>0,$ denote  
\begin{equation}\label{eq_defntalpha}
\mathsf{T}_{\alpha}:=
\begin{cases}
C \log n, & \alpha=1; \\
Cn^{\alpha-1}, & 1<\alpha < 2. 
\end{cases}
\end{equation}

\begin{theorem}\label{thm_informativeregion}
Suppose \eqref{eq_defnw} and \eqref{eq_modelsteptwo}-\eqref{eq_lambdadefinition} hold true, $d=1$ and $h=p$.
\begin{enumerate}
\item When $1 \leq \alpha < 2,$ we have that 
\begin{equation}\label{eq_informativeequationone}
\left\| \frac{1}{n} \Wb-\frac{1}{n} \Wb_{a_1} \right\| \prec n^{-1/2}. 
\end{equation}
Moreover, for some universal large constant $D>2,$ we have that for $i \geq \mathsf{T}_{\alpha}$ in (\ref{eq_defntalpha})
\begin{equation}\label{eq_informativeequationtwo}
\left|\lambda_i(\Wb_{a_1})-(1-\exp(-2\upsilon)) \right|\prec n^{-D}. 
\end{equation}
\item When $\alpha \geq 2,$ we have that 
\begin{equation}\label{eq_informativeequationoneone}
\left\| \frac{1}{n} \Wb-\frac{1}{n} \widetilde{\Wb}_{a_1} \right\| \prec n^{-\alpha/2}+n^{-3/2}. 
\end{equation}
Furthermore, when $\alpha$ is larger and satisfies
\begin{equation}\label{eq_conditionexponential}
\alpha>\frac{2}{t}+1
\end{equation}  
for a constant $t \in (0,1)$, 
we have that with probability $1-O(n^{1-t(\alpha-1)/2}),$  for some constant $C>0$, 
\begin{equation}\label{equationreduced_11}
\left\|\widetilde{\Wb}_{a_1}-\mathbf{I}_n \right\| \leq C n \exp(-\upsilon (\lambda/p)^{1-t} ),
\end{equation}
and consequently 
\begin{equation}\label{eq_c1exp}
\left\| \Wb-\mathbf{I}_n \right\| \leq n \exp(-\upsilon (\lambda/p)^{1-t} ).   
\end{equation}  
\end{enumerate}
\end{theorem}

The scaling $n^{-1}$ in (\ref{eq_informativeequationone}) is commonly used in many manifold learning and machine learning literature \cite{AR,JMLR:v7:braun06a,elkaroui20102,koltchinskii2000random,MR2600634,MR2396807}. 
On one hand, (1) of Theorem \ref{thm_informativeregion} shows that once the SNR is ``relatively large'' ($1 \leq \alpha<2$), we may access the spectrum of the clean affinity matrix $\Wb_1$ via the noisy affinity matrix $\Wb$ as is described in (\ref{eq_informativeequationone}), and the clean affinity matrix may contain useful information of the signal. In this case, the signal is strong enough to compete with the noise so that we are able to recover the ``top few'' eigenvalues of the kernel matrix associated with the clean data via $\Wb_{a_1}.$ Especially, (\ref{eq_informativeequationtwo}) implies that we should focus on the top $\mathsf{T}_{\alpha}$ eigenvalues of $\Wb$, since the remaining eigenvalues are not informative. This coincides with what practitioners usually do in data analysis.
Note that $\mathsf{T}_{\alpha}$ increases with $\alpha$, which fits our intuition, since the SNR becomes larger.

On the other hand, we find that the classic bandwidth choice $h \asymp p $ is not a good choice when the SNR is ``too large'' ($\alpha\geq 2$). First, (2) of Theorem \ref{thm_informativeregion} states that when $\alpha \geq 2,$ since the bandwidth is too small compared with the signal strength, the noisy affinity matrix will be close to a mixture of signal and noise. Especially, when the signals are stronger in the sense of (\ref{eq_conditionexponential}), we will not be able to obtain information from the noisy affinity matrix. This can be understood as follows.  Since the signal is far stronger than the noise, equivalently we could say that the signal is contaminated by ``small'' noise. However, since the bandwidth is set to $h=p$, which is ``small'' compared with the signal strength, the exponential decay of the kernel forces each point to be ``blind'' to see other points. In this case, $\Wb_{i,j}$ is close to $0$ when $i\neq j$, and the affinity matrix $\Wb$ is close to an {identity} matrix, which leads to 
\eqref{eq_c1exp}.
As we will see in Theorem \ref{sec_mainresultssecondcase}, all these issues will be addressed using a different bandwidth. For the readers' convenience, in Figure \ref{Illustration of phase transition}, we use a simulation to summarize the phase transitions  observed in Theorems \ref{thm_affinity matrix}, \ref{lem_affinity_slowly} and \ref{thm_informativeregion}. For numerical accuracy of our established theorems, we refer the readers to Section \ref{sec_simu_accuracy} below for more details.

\begin{figure*}[!ht]
\centering
\includegraphics[trim=70 0 80 0, clip, width=\textwidth]{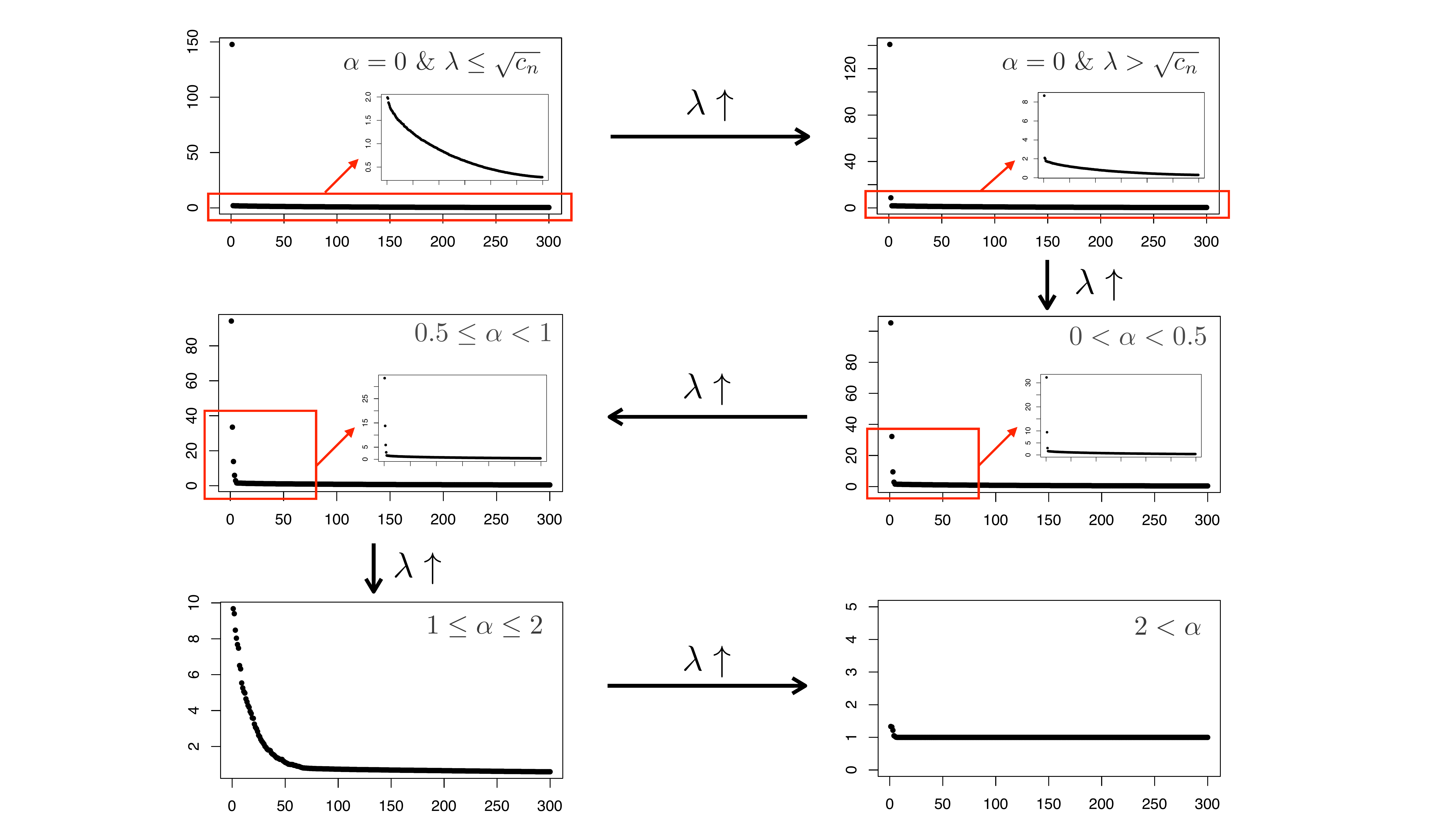}
\caption{An illustration of the phase transition phenomenon of the affinity matrix spectrum. The noise $\yb_i$, $i=1,\ldots,n$, is i.i.d. sampled from $\mathcal{N}(0,I_p)$ and the clean data $\xb_i$ is i.i.d. sampled from $\mathcal{N}(0,\lambda)$, where $\lambda>0$. We take $n=p=300$. The kernel is $f(x)=\exp(-x/2)$ and the bandwidth is $h=p$. For each $\lambda$, the 300 eigenvalues are plotted in the descending order. In the bounded and slowly divergent regions, the second to the 300-th eigenvalues are zoomed in for a better visualization.
\label{Illustration of phase transition}}
\end{figure*}

Technically, in the previous results when $\alpha<1$, the kernel function $f(x)$ only contributes to the measure $\nu_0$ via \eqref{eq_nudefinition} and its decay rate does not play a role in the conclusions.  However, once the signal becomes stronger, the kernel decay rate plays an essential role. We focus on $f(x)=\exp(-\upsilon x)$ to simplify the discussion in this paper, and postpone the discussion of general kernels to our future work.

\begin{remark}\label{rmk_partone}
When $1 \leq \alpha <2$ in Theorem \ref{thm_informativeregion}, we have shown that besides the top $\mathsf{T}_{\alpha}$ eigenvalues of $\Wb_{a_1}$, the remaining eigenvalues of $\Wb_{a_1}$ are trivial. Since the error bound in (\ref{eq_informativeequationtwo}) is much smaller than the one in (\ref{eq_informativeequationone}), the smaller eigenvalues of $\Wb$ may fluctuate more and have a non-trivial ESD. The ESD of $\Wb$ is best formulated using the Stieltjes transform \cite{MR2567175}. Here we provide a remark on the approximation of the Stieltjes transforms. Denote $\Wb_{b_1}$ as 
\begin{align}\label{eq_defnwxa2}
\Wb_{b_1}:=\Big( & \frac{2\upsilon\exp(-2\upsilon)}{p} \Yb^\top \Yb +2 \upsilon \exp(-4 \upsilon)\mathbf{I}_n \Big) \circ \Wb_1\,.
\end{align}
Note that compared with $\Wb_{a_1}$, the matrix $\frac{2\upsilon\exp(-2\upsilon)}{p} \Yb^\top \Yb+2 \upsilon \exp(-4 \upsilon)\mathbf{I}_n$ in $\Wb_{b_1}$ comes from a higher order Taylor expansion of $\Wb_y$.
Let $m_{\Wb}(z)$ and $m_{\Wb_{b_1}}(z)$ be the Stieltjes transforms of $\Wb$ and $\Wb_{b_1}$ respectively. In Section \ref{sec_rmkjust}, we show that
\begin{equation}\label{eq_stieltjestransformlimit}
\sup_{z \in \mathcal{D}}|m_{\Wb}(z)-m_{\Wb_{b_1}}(z)| \prec \frac{1}{n^{1-\alpha/2} \eta^2}\,,
\end{equation}
where $\mathcal{D}$ is the domain of spectral parameters defined as 
\begin{align}\label{defn_domaind}
\mathcal{D} := \mathcal{D}(1/4,\mathsf{a}):=\Big\{ z=& E+\ri \eta: \mathsf{a} \leq E \leq \frac{1}{\mathsf{a}}, \nonumber \\
 & n^{-1/2+\alpha/4+\mathsf{a}} \leq \eta \leq \frac{1}{\mathsf{a}} \Big\},
\end{align}
and $0<\mathsf{a}<1$ is some fixed (small) constant. This result helps us further peek into the intricate relationship between the clean and noisy affinity matrices. However, it does not provide information about each single eigenvalue.

\end{remark}

\begin{remark}\label{rmk_multipled}
In the above theorems, we focus on reporting the results for the case $d=1$ in (\ref{eq_modelsteptwo}). We now discuss how to generalize our results to $d>1$. There are two major cases. The first case is when all signal strengths are in the same SNR region, and the second case is when the signal strengths might be in different SNR regions.   
We start from the first case, and there are four regions we shall discuss.

\begin{enumerate}

\item When all $\alpha_i$, $1 \leq i \leq d$, are very large in the sense that they satisfy the condition (2) of Theorem \ref{thm_informativeregion}, following the same argument, we can immediately conclude that the results stated in (2) in Theorem \ref{thm_informativeregion} still hold true by setting $\alpha:=\max_i \alpha_i$ in (\ref{eq_conditionexponential}).

\item When $\alpha_i >1$, $1 \leq i \leq d$, satisfy condition (1) of Theorem \ref{thm_informativeregion}, the results of (\ref{eq_informativeequationone}) and (\ref{eq_informativeequationoneone}) of Theorem \ref{thm_informativeregion} hold with some changes. Indeed, $n^{-\alpha/2}$ in (\ref{eq_informativeequationoneone}) should be replaced by $\sum_{l=1}^d n^{-\alpha_l/2}$ and $Cn^{\alpha-1}$ in (\ref{eq_defntalpha}) should be replaced by $ C n^{\min_l \alpha_l-1}$. Moreover, in this setup, since the noisy affinity matrix will be close to a matrix depending on the clean affinity matrix $\Wb_{1}$, where $\Wb_1(i,j)=\exp\left(-\frac{\| \zb_i-\zb_j \|_2^2}{h} \right)$, which in general does not follow the MP law, the spectrum of $\Wb$ vary according to the specific values of $\lambda_i$, $1 \leq i \leq d$. More discussion with simulation of this setup with $d=2$ can be found in Section \ref{sec_generalizationrevision2}. 

\item When $\alpha_i = 1$, $i=1,\ldots,d$, this is the region our argument cannot be directly applied and we need a substantial generalization of the proof. Especially, our proof essentially relies on the Mehler's formula in Section \ref{sec_mehler} that has only been proved for $d=1$ to our knowledge. Nevertheless, we believe it is possible to generalize this formula to $d>1$ following the arguments of \cite{FOATA1978367}.  

\item When $0 \leq \alpha_i<1$, $1 \leq i \leq d,$ we could directly generalize the result to $d>1$ with all the key ingredients,  i.e., Lemmas \ref{lem_concentrationinequality}--\ref{lem_karoui}, in Section \ref{sec_affinitymatrixpreliminaryresults}. Specifically, to extend Theorem \ref{thm_affinity matrix} concerning $\alpha_i=0$, $1 \leq i \leq d$, denote $r=0$ if $\lambda_i \leq \sqrt{c_n}$ for all $1 \leq i \leq d$, $r=d$ if $\lambda_i > \sqrt{c_n}$ for all $1 \leq i \leq d$, or $1\leq r<d$ if $\lambda_1 \geq \cdots \geq  \lambda_{r}>\sqrt{c_n} \geq \lambda_{r+1} \geq \cdots \lambda_d$. The proof of Theorem \ref{thm_affinity matrix} still holds by updating $\mathsf{S}=3+r.$ In fact, according  to (\ref{eq_spikedpartofmatrix}), the $\mathsf{O}$ matrix therein is of rank three and the Gram matrix $\Xb^\top \Xb$ can generate $r$ outliers according to  Lemma \ref{lem_gramsummary}. 
Theorem \ref{lem_affinity_slowly} still holds when $d>1$. Following its proof,  when $0<\alpha_i<0.5-\epsilon$ for all $1 \leq i \leq d,$ the first part of the theorem still holds when $i>d+r$ and the rate $\lambda/\sqrt{n}$ should be replaced by $\sum_{l=1}^d \lambda_l /\sqrt{n}.$ In fact, in this setting, all $\lambda_i>\sqrt{c_n}$  and all will generate outliers. Moreover, when all $\alpha_i \geq 0.5-\epsilon,$ by replacing (\ref{eq_defnd}) with $\sum_{l=1}^d \fd_l$, where $\fd_l=\left\lceil \frac{1}{1-\alpha_l} \right\rceil+1,$ we conclude that the second part of the theorem still holds true by replacing $\lambda/p$ with $\sum_{l=1}^d \lambda_l/p$ and $p^{\mathcal{B}(\alpha)}$  with $\sum_{l=1}^d p^{\mathcal{B}(\alpha_l)}$. We emphasize that in the above settings, as $\tau$ defined in (\ref{eq_defntau}) satisfies that $\tau \rightarrow 2$ as $n \rightarrow \infty$, the bulk eigenvalues can be characterized by the same MP law as in (\ref{eq_formrigidity}) and (\ref{eq_formrigidity2}). However, the number of the outliers  may depend on the signal strength $\lambda_i$, $1 \leq i \leq d.$ See Section \ref{sec_generalizationrevision2} for more discussion and simulation with $d=2$.
\end{enumerate}
The second case includes various combinations, and some of them might be challenging. We discuss only one setup assuming the signal strengths fall in two different regions; that is, there exists $1<r<d$ such that
\begin{equation*}
\alpha_1 \geq \alpha_2 \geq \alpha_r\geq 1>\cdots>\alpha_{r+1} \geq \cdots \alpha_d\,.
\end{equation*}
Theorem \ref{thm_informativeregion} still holds by replacing $\alpha$ by $\max_i\alpha_i$ and $\Wb_1$ in (\ref{eq_wba1}) by $\mathsf{W}_1$, where $ \bm{z}_i=({\zb}_{i1}, \cdots, {\zb}_{ir},0, \cdots, 0)$ and 
$\mathsf{W}_1(i,j)=\exp\left(-\frac{\| \bm{z}_i-\bm{z}_j \|_2^2}{h} \right)$,
i.e., the clean affinity matrix is defined by only using those components with large SNRs. That is to say, the spectrum of $\Wb$ for the value of $d$ is close to the clean affinity matrix for the value of $r$ with the same signals $\lambda_1 \geq \lambda_2 \geq \cdots \geq \lambda_r.$ The detailed statements and proofs are similar to the case $d=1$ except for extra notional complication. We defer more discussion of this setup with $d=2$ to Section \ref{sec_generalizationrevision2}. 

In short, there are still several open challenges when $d>1$, and we will explore them in our future work.

\end{remark}

\subsection{Spectrum of transition matrices}\label{sec_transitionmatrixab}
{In this subsection, we state the main results for the transition matrix $\Ab$ defined in (\ref{eq_noramlizedmatrix}).}
Even though there is an extra normalization step by the degree matrix $\Db$, we will see that most of the spectral studies of $\Ab$ boil to those of $\Wb.$ 
In what follows, we provide the counterparts of the results in Sections \ref{sec_result_fixedepsilon} and \ref{sec_informativeregion} for $\Ab.$ Similar discussions as those in Remark \ref{rmk_partone} and Remark \ref{rmk_multipled} also hold. 

For $\Wb_{a_1}$ defined in (\ref{eq_wba1}), denote $\Ab_{a_1}=\Db_{a_1}^{-1} \Wb_{a_1}$ similar to the definition (\ref{eq_noramlizedmatrix}). Similarly, for $\widetilde{\Wb}_{a_1}$ defined in (\ref{eq_defntildewa1}), denote $\widetilde{\Ab}_{a_1}=\widetilde{\Db}_{a_1}^{-1}\widetilde{\Wb}_{a_1}.$

\begin{corollary}\label{thm_normailizedaffinitymatrix}
Suppose \eqref{eq_defnw}-\eqref{eq_lambdadefinition} hold true, $d=1$, $h=p$ and $f(x)=\exp(-\upsilon x)$. When $0 \leq \alpha <1,$ the results of Theorems \ref{thm_affinity matrix} and \ref{lem_affinity_slowly} hold for the eigenvalues of $n\Ab$ by replacing $\Wb$ with $n\Ab$ and the measure $\nu_0$ with 
$\check\nu_0=T_{\varsigma(0)} \mu_{c_n, -2f'(\tau(0))/( f(\tau(\lambda)))}$.
When $1\leq  \alpha<2,$ for (1) of Theorem \ref{thm_informativeregion}, the counterpart of (\ref{eq_informativeequationone}) reads as 
\begin{equation*}
\| \Ab-\Ab_{a_1} \| \prec n^{\frac{\alpha-2}{2}}, 
\end{equation*}
and the counterpart of (\ref{eq_informativeequationtwo}) is 
\begin{equation*}
\lambda_i(\Ab_{a_1}) \prec n^{\frac{\alpha-3}{2}}. 
\end{equation*}
Moreover, when $\alpha \geq 2,$ for (2) of Theorem \ref{thm_informativeregion}, the counterpart of (\ref{eq_informativeequationoneone}) reads as
\begin{equation*}
\| \Ab-\widetilde{\Ab}_{a_1} \| \prec n^{-\frac{1}{2}}. 
\end{equation*}
The rest parts hold by replacing $\widetilde{\Wb}_{a_1}$ and $\Wb$ with $\widetilde{\Ab}_{a_1}$ and $\Ab,$ respectively. 
\end{corollary}

\section{Main results (II): a different bandwidth choice $h \asymp (p+\lambda)$}\label{section_result_2ndchoice}
{As discussed after Theorem \ref{thm_informativeregion}, when the SNR is large, the classic bandwidth choice $h \asymp p$ is too small compared with the signal. For example, according to (2) of Theorem \ref{thm_informativeregion}, we cannot obtain any information about the clean signal when $\alpha \geq 2$ if $h \asymp p.$
To address this issue, we consider a different bandwidth $h \asymp (p+\lambda)$, where $\lambda$ is the signal strength. We show that this signal dependent bandwidth will result in a meaningful spectral convergence result, which is stated in Theorem \ref{thm_adaptivechoiceofc}.} As in Section \ref{section main result}, we focus on the setting $d=1$, and set $\lambda:=\lambda_1$. The discussion for the setting $d>1$ is similar to that in Remark \ref{rmk_multipled}, which we only state the difference here. When $0 \leq \alpha_i \leq 1,$ $i=1,\ldots,d$, we have $h \asymp \sum_{l=1}^d \lambda_l +p \asymp p$, so all the arguments in (3) and (4) of Remark \ref{rmk_multipled} directly apply. When $\alpha_i>1$, $1 \leq i \leq d$, which corresponds to the strong signal cases (1) and (2) in Remark \ref{rmk_multipled}, following the proof of (\ref{eq_largealpharesultadap}) below, a similar argument of case (2) of Remark \ref{rmk_multipled} still apply.
For definiteness, below we state our results for the spectra of $\Wb$ and $\Ab$ assuming $h=p+\lambda$ in Section \ref{sec_mainresultssecondcase}. 
Since $\lambda$ is usually unknown in practice, in Section \ref{sec_bandwidthselectionalgo}, we propose a bandwidth selection algorithm for practical implementation. With this algorithm, even without the knowledge of the signal strength, we can still get meaningful spectral results. See Corollary \ref{coro_adaptivechoiceofc}.

\subsection{Spectra of affinity and transition matrices}\label{sec_mainresultssecondcase} 

In this subsection, we state the results for the spectra of $\Wb$ and $\Ab$ when $h=\lambda+p$. Denote 
\begin{align}\label{eq_wa2}
\Wb_{a_2} =\exp\left(-\frac{2p\upsilon}{h} \right) \Wb_1 +\left(1-\exp\left(-\frac{2p\upsilon}{h} \right)\right) \mathbf{I}_n,  
\end{align}
where $\Wb_1$ is constructed using the bandwidth $h=\lambda+p.$

\begin{theorem}\label{thm_adaptivechoiceofc}
Suppose \eqref{eq_defnw}-\eqref{eq_lambdadefinition} hold true, $d=1$ and $h=\lambda+p$. The following results hold.
\begin{enumerate}
\item When $0 \leq \alpha<1,$ Theorems \ref{thm_affinity matrix} and \ref{lem_affinity_slowly}  hold  with $\nu_0$ replaced by 
\begin{equation*}
\tilde\nu_0=T_{\varsigma_h} \mu_{c_n,\eta}\,, 
\end{equation*}
where 
\begin{align}\label{eq_varsigmah}
&\eta:=\frac{2p \upsilon \exp(-2 p \upsilon/h)}{h},  \\
 & \varsigma_h := \varsigma_h(\tau):=1-\frac{2 \upsilon p}{h}\exp(-\upsilon \tau p/h)-\exp(-\upsilon \tau p/h). \nonumber
\end{align}

\item When $\alpha \geq 1,$ we have that 
\begin{equation}\label{eq_closesecondbandwidth}
\left\| \frac{1}{n} \Wb-\frac{1}{n} \Wb_{a_2} \right\| \prec n^{-1/2}, 
\end{equation}
and for some large constants $D>2$ and $C>0,$ we have that when $i \geq C \log n$ 
\begin{align}\label{eq_closebandwidthtwo}
& \left|\lambda_i(\Wb_{a_2})-(1-\exp(-2p \upsilon/(p+\lambda)))\right|\prec n^{-D}. 
\end{align}
Moreover, when $\alpha>1,$ we have that 
\begin{equation}\label{eq_largealpharesultadap}
\left\| \frac{1}{n} \Wb-\frac{1}{n} \Wb_{1} \right\| \prec n^{-1/2}+n^{1-\alpha}.
\end{equation} 
\end{enumerate} 
Finally, similar results hold for the transition matrix $\Ab$ by replacing $\frac{1}{n} \Wb$, $\frac{1}{n} \Wb_{a_2}$ and $\frac{1}{n} \Wb_{1}$ in \eqref{eq_closesecondbandwidth}-\eqref{eq_largealpharesultadap} by $\Ab$, $\Ab_{a_2}$ and $\Ab_1$ respectively, where $\Ab_{a_2}$ and $\Ab_1$ are defined by plugging $\Wb_{a_2}$ and $\Wb_1$ into \eqref{eq_noramlizedmatrix}. 
\end{theorem}

With this bandwidth, we have addressed the issue we encountered in (2) of Theorem \ref{thm_informativeregion}. Specifically, according to Theorem \ref{thm_adaptivechoiceofc}, we find that when $\alpha<1,$ the noise dominates and $h=\lambda+p \asymp p$ does not lead to any essential difference in the spectrum compared with that from the fixed bandwidth $h=p.$ 
On the other hand, when $\alpha \geq 1$, such a bandwidth choice contributes significantly to the spectrum. Especially, compared to (\ref{eq_defntalpha}), $\Wb_{a_2}$ only has $O(\log n)$ nontrivial eigenvalues. Combined with \eqref{eq_closesecondbandwidth}, by properly choosing the bandwidth, once the SNR is relatively large, the noisy affinity matrix $\Wb$ captures the spectrum of the clean affinity matrix $\Wb_1$ via $\Wb_{a_2}$. In addition, when $\alpha>1$, we can replace $\Wb_{a_2}$  with $\Wb_1$ directly. Finally, we mention that in the small SNR region $1/2<\alpha \leq 1,$ modifying the transition matrix $\Ab$ by zeroing out the diagonal terms before normalization \cite{el2016graph} could be useful. For more discussions, we refer the readers to Section \ref{sec_generalizationrevision1}.

\subsection{An adaptive choice of bandwidth} \label{sec_bandwidthselectionalgo}

While the above result connects the spectra of noisy affinity matrices and those of clean affinity matrices, in general $\lambda$ is unknown. 
In this subsection, we provide an adaptive choice of $h$ depending on the dataset without providing an estimator for $\lambda.$ Such a choice will enable us to recover the results of Theorem \ref{thm_adaptivechoiceofc}.

Given some constant $0<\omega<1,$  we choose $h \equiv h(\omega)$ according to 
\begin{equation}\label{eq_bandwithchoice}
\int_{0}^h \mathrm{d} \mu_{\mathtt{dist}}= \omega\,,
\end{equation} 
where $\mu_{\mathtt{dist}}$ is the  empirical distribution of the pairwise distance $\{\|\mathbf{x}_i-\mathbf{x}_j \|_2^2\}, \ i \neq j.$ 
In this subsection, when there is no danger of confusion, we abuse the notation and denote the affinity and transition matrices conducted using $h$ from (\ref{eq_bandwithchoice}) as $\Wb$ and $\Ab$ respectively, and define $\Wb_1$ in (\ref{eq_singlematrix}) and $\Wb_{a_2}$ in (\ref{eq_wa2}) with $h$ in (\ref{eq_bandwithchoice}).

As we show in the proof, when $0\leq \alpha<1$, $h \asymp p$ (see (\ref{eq_claimhp}) in the proof), and when $\alpha\geq 1$, $h \asymp \lambda\asymp p+\lambda$ (see (\ref{eq_claimlambadasecond}) in the proof). In this sense, the following corollary recovers the results of Theorem \ref{thm_adaptivechoiceofc}, while the choice of bandwidth is practical.

\begin{corollary}\label{coro_adaptivechoiceofc}
Suppose \eqref{eq_defnw}-\eqref{eq_lambdadefinition} hold true and $d=1$.  For any $0<\omega<1$, let $h$ be the bandwidth chosen according to (\ref{eq_bandwithchoice}). Then we have Theorem \ref{thm_adaptivechoiceofc} holds true. 
\end{corollary}

Since Corollary \ref{coro_adaptivechoiceofc} recovers the results of Theorem \ref{thm_adaptivechoiceofc}, the same comments after Theorem \ref{thm_adaptivechoiceofc} and the discussions about manifolds in Subsection \ref{section manifold model} hold for Corollary \ref{coro_adaptivechoiceofc}. 
We comment that in practice, usually researchers choose the 25\% or 50\% percentile of all pairwise distances, or those distances of nearest neighbor pairs, as the bandwidth; see, for example, \cite{shnitzer2019recovering,lin2021wave}. Corollary \ref{coro_adaptivechoiceofc} provides a theoretical justification for this commonly applied ad hoc bandwidth selection method.

Next, we discuss how to choose $\omega$ in practice. 
Based on the obtained theoretical results, when $\alpha \geq 1,$ the outliers stand for the signal information (c.f. (\ref{eq_wa2})), except those associated with the kernel effect. Thus, we propose  Algorithm \ref{alg:choice} to choose $\omega$ adaptively which seeks for a bandwidth so that the affinity matrix has the most number of outliers.

\begin{algorithm}[ht]
\caption{Adaptive choice of $\omega$}
\label{alg:choice}
\begin{enumerate}   
\item  Take fixed constants $\omega_L< \omega_U,$ where $0<\omega_L, \omega_U<1,$ e.g., $\omega_L=0.05$ and $\omega_U=0.95$. For some large integer $T$, we construct a partition of the intervals $[\omega_L, \omega_U],$ denoted as $\mathcal{P}=(\omega_0, \omega_1, \cdots, \omega_T)$, where $\omega_i=\omega_L+\frac{i}{T}(\omega_U-\omega_L)$.

\item For the sequence of quantiles $\{\omega_i\}_{i=0}^T,$ calculate the associated bandwidths according to (\ref{eq_bandwithchoice}), denoted as $\{h_i\}_{i=0}^T.$ 
\item For each $0 \leq i \leq T,$ calculate the eigenvalues of the affinity matrices $\Wb_i$ which is conducted using the bandwidth $h_i.$ Denote the eigenvalues of $\Wb_i$ in the decreasing order as $\{\lambda_{k}^{(i)}\}_{i=1}^n.$ 

\item For a given threshold $\mathsf{s}>0$ satisfying $\mathsf{s} \rightarrow 0$ as $n \rightarrow \infty,$ denote
\begin{equation*}
\mathsf{k}(\omega_i): =\max_{1 \leq k \leq n-1}\left\{ k\bigg|\,\frac{\lambda_k^{(i)}}{\lambda_{k+1}^{(i)}} \geq 1+\mathsf{s} \right\}.
\end{equation*}

\item Choose the quantile $\omega$ such that 
\begin{equation}\label{eq_choiceomegalarger}
\omega=\max\Big[\underset{\omega_i}{\mathrm{arg max}} \ \mathsf{k}(\omega_i)\Big]\,.
\end{equation}
\end{enumerate}
\end{algorithm}

{Note that we need a threshold $\mathsf{s}$ in step 4 of Algorithm \ref{alg:choice}. We suggest to adopt the resampling method  established in \cite[Section 4]{PASSEMIER2014173} and \cite[Section 4.1]{DYS}. This method provides a choice of $\mathsf{s}$ to distinguish the outlying eigenvalues and bulk eigenvalues given the ratio $c_n=p/n$. 
The main rationale supporting this approach is that the bulk eigenvalues are close to each other (c.f. Remark \ref{rmk_partone}) and hence the ratios of the two consecutive eigenvalues will be close to one. }
Moreover, in step 5 of the algorithm, when there are multiple $\omega_i$ that achieve the argmax, we choose the largest one for the purpose of robustness.

Next, we numerically illustrate how the chosen $\omega$ by Algorithm \ref{alg:choice} depends on $\alpha$.
Consider the nonlinear manifold $S^1$, the canonical $1$-dim sphere, isometrically embedded in the first two axes of $\mathbb{R}^p$ and scaled by $\sqrt{\lambda}$, where $\lambda>0$; that is,
$\bz_i:=\sqrt{\lambda}[\cos \theta_{i}, \sin \theta_i, 0, \cdots, 0]^\top\in \mathbb{R}^p$, 
where $\theta_i$ is uniformly sampled from $[0,2\pi]$. Next, we add Gaussian white noise to $\bz_i$ via
$\mathbf{x}_i=\bz_i+\mathbf{y}_i \in \mathbb{R}^p,$ where $\mathbf{y}_i \sim \mathcal{N}(\mathbf{0}, \mathbf{I})$, $i=1,2,\cdots,n, $ are noise independent of $\xi_i.$ 
We consider this example since its topology is nontrivial and we know the ground truth. 
In Figure \ref{fig_omegachoice}, we record the chosen $\omega$ for different $\alpha$ from $\Wb$. When $\alpha$ is small, (i.e. $\alpha<1$), since the bandwidth choice will not essentially influence the transition, the algorithm will offer a large quantile in light of (\ref{eq_choiceomegalarger}). When $\lambda$ is large (i.e., $\alpha>2$), we get a small quantile. Intuitively, the larger selected bandwidth when $\alpha$ gets small can be understood as the algorithm trying to combat the noise with a larger bandwidth. {In Figure \ref{fig_omegachoice}, we also record the chosen $\omega$ for different $\alpha$ from $\Ab$, where we simply replace the role of $\Wb$ by $\Ab$. We can see the same result as that from $\Wb$.}
Finally, note that the choices of $\omega$ are irrelevant of the aspect ratio $c_n$ in (\ref{eq_ratio}). This finding suggests that the bandwidth selection is not sensitive to the ambient space dimension.

\begin{figure*}[!ht]
\centering
	\includegraphics[width=6cm]{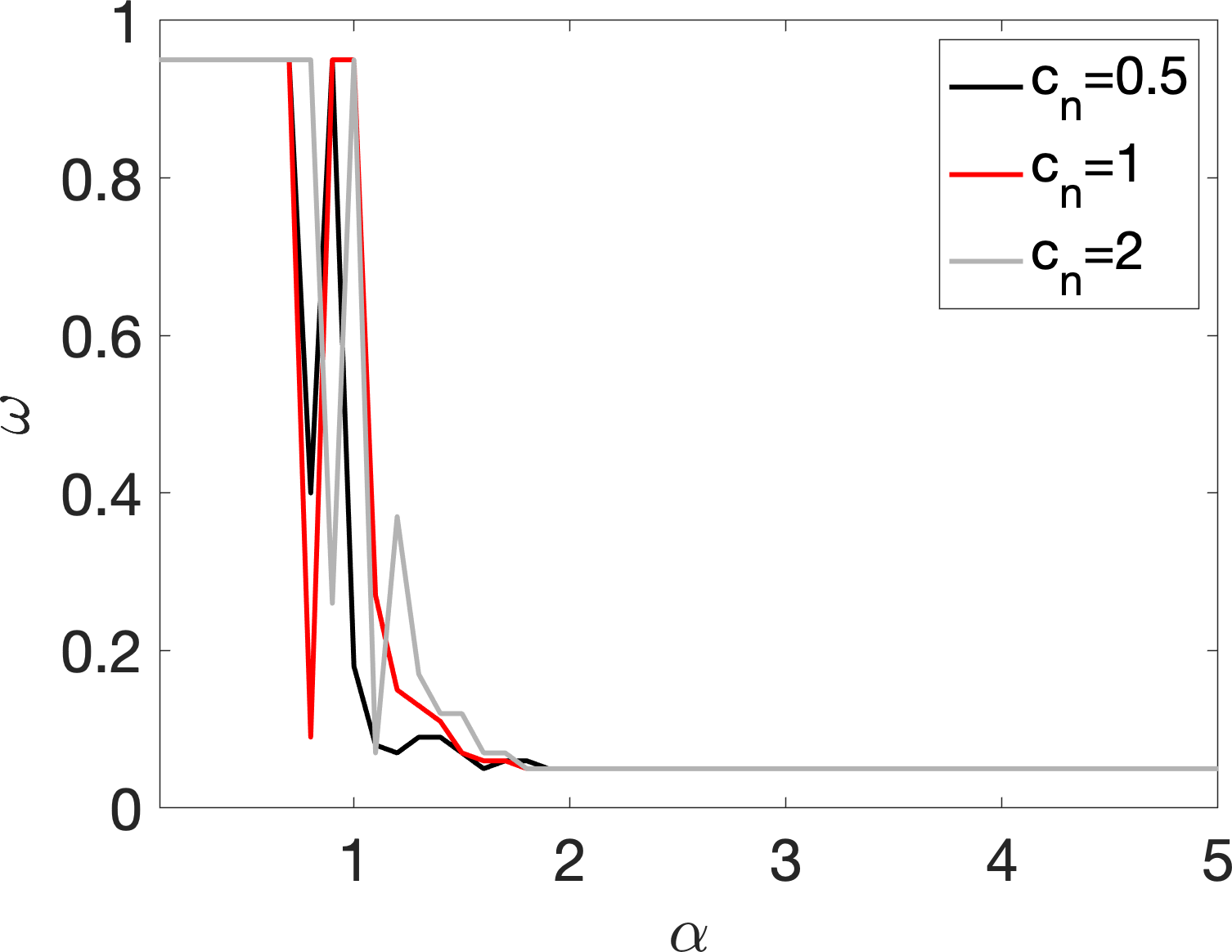}
	\includegraphics[width=6cm]{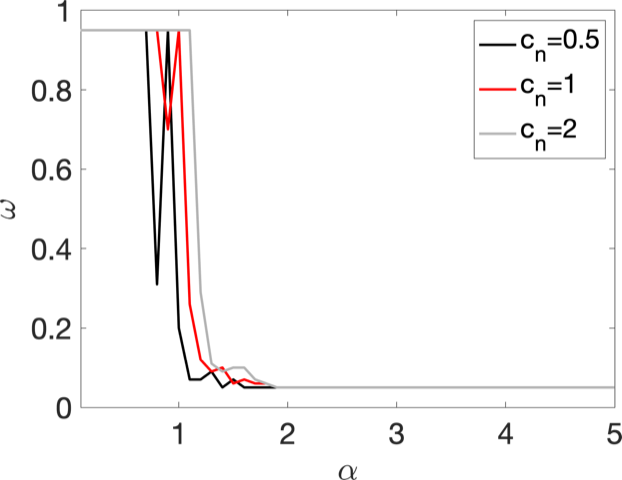}
	\caption{Adaptive choices of $\omega$ by the affinity matrix $\mathbf{W}$ and the transition matrix $\mathbf{A}$ are shown on the left and right subplots respectively. In the simulation, we use the Gaussian random vectors with $\lambda=n^{\alpha}$ under the setting $n=300$ for different values of $c_n$ in (\ref{eq_ratio}). The kernel function is $f(x)=\exp(-x/2)$ and the bandwidth is  chosen adaptively according to (\ref{eq_bandwithchoice}), where $\omega$ is chosen using Algorithm \ref{alg:choice} with $T=91$, $\omega_L=0.05$, and $\omega_U=0.95$.  $\mathsf{s}=0.24, 0.17, 0.12$ {are chosen using the resampling method as in \cite{PASSEMIER2014173} }for $c_n=0.5,1, 2,$ respectively.}\label{fig_omegachoice}
	\end{figure*}

\subsection{Connection with the manifold learning}\label{section manifold model}

To discuss the connection with the manifold learning, we focus on Theorem \ref{thm_adaptivechoiceofc} and $\alpha>1$, and hence Corollary \ref{coro_adaptivechoiceofc}, which states that via replacing $\frac{1}{n} \Wb$, $\frac{1}{n} \Wb_{a_2}$ and $\frac{1}{n} \Wb_{1}$ in \eqref{eq_closesecondbandwidth}-\eqref{eq_largealpharesultadap} by $\Ab$, $\Ab_{a_2}$ and $\Ab_1$, the relationship between the eigenvalues of $\Ab$ and $\Ab_1$ is established when $\alpha>1$.
We know that all except the top $C\log(n)$ eigenvalues of $\Wb$ are trivial according to \eqref{eq_closebandwidthtwo}, and by Weyl's inequality, the top $C\log(n)$ eigenvalues of $\Ab$ and $\Ab_1$ differ by $n^{-1/2}+n^{1-\alpha}$.
On the other hand, the eigenvalues of $\Ab_1$ have been extensively studied in the literature. Below, take the result in \cite{koltchinskii2000random} as an example. 
Suppose the clean data $\mathcal{Z}$ is sampled from a $m$-dim closed (compact without boundary) and smooth manifold, which is embedded in a $d$-dim subspace in $\mathbb{R}^p$, following a proper sampling condition on the sampling density function $\mathsf p$ (See Section \ref{sec_reducedproblem} for details of this setup). 
To link our result to that shown in \cite{koltchinskii2000random}, note that $\|\zb_i-\zb_j\|^2$ is of order $\lambda$ by assumption so that the selected bandwidth is of the same order as that of $\|\zb_i-\zb_j\|^2$. Thus, since $\lambda/(p+\lambda)\asymp 1$ when $n\to \infty$, the eigenvalues of $\Ab_1$ converge to the eigenvalues of the integral operator 
\[
Th(x)=\int_M \exp\left(-\upsilon\|x-y\|^2_2\right)h(y)\mathsf p(y)dV(y)\,, 
\]
where $h$ is a smooth function defined on $M$ and $dV$ is the volume density. 
By combining the above facts, we conclude that under the high-dimensional noise setup, when the SNR is sufficiently large and the bandwidth is chosen properly, we could properly obtain at least the top few eigenvalues of the associated integral kernel from the noisy transition matrix $\Ab$. Since our focus in this paper is not manifold learning itself but how the high-dimensional noise impacts the spectrum of GL, for more discussions and details about manifold learning, we refer readers to \cite{dunson2019diffusion} and the citations therein.

{\section{Numerical studies}\label{section numerical studies}
In this section, we conduct Monte Carlo simulations to illustrate the accuracy and usefulness of our results and proposed algorithm. In Section \ref{sec_simu_accuracy}, we conduct numerical simulations to illustrate the accuracy of our established theorems for various values of $c_n=0.5, 1,2.$ We also show the impact of $n.$ 
In Section \ref{sec_simu_algoeff}, we examine the usefulness of our proposed Algorithm \ref{alg:choice} and compare it with some methods in the literature with a linear manifold and a nonlinear manifold.

\subsection{Accuracy of our asymptotic results}\label{sec_simu_accuracy}

In this subsection, we conduct numerical simulations to examine the accuracy of the established results. For simplicity, we focus on checking the results in Section \ref{section main result} when $h=p,$ which is the key part of the paper. Similar discussions can be applied to the results in Section \ref{section_result_2ndchoice} when $h=\lambda+p.$ 

In Figure \ref{fig_accuray05lowsnrregime}, we study the low SNR setting when $0 \leq \alpha<1$ as in Section \ref{sec_result_fixedepsilon}, particularly the closeness of  bulk eigenvalues of $\Wb$ and the quantiles of the MP law $\nu_0$ shown in (\ref{eq_formrigidity}) and (\ref{eq_formrigidity2}). We also show that even for a relative small value of $n=200,$ our results are reasonably accurate for various values of $c_n=0.5,1,2.$ 
Then we study the region when $\alpha \geq 1$ as in Section \ref{sec_informativeregion}. In Figure \ref{fig_accuray05intermediateregime}, we study the SNR region when $1 \leq \alpha<2$ and check (\ref{eq_informativeequationone}). Moreover, in Figure \ref{fig_accuray05largeregime}, we examine the accuracy of (\ref{eq_c1exp}) when $\alpha$ is very large. Again, we find that our results are accurate even for a relatively small value of $n=200$ under different settings of $c_n=0.5,1,2.$

Since our results are stated in the asymptotic sense when $n$ is sufficiently large, in Figure \ref{fig_dimensionanalysis}, we examine how the value of $n$ impact our results. For various values of $c_n$ and SNRs, we find that our results are reasonable accurate once $n \geq 100$. We see that when the SNR becomes large, our results for small $n$, like $n<100$, are more accurate. 

We point out that for a better visualization, we report the results for the sorted eigenvalues instead of the histogram in the above plots regarding eigenvalues. The main reason is that we focus on each individual eigenvalue rather than the global empirical distribution, which has been known in the literature. The simulation results are based on only one trial which emphasizes the concentration with high probability as established in our main results. For the visualization of histogram, we have to run several simulations and look at the average. Since this is not the main focus of the paper, we only generate one such plot in Figure \ref{fig_accuray05lowsnrregimehist} based on 1,000 trials.

 \begin{figure*}[!ht]
\includegraphics[width=4cm]{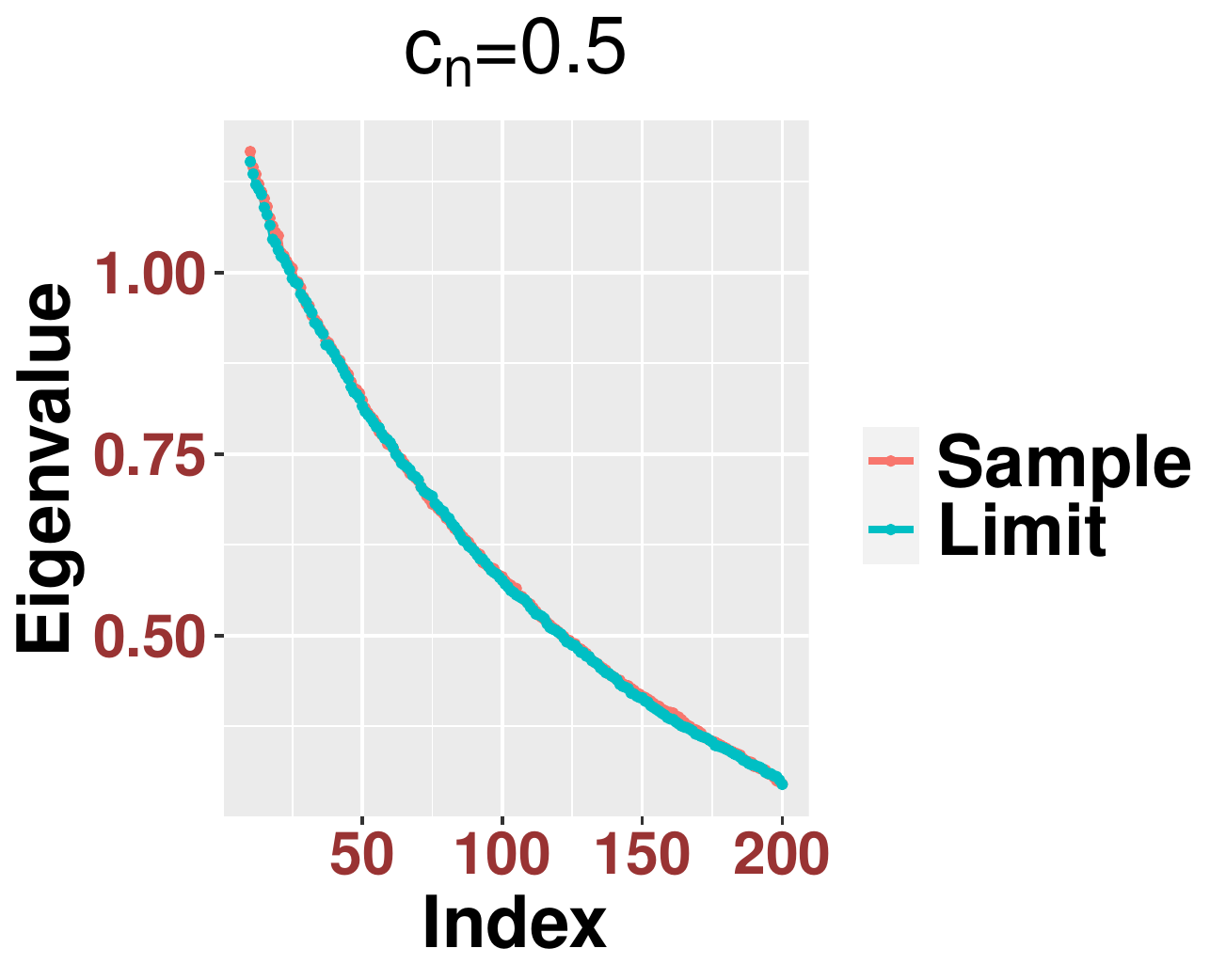}
\includegraphics[width=4cm]{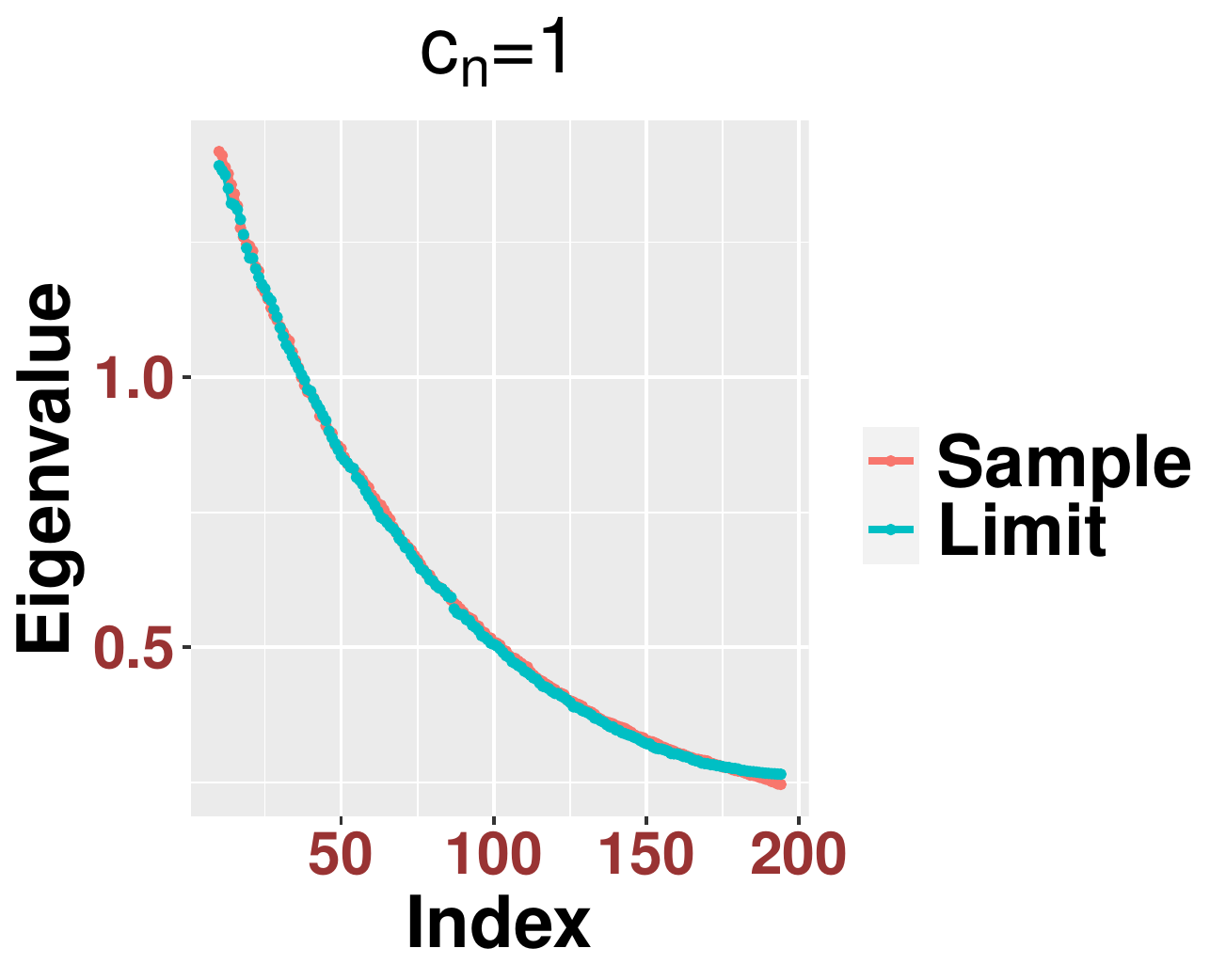}
	\includegraphics[width=4cm]{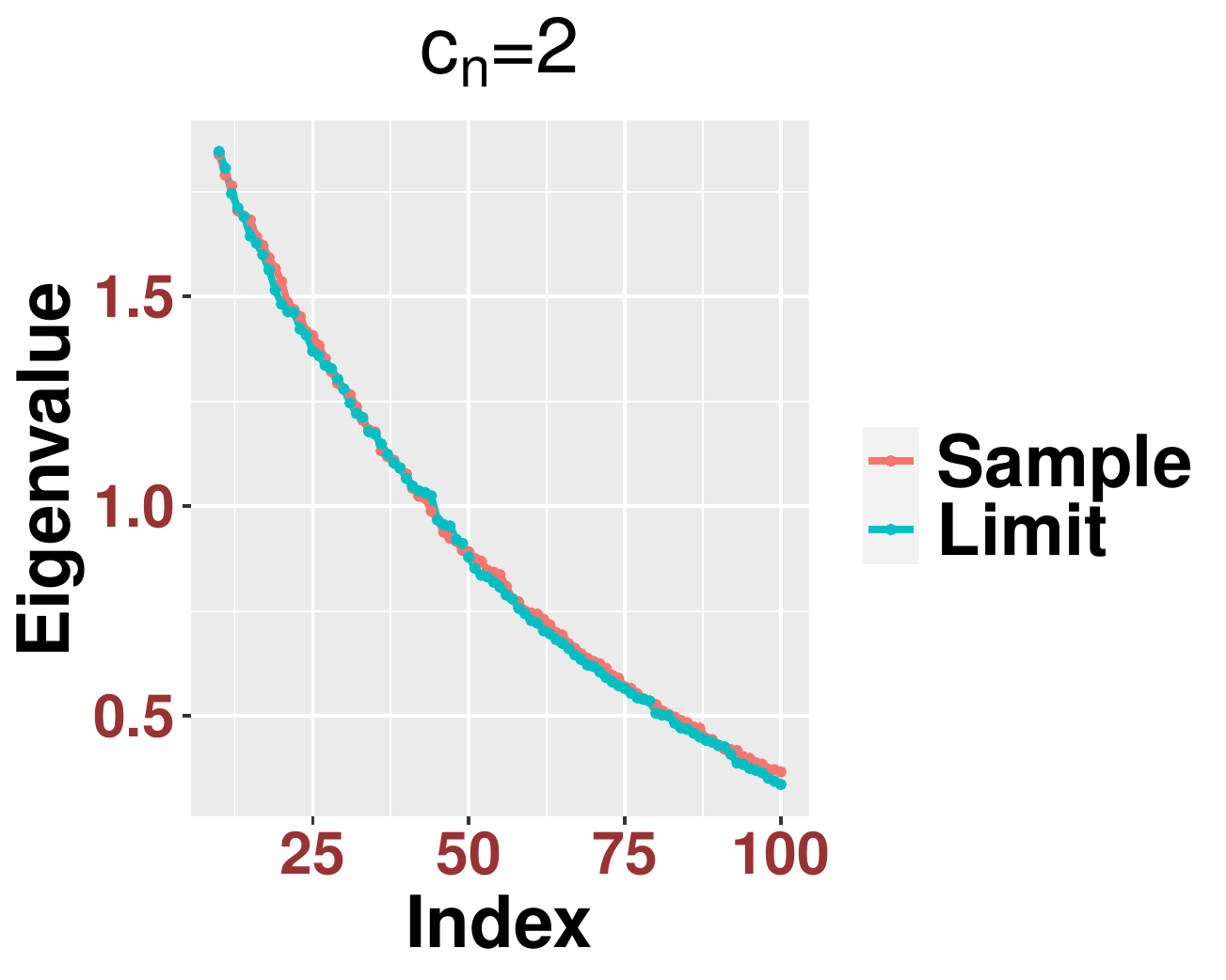}
	\caption{Low SNR setting. Here $\lambda=p^{0.2}$, $n=200$ with $c_n=n/p, h=p$ and the random variables are standard Gaussian. To examine the bulk eigenvalues, we start from the 10th eigenvalue of $\Wb$, i.e., $\lambda_i(\Wb)$, where $i \geq 10.$ For the legend, the sample means the eigenvalues of $\Wb$ and the limit means the quantiles (c.f. (\ref{eq_typical})) of the MP law $\nu_0$ defined in (\ref{eq_nudefinition}).} \label{fig_accuray05lowsnrregime}
\end{figure*}

 \begin{figure*}[!ht]
\includegraphics[width=4cm]{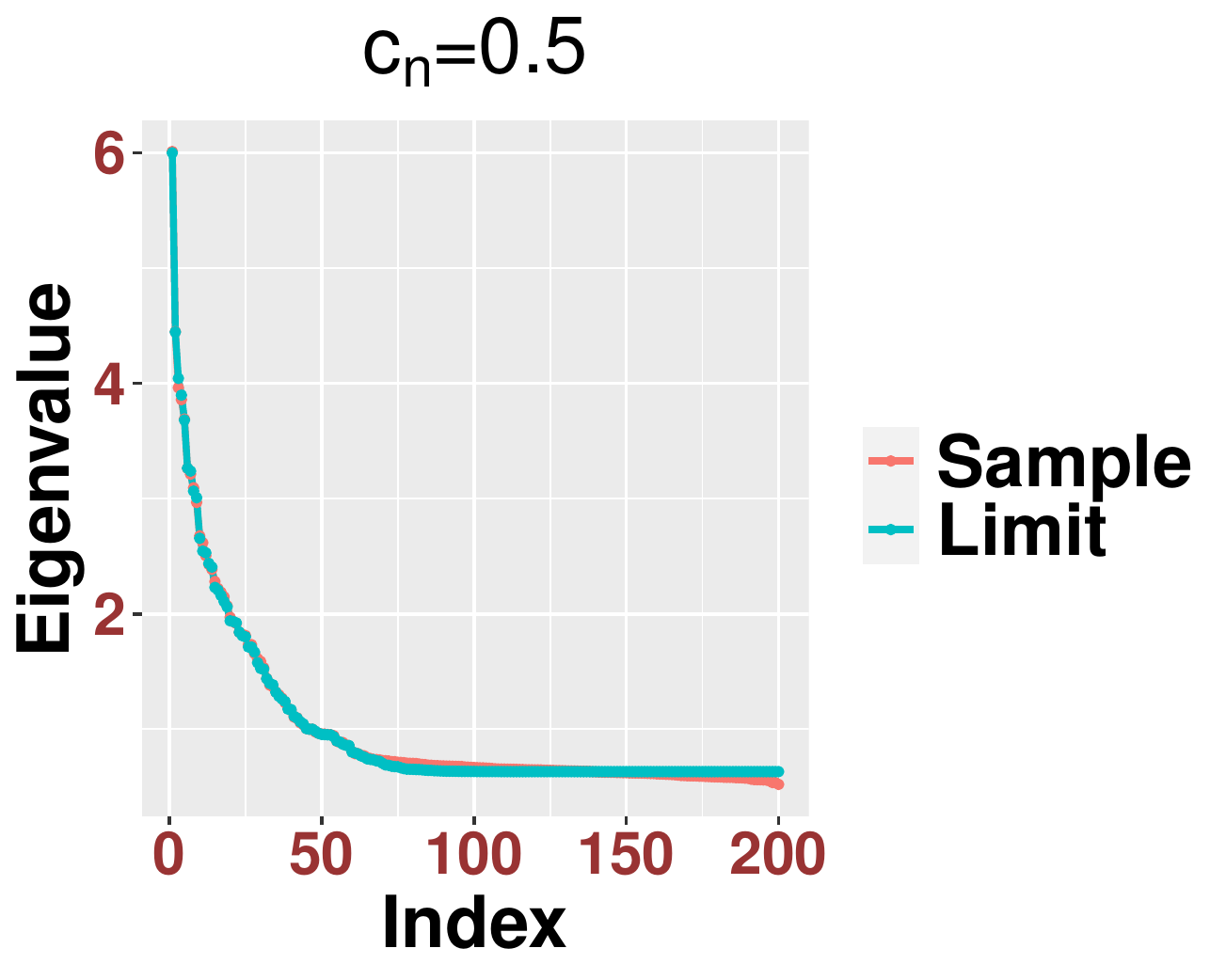}
\includegraphics[width=4cm]{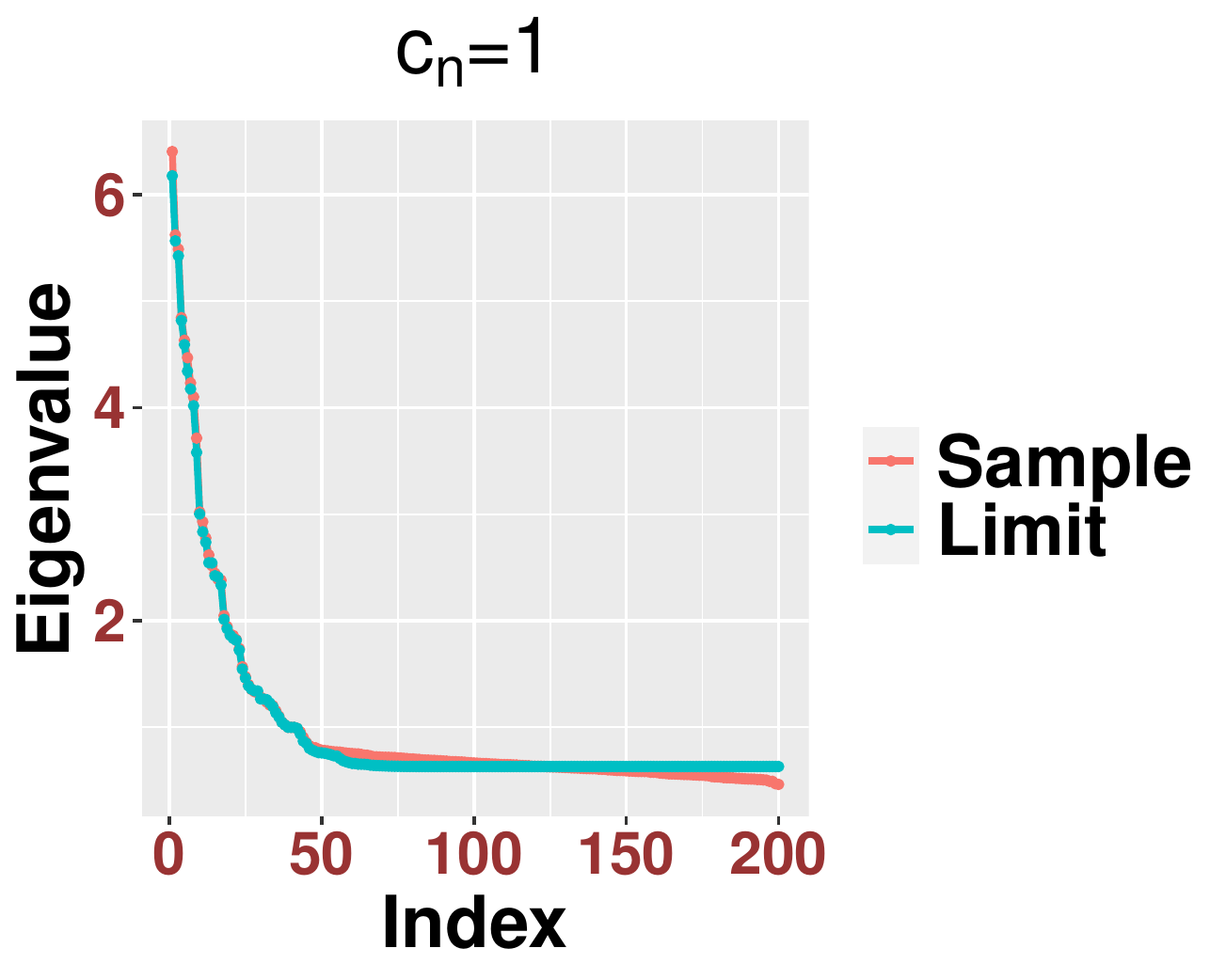}
	\includegraphics[width=4cm]{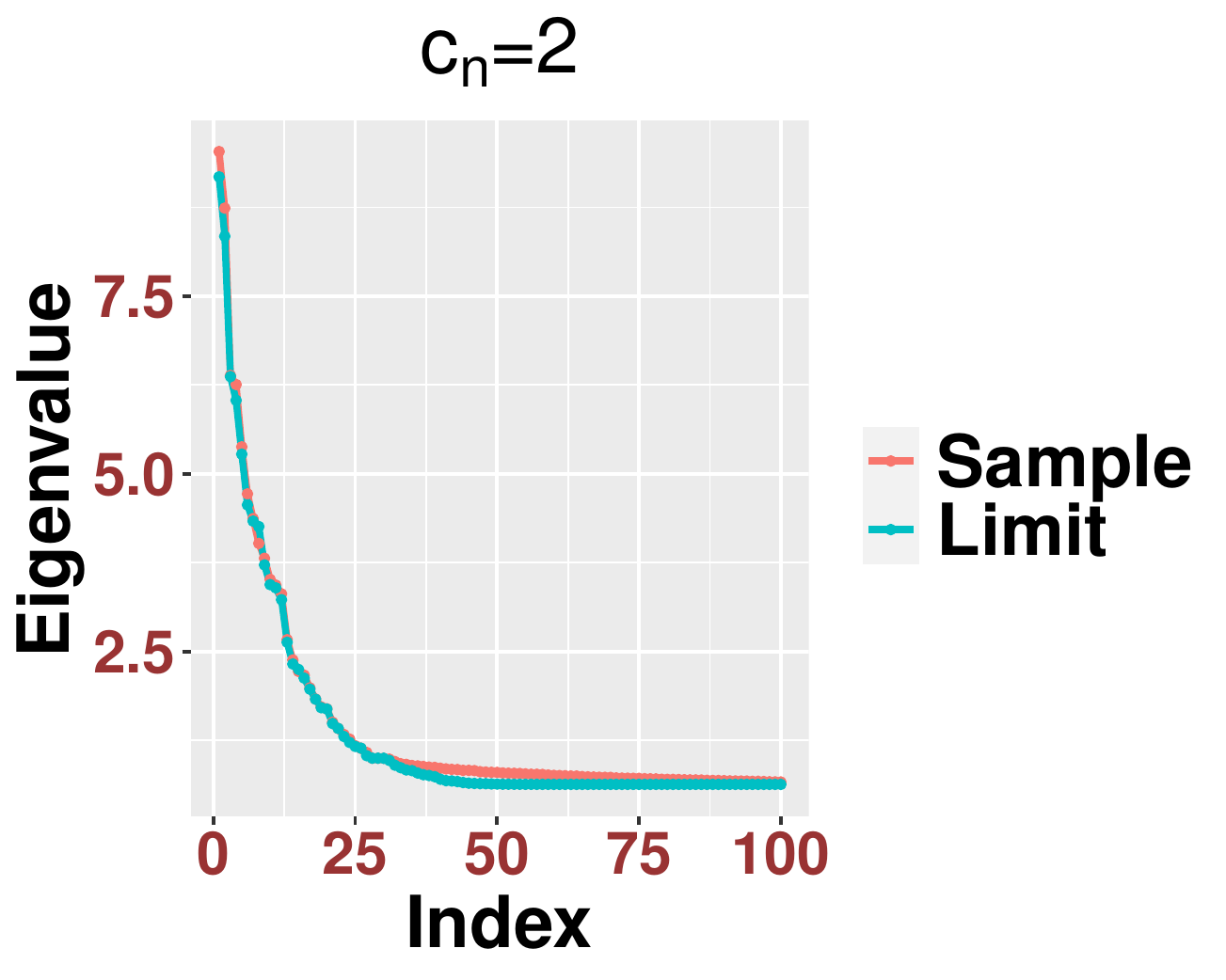}
	\caption{Moderate SNR region.  Here $\lambda=p^{1.9}$, $n=200$ with $c_n=n/p, h=p$ and the random variables are standard Gaussian. For the legend, the sample means the eigenvalues of $\Wb$ and the limit means the eigenvalues of $\Wb_{a_1}$ as in (\ref{eq_informativeequationone}). } \label{fig_accuray05intermediateregime}
\end{figure*}

 \begin{figure*}[ht]
\includegraphics[width=4cm]{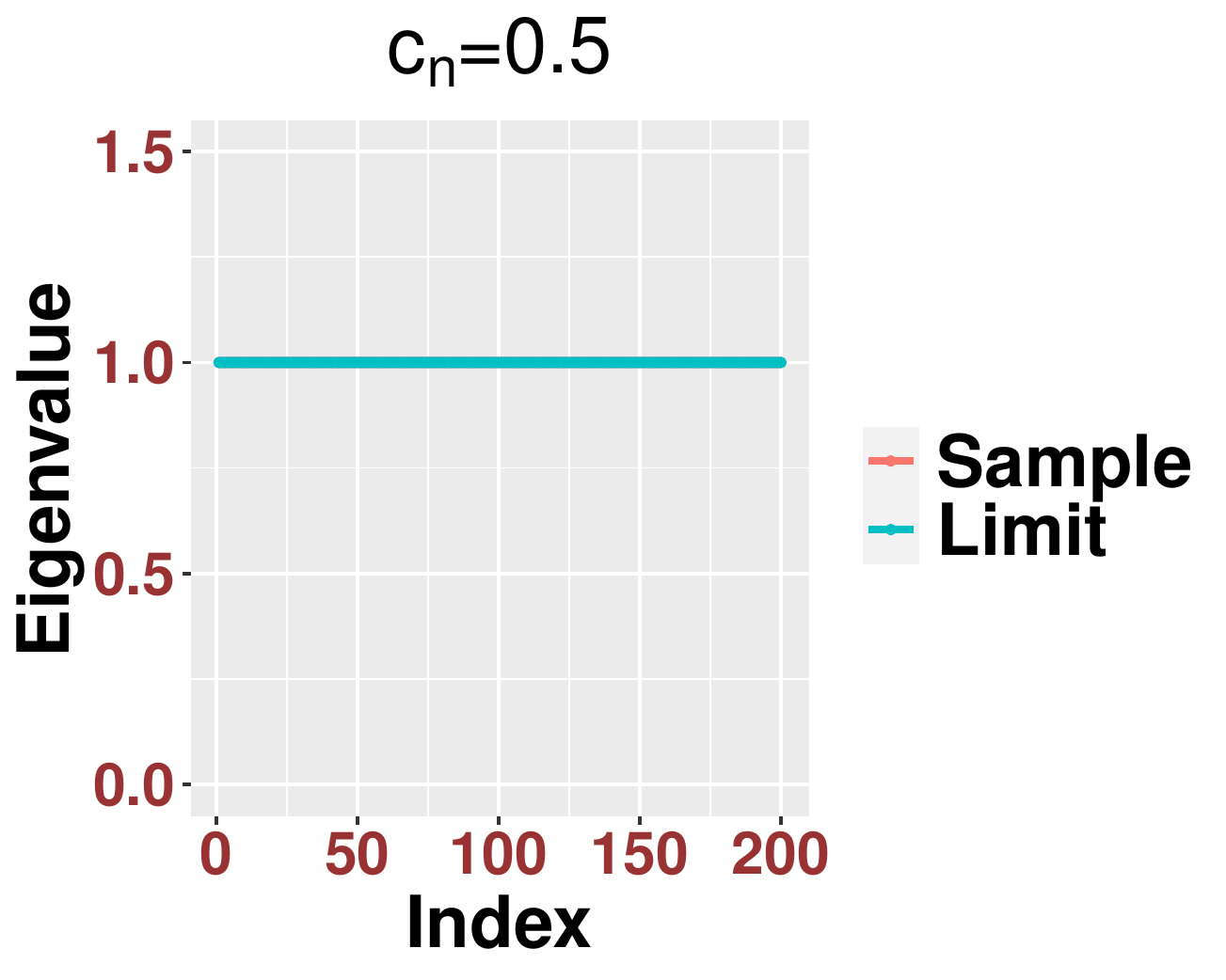}
\includegraphics[width=4cm]{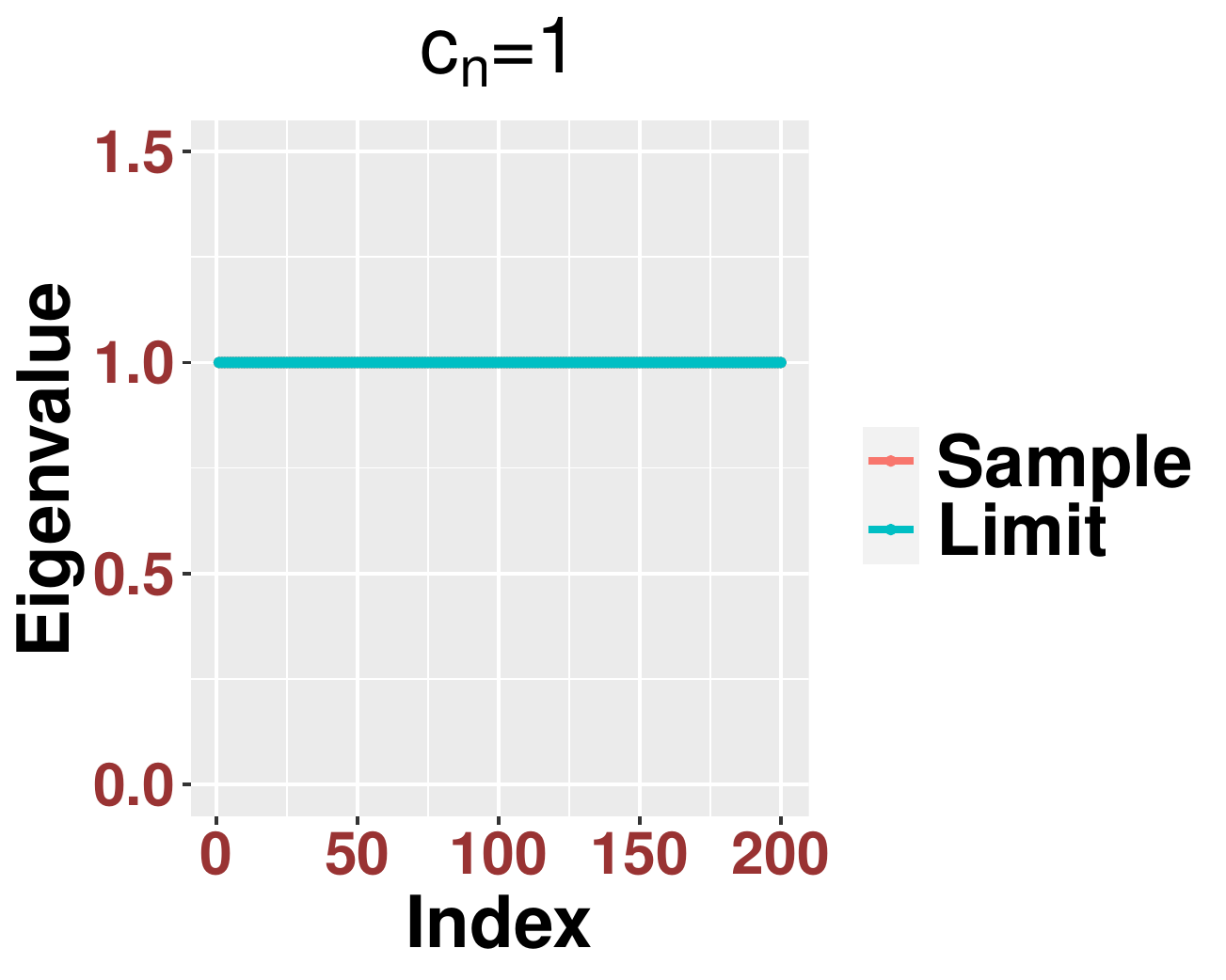}
	\includegraphics[width=4cm]{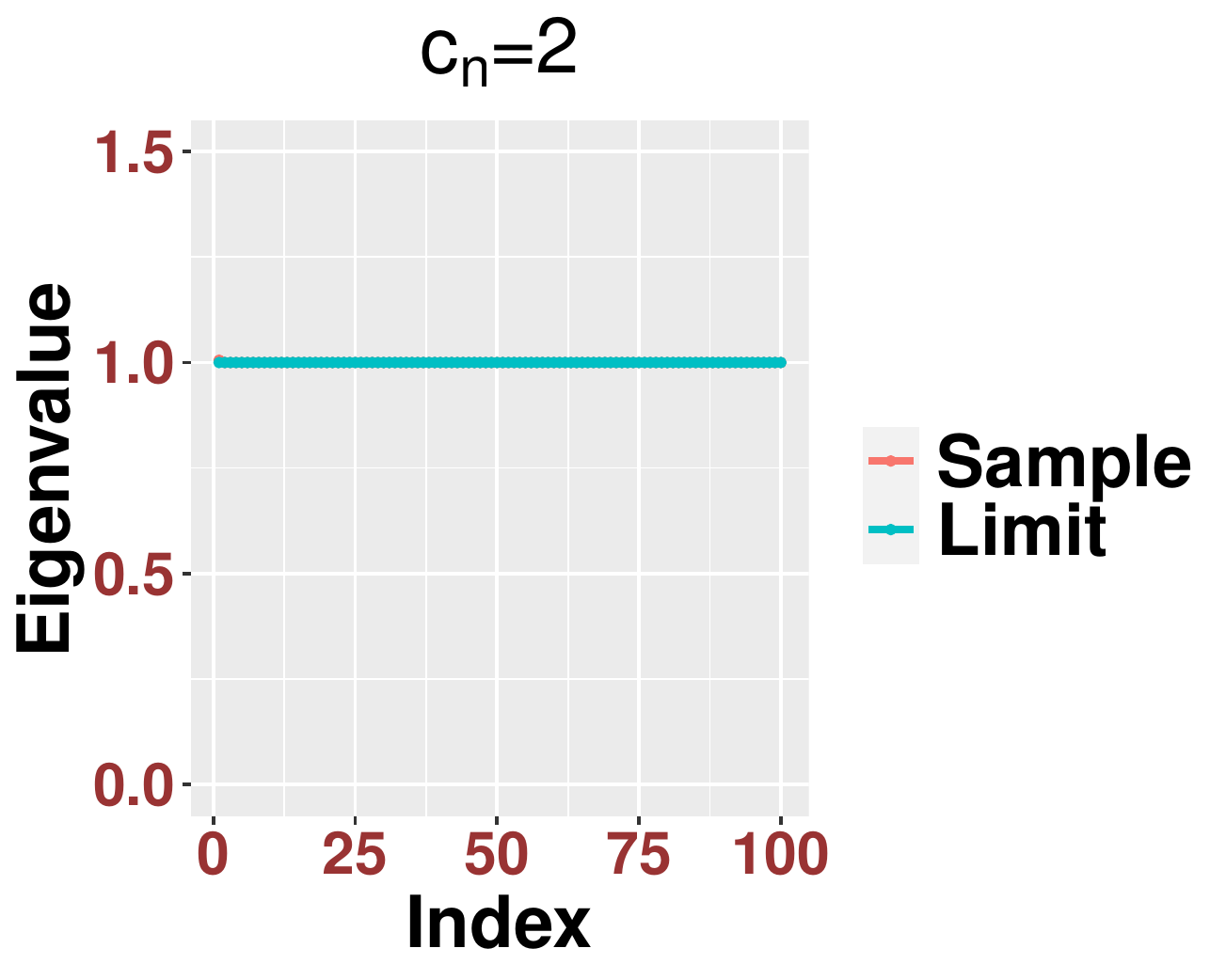}
	\caption{Large SNR setting.  Here $\lambda=p^{5}$, $n=200$ with $c_n=n/p, h=p$ and the random variables are standard Gaussian. For the legend, the sample means the eigenvalues of $\Wb$ and the limit means unity as predicted in (\ref{eq_c1exp}).} \label{fig_accuray05largeregime}
\end{figure*}

 \begin{figure*}[!ht]
\includegraphics[width=4cm]{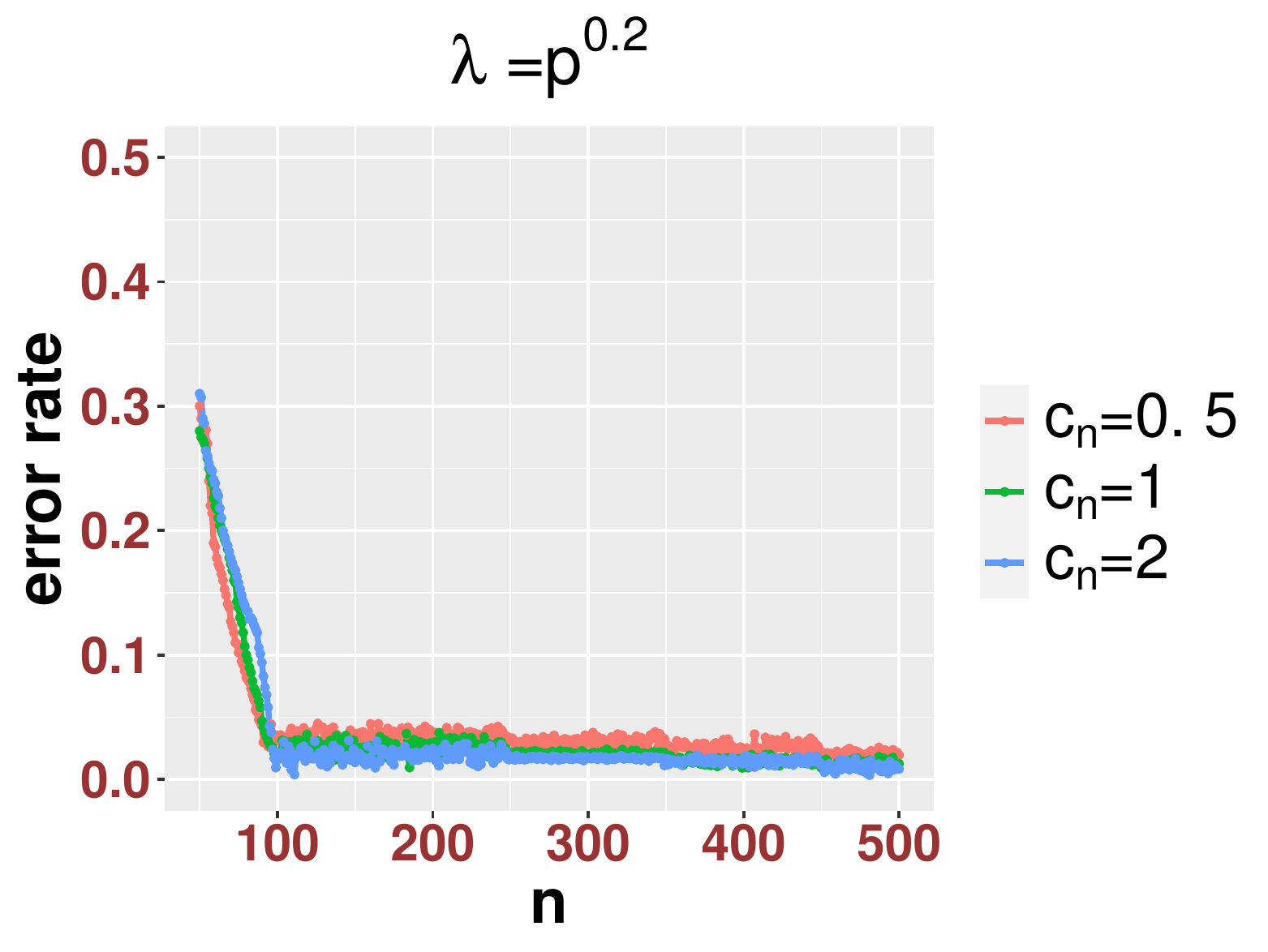}
\includegraphics[width=4cm]{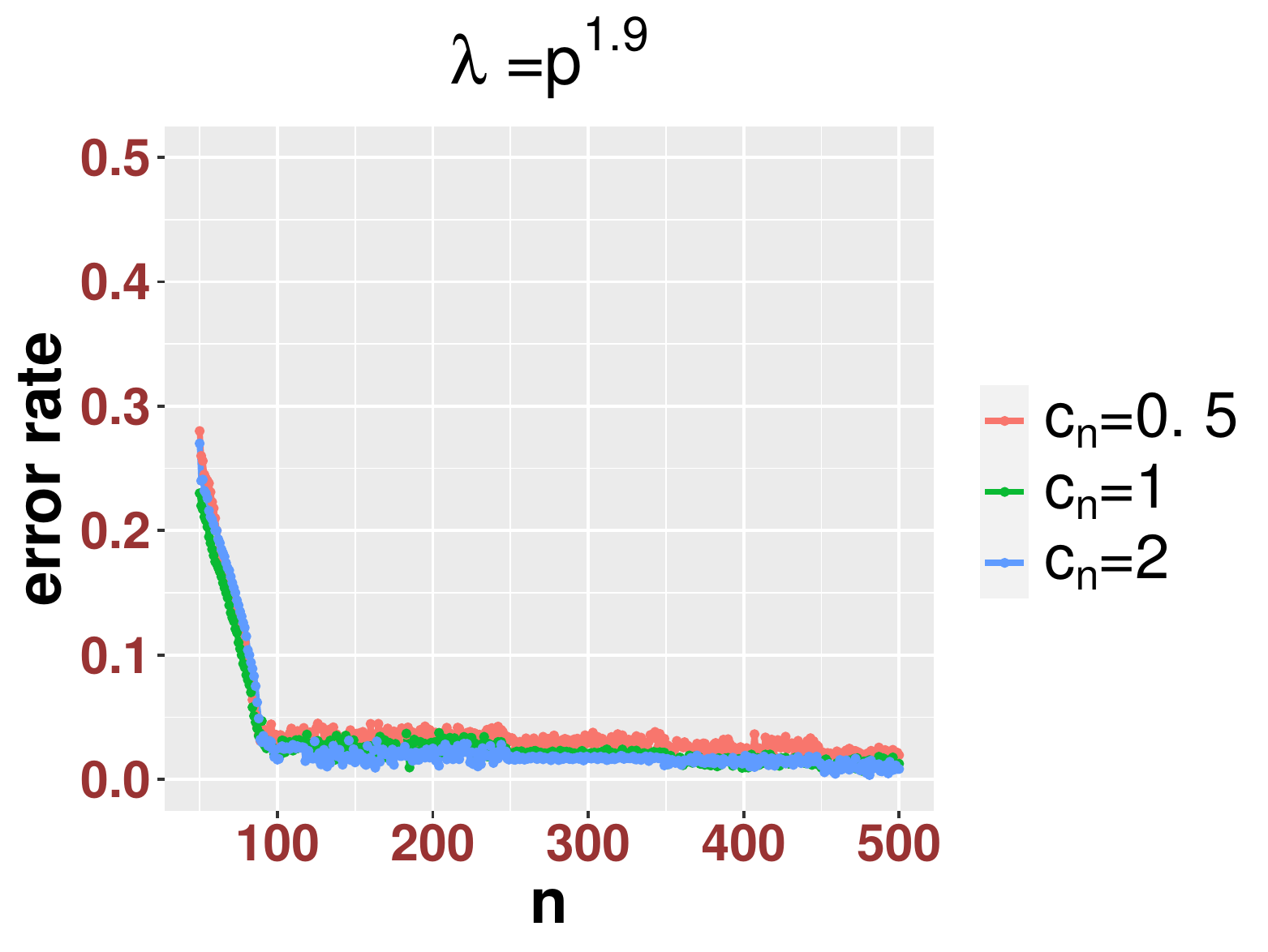}
	\includegraphics[width=4cm]{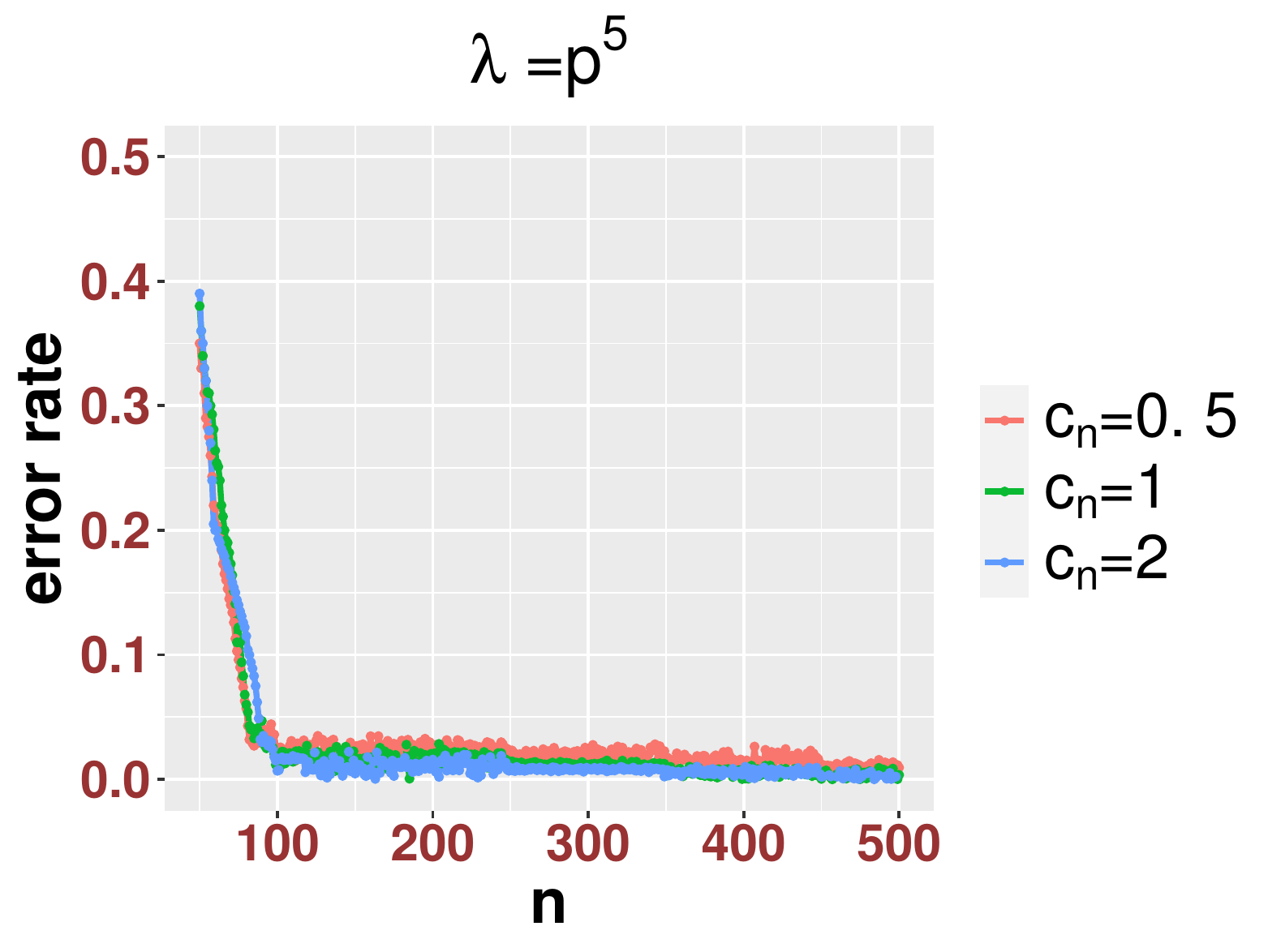}
	\caption{Dimension analysis on $n$. We examine how the values of $n$ influence the results as in Figures \ref{fig_accuray05lowsnrregime}--\ref{fig_accuray05largeregime}. For the error rate, when the SNR is small that $\lambda=p^{0.2},$ the error rate is denoted as $\sup_{i \geq 10}|\lambda_i(\Wb)-\gamma_{\nu_0}(i)|.$ When the SNR is large that $\alpha \geq 1,$ the error rate is denoted as the operator norm of difference between $\Wb$ and its corresponding limit matrix. Here $c_n=n/p, h=p$ and the random variables are standard Gaussian. We conclude that once $n$ is  relatively large say, $n \geq 100,$ our results will be reasonably accurate. We also observe that when the SNRs are large, it seems that we will need smaller values of $n$ for the purpose of accuracy. } \label{fig_dimensionanalysis}
\end{figure*}

\begin{figure*}[!ht]
\includegraphics[width=4cm]{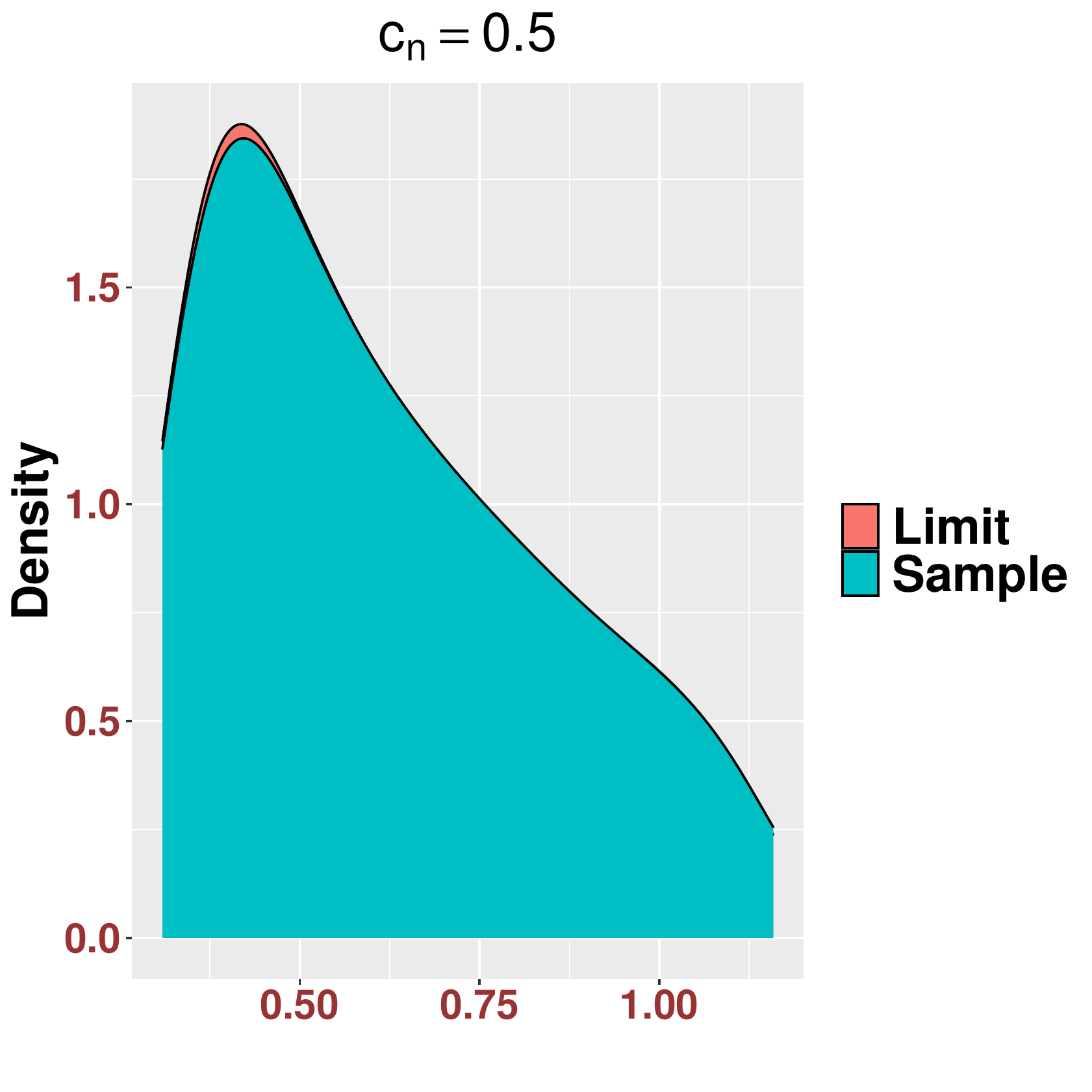}
\includegraphics[width=4cm]{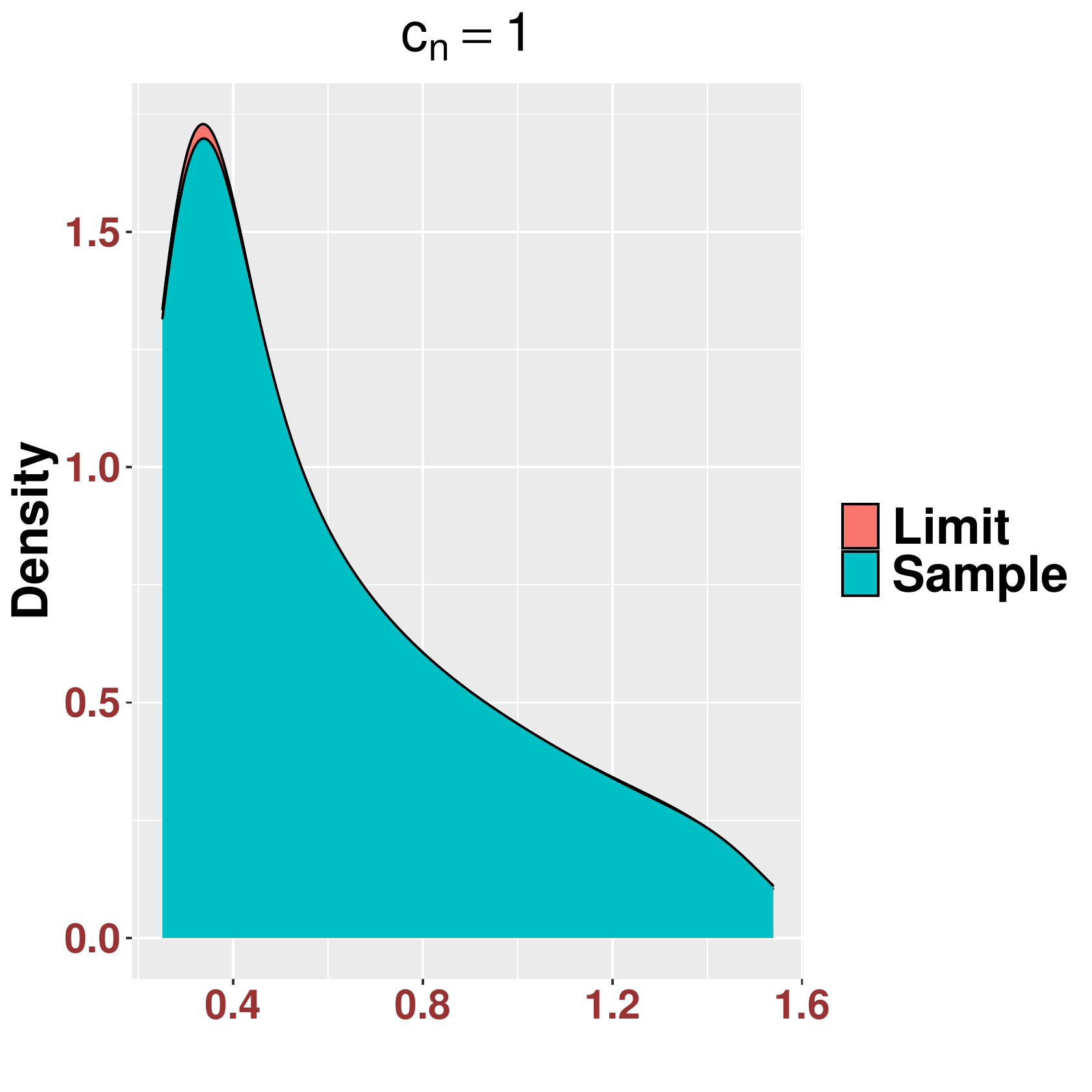}
	\includegraphics[width=4cm]{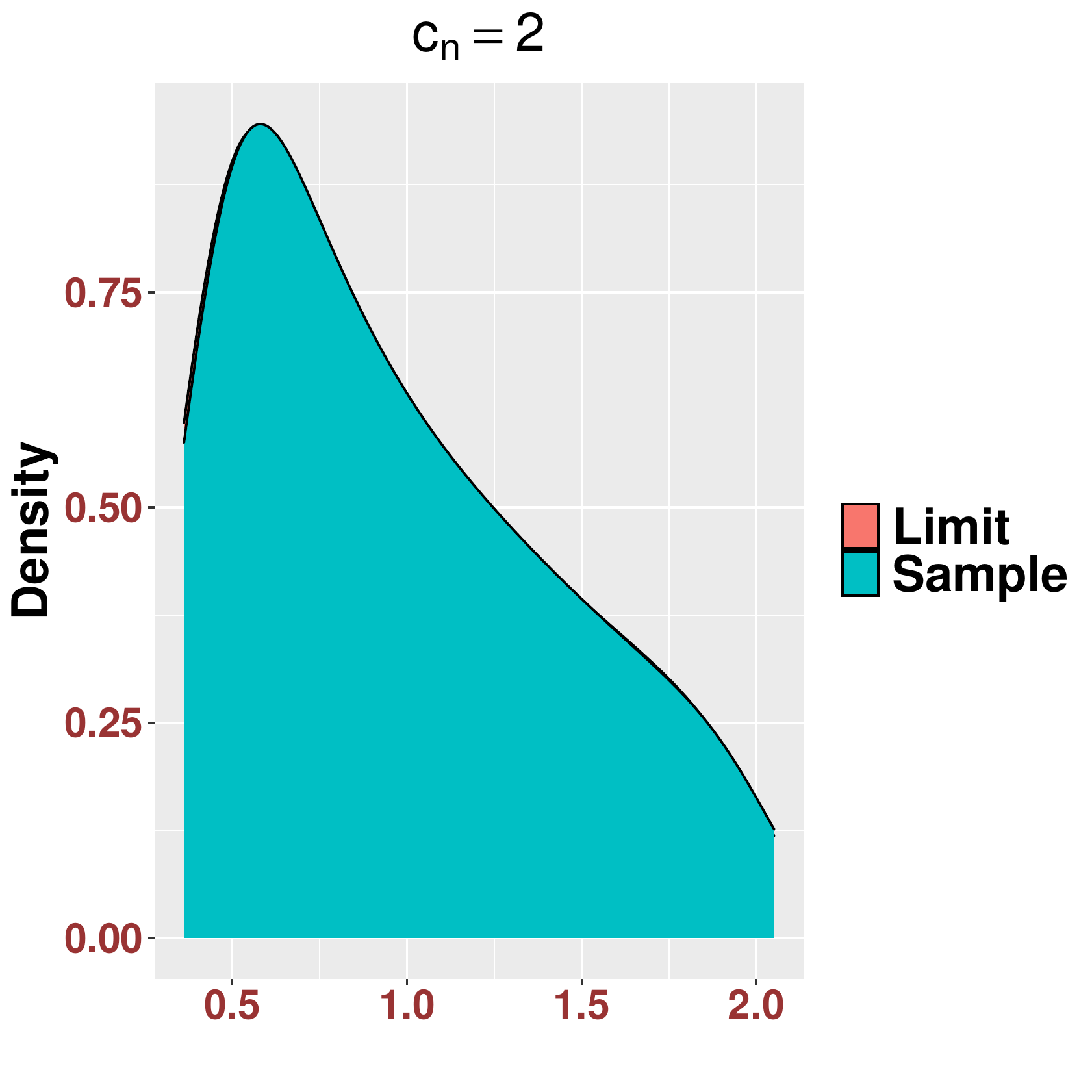}
	\caption{Histogram of low SNR setting. The settings are the same as in the caption of Figure \ref{fig_accuray05lowsnrregime}. The simulations are based on 1,000 repetitions. In the plots, we have removed the point mass. } \label{fig_accuray05lowsnrregimehist}
\end{figure*}

\subsection{Efficiency of our proposed bandwidth selection algorithm (Algorithm \ref{alg:choice})}\label{sec_simu_algoeff}

In this subsection, we show the usefulness of our proposed Algorithm \ref{alg:choice} and compare it with some existing methods in the literature. We consider two manifolds of different dimensions, and compare
the following four setups. (1) The clean affinity matrix $\Wb_1$ with the bandwidth $h=\lambda+p$; (2) $\Wb$  constructed using our Algorithm \ref{alg:choice}; (3) $\Wb$  constructed using the bandwidth via (\ref{eq_bandwithchoice}) with $\omega=0.5,$ which has been used in \cite{shnitzer2019recovering,lin2021wave}; (4)  $\Wb$ using the bandwidth $h=p$ as in \cite{elkaroui2010,DW1,do2013spectrum,Liao,el2015graph,el2016graph}. 
{Since the accuracy of the eigenvalue has been discussed in Section \ref{sec_simu_accuracy}, in what follows,} while we do not explore eigenvectors of GL in this paper, we demonstrate the eigenvectors of $\Wb$  and compare them with those of $\Wb_1$ constructed from the clean signal $\{\xi_i\}_{i=1}^n$ to further understand the impact of bandwidth selection.

We start with an 1-dimensional smooth and closed manifold $M_1$, which is parametrized by $\Phi:\,u\mapsto aR[2\cos(u),\,3(1-0.8e^{-8\cos(u)^2})\cos(2\pi(u/(2\pi))^2),\, \\(1-0.8e^{-8\cos(u)^2})\sin(u),\,0,\cdots, 0]^\top \in \mathbb{R}^p$, where $R\in O(p)$, $a>0$ and $u\in (0,2\pi]$. In other words, the 1-dimensional manifold $M_1$ is embedded in a 3-dim Euclidean space in $\mathbb{R}^p$. Now, sample independently  and uniformly $n$ points from $\mathtt{Uniform}(0,2\pi)$, $u_1,\ldots,u_n$, and hence $n$ points $\xi_i=\Phi(u_i)$. 
Next, we add Gaussian white noise to $\xi_i$ via {
$\mathbf{x}_i=\xi_i+\mathbf{y}_i \in \mathbb{R}^p,$ where $\mathbf{y}_i \sim \mathcal{N}(\mathbf{0}, \mathbf{I})$, $i=1,2,\cdots,n,$} are noise independent of $\xi_i.$ 
The results of embedding this manifold by the top three trivial eigenvectors are shown in Figure \ref{fig_1dim_comparisonone}. We could see that with the proposed bandwidth selection algorithm, the embedding of the noisy data is closer to that from the clean data.
To further quantify this closeness, we view the eigenvectors from $\Wb_1$ as the truth, and compare these with the eigenvectors of $\Wb$ with different bandwidths by evaluating the root mean square (RMSE). Note that the freedom of eigenvector sign is handled by taking the smaller value of $\|v^{(c)}_j-v^{(w)}_j\|_2$ and $\|v^{(c)}_j+v^{(w)}_j\|_2$, where $v^{(c)}_j$ if the $j$-th eigenvector of $\Wb_1$ associated with the $j$-th largest eigenvalue and $v^{(w)}_j$ is the $j$-th eigenvector of $\Wb$ associated with the $j$-th largest eigenvalue. We repeat the random sample for 300 times, and plot the errobars with mean$\pm$standard deviation in Figure \ref{fig_1dim_comparisonRMSE}. It is clear that with the adaptive bandwidth selection algorithm, the top several eigenvectors of $\mathbf{W}$ are close to those of $\mathbf{W}_1$.

 \begin{figure*}[t]
\includegraphics[width=1\textwidth]{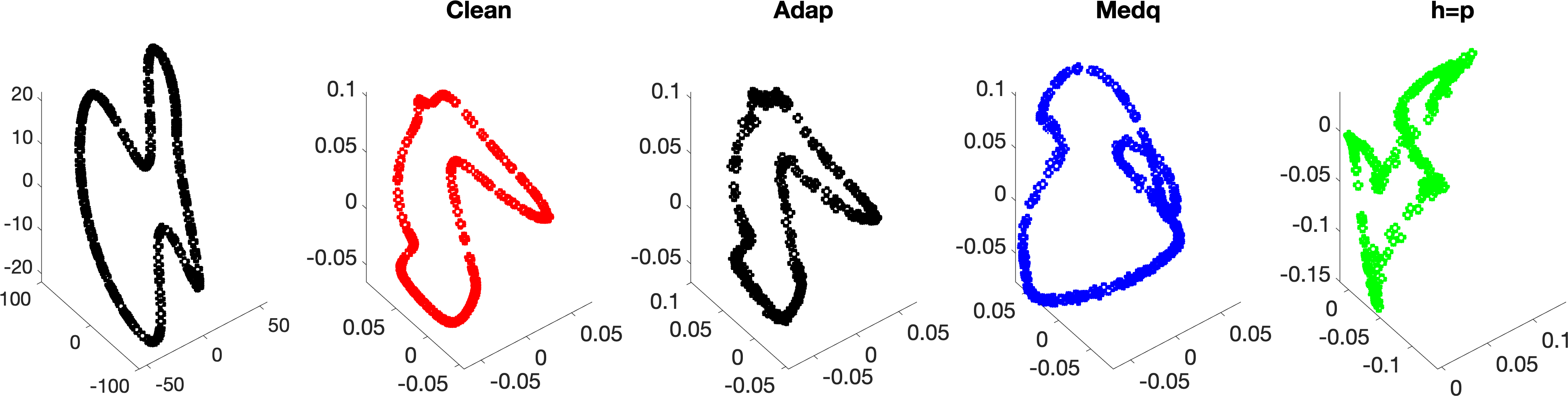}	
	\caption{Comparison of different bandwidth choices for $c_n=0.5.$ We consider the one-dim manifold $M_1$ with $a=20\sqrt{p}$ and take $n=400.$ For the title in each subplot, Clean means the matrix $\Wb_1,$ Adap means using the bandwidth according to our proposed Algorithm \ref{alg:choice}, Medq means using the bandwidth via (\ref{eq_bandwithchoice}) with $\omega=0.5$ and $h=p$ means using the bandwidth $h=p.$} \label{fig_1dim_comparisonone}
\end{figure*}

\begin{figure*}[ht]
\includegraphics[width=0.32\textwidth]{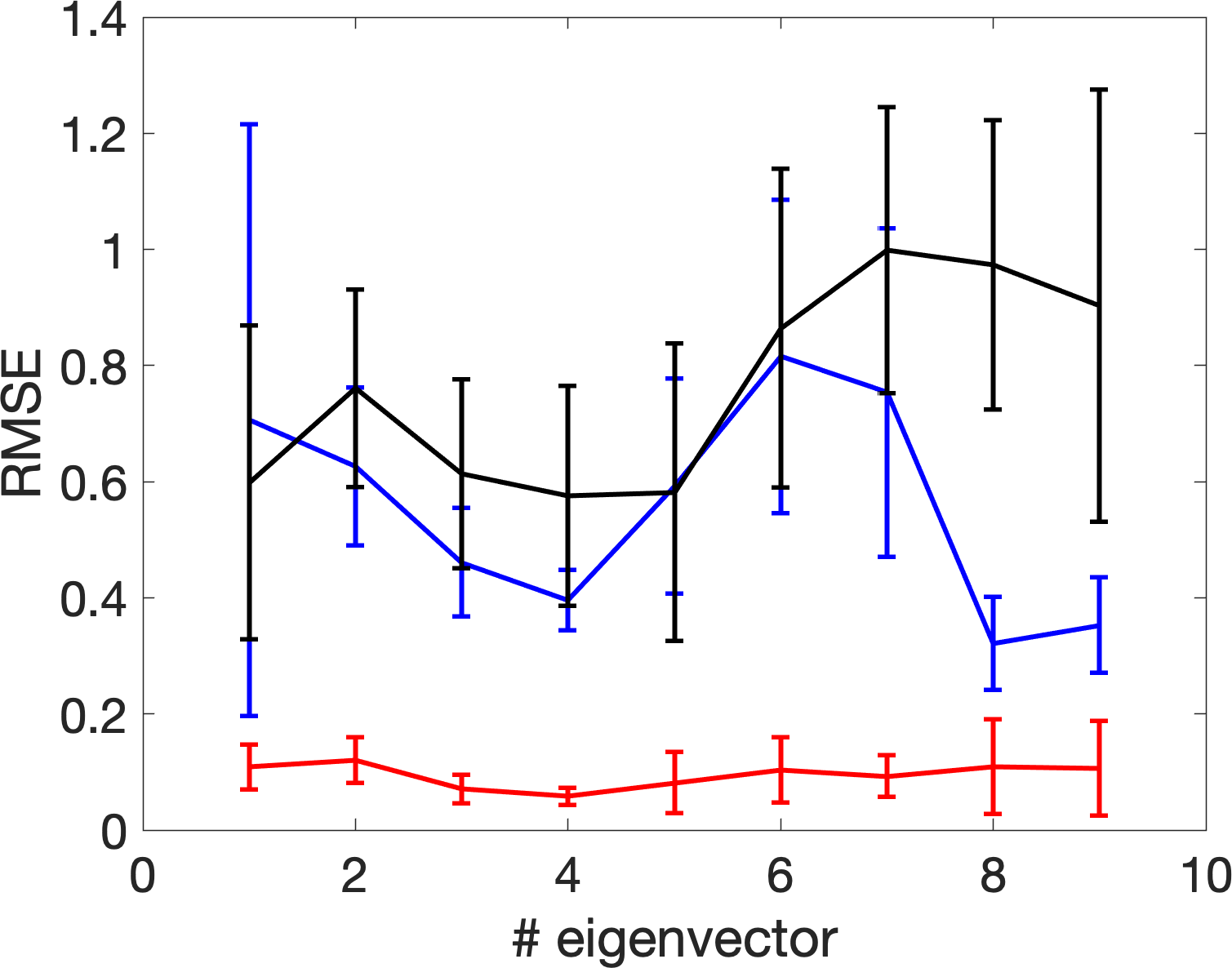}
\includegraphics[width=0.32\textwidth]{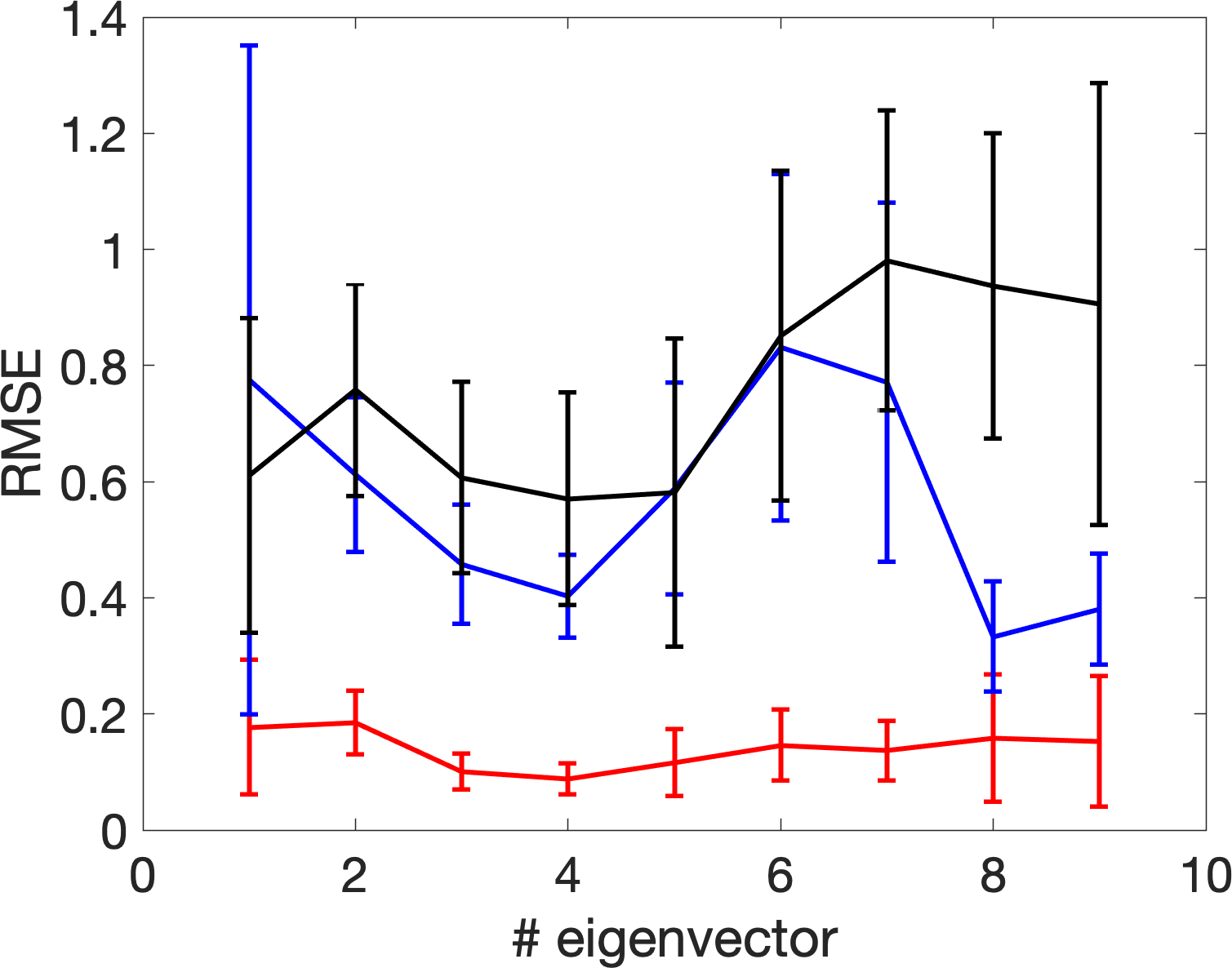}
\includegraphics[width=0.32\textwidth]{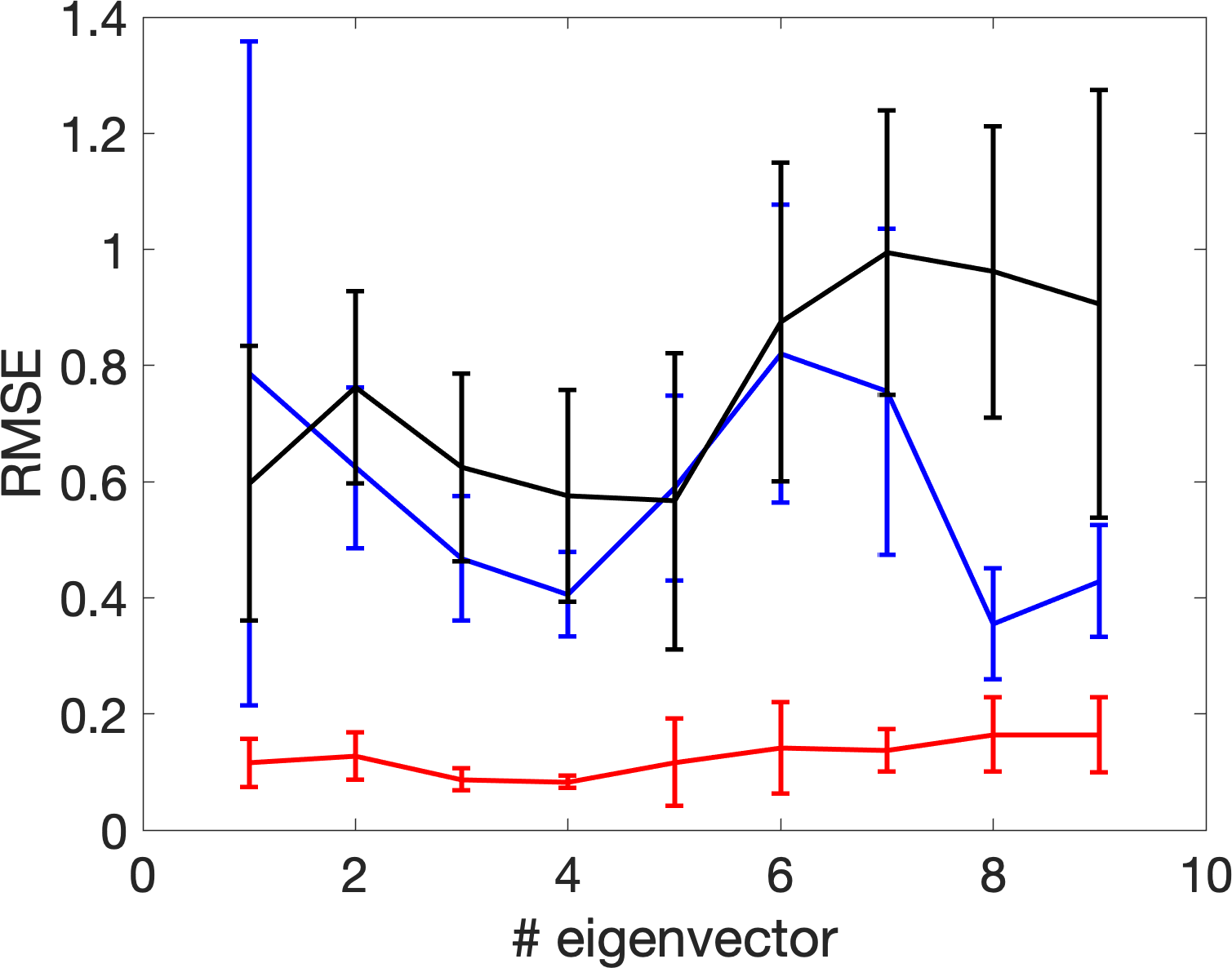}
	\caption{From left to right: comparison of eigenvectors determined by different bandwidth choice schemes for $c_n=0.5,\ 1, 2$. We consider the one-dim manifold $M_1$ with $a=20\sqrt{p}$ and take $n=400.$ The red, blue and black lines indicate the mean$\pm$standard deviation of RMSE when we compare the first nine eigenvectors evaluated from Clean and Adap, Clean and Medq, and Clean and $h=p$, where Clean means the matrix $\Wb_1$ with the bandwidth $p+\lambda$, Adap means using the bandwidth according to our proposed Algorithm \ref{alg:choice}, Medq means using the bandwidth via (\ref{eq_bandwithchoice}) with $\omega=0.5$ and $h=p$ means using the bandwidth $h=p.$} \label{fig_1dim_comparisonRMSE}
\end{figure*}

Next, we consider the Klein bottle, which is a $2$-dim compact and smooth manifold that cannot be embedded into a three dimensional Euclidean space. First, set  
\begin{align}\label{parametrization KB}
\mathbf{z}_i=aR\begin{bmatrix} 
(2\cos(\mathbf u_i(1))+1)\cos(\mathbf u_i(2))\\
(2\cos(\mathbf u_i(1))+1)\sin(\mathbf u_i(2))\\
2\sin(\mathbf u_i(1))\cos(\mathbf u_i(2)/2)\\
2\sin(\mathbf u_i(1))\sin(\mathbf u_i(2)/2)\\
0\\
\vdots\\ 0
\end{bmatrix}\in \mathbb{R}^p, 
\end{align}
where $\mathbf{u}_i$, $i=1,\ldots,n$ are randomly sampled uniformly from $[0, 2\pi]\times [0,2\pi]\subset \mathbb{R}^2$ and $R\in O(p)$ and $a>0$ is the signal strength. In other words, we sample nonuniformly from the Klein bottle isometrically embedded in a $4$-dimension subspace of $\mathbb{R}^p$.
Next, we add noise to $\mathbf{z}_i$ by setting
$\mathbf{x}_i=\mathbf{z}_i+\mathbf{y}_i \in \mathbb{R}^p,$ where $\mathbf{y}_i \sim \mathcal{N}(\mathbf{0}, \mathbf{I})$, $i=1,2,\cdots,n, $ are noise independent of $\mathbf{z}_i.$ The eigenvectors under the mentioned four setups when the signal strength is equivalent to $\alpha=1$ are shown in Figure \ref{fig_KB_comparisonone}, where $c_n=0.5$ and we plot the magnitude of different eigenvector over $\mathbf u_i$ with colors; that is, the color at $\mathbf u_i$ represents the $i$-th entry of the associated eigenvector. We apply the same RMSE evaluation used in the $M_1$ example above; that is, we show the RMSE of the eigenvectors of different $\Wb$ when compared with those from $\Wb_1$ as the truth. We repeat the random sample for 300 times, and plot the errobars with mean$\pm$standard deviation in Figure \ref{fig_KB_comparisonRMSE}. It is clear that with the adaptive bandwidth selection algorithm, the top several eigenvectors of $\mathbf{W}$ are close to those of $\mathbf{W}_1$.

\begin{figure*}[!ht]
\includegraphics[width=1\textwidth]{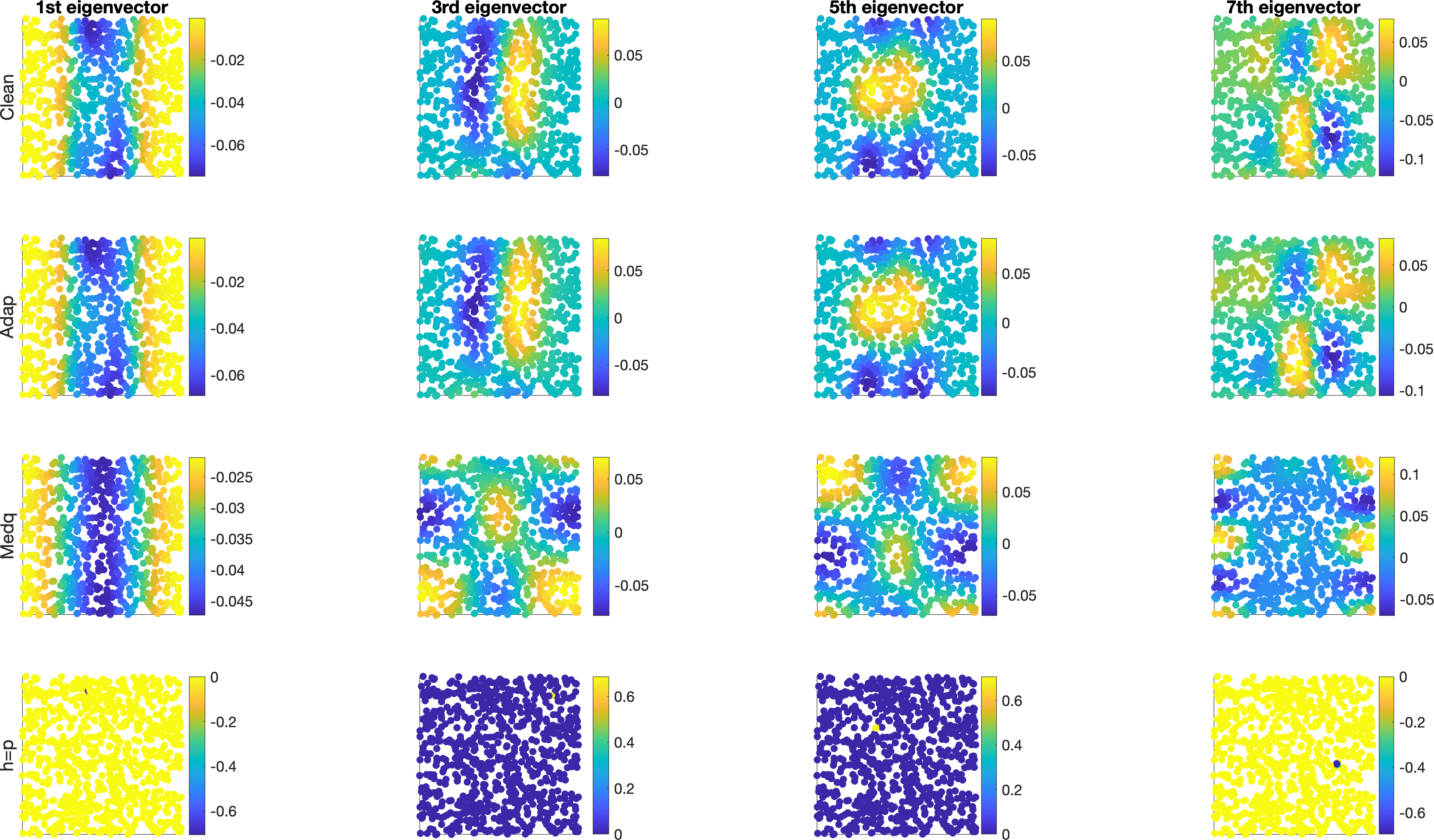}
	\caption{Comparison of different bandwidth choices for $c_n=0.5.$ We consider the Klein bottle parametrized in \eqref{parametrization KB} with $a=20\sqrt{p}.$ Here $n=800.$ For the ylabel on the first column, Clean means the matrix $\Wb_1$ with the bandwidth $p+\lambda$, Adap means using the bandwidth according to our proposed Algorithm \ref{alg:choice}, Medq means using the bandwidth via (\ref{eq_bandwithchoice}) with $\omega=0.5$ and $h=p$ means using the bandwidth $h=p.$} \label{fig_KB_comparisonone}
\end{figure*}

\begin{figure*}[!ht]
\includegraphics[width=0.32\textwidth]{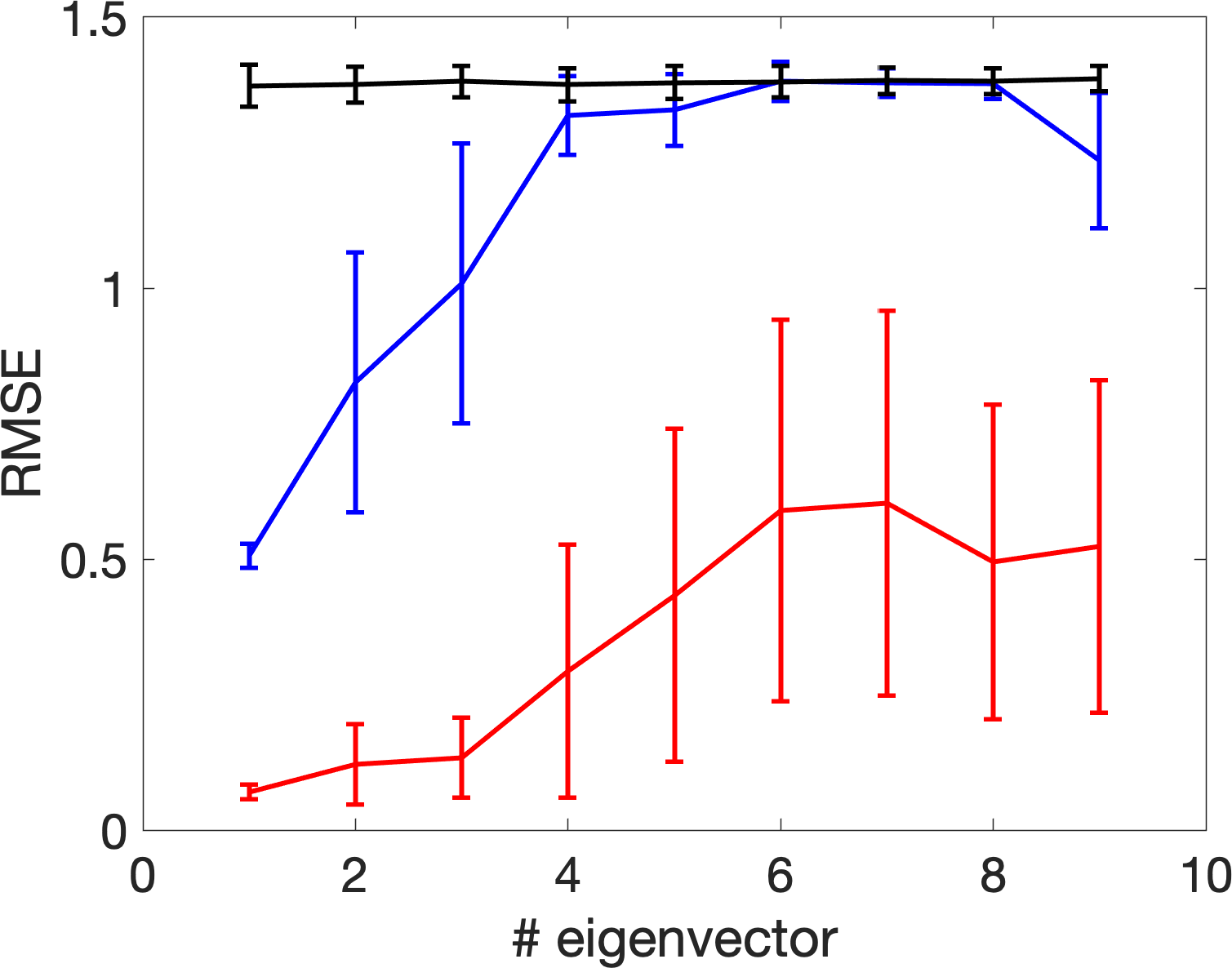}
\includegraphics[width=0.32\textwidth]{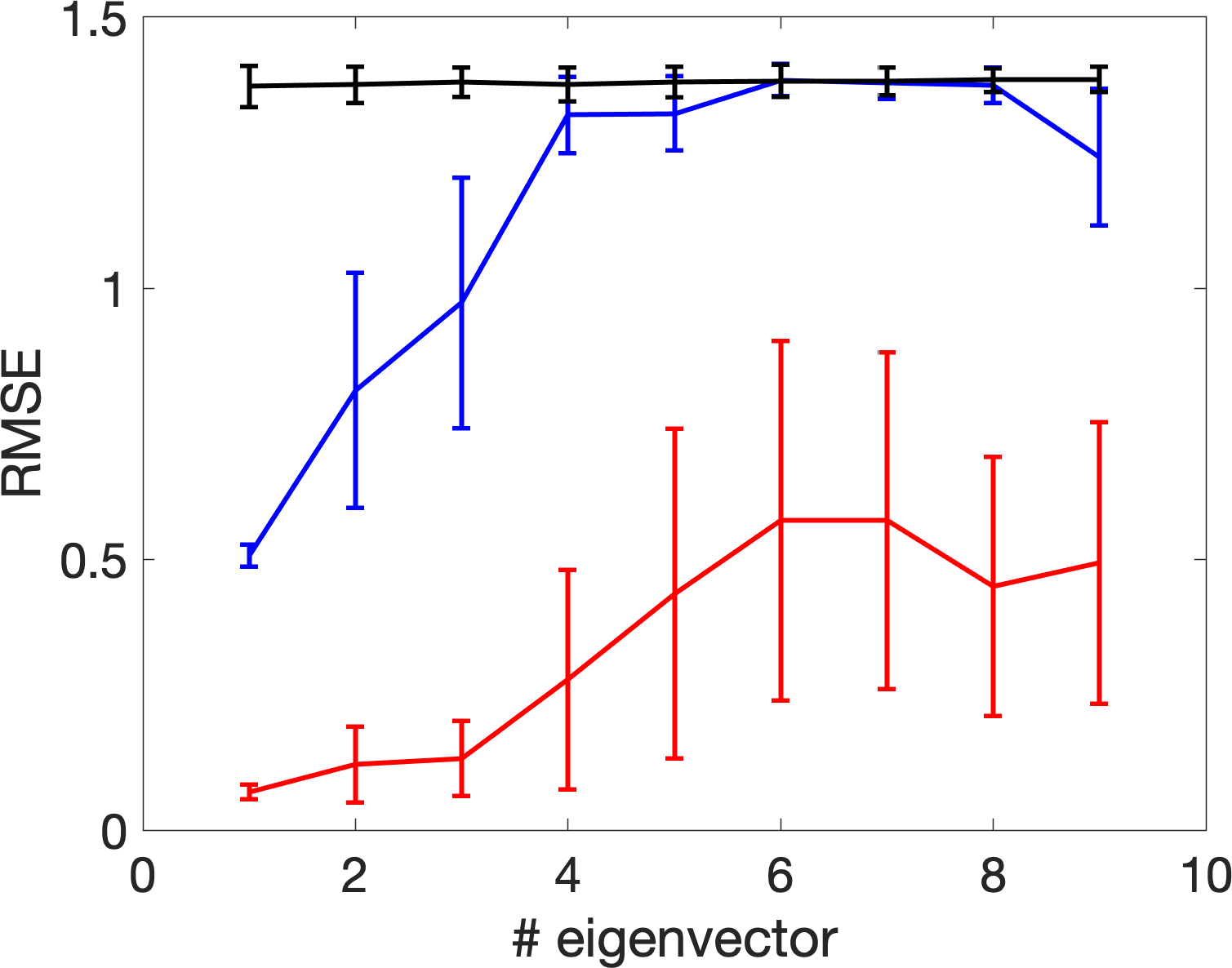}
\includegraphics[width=0.32\textwidth]{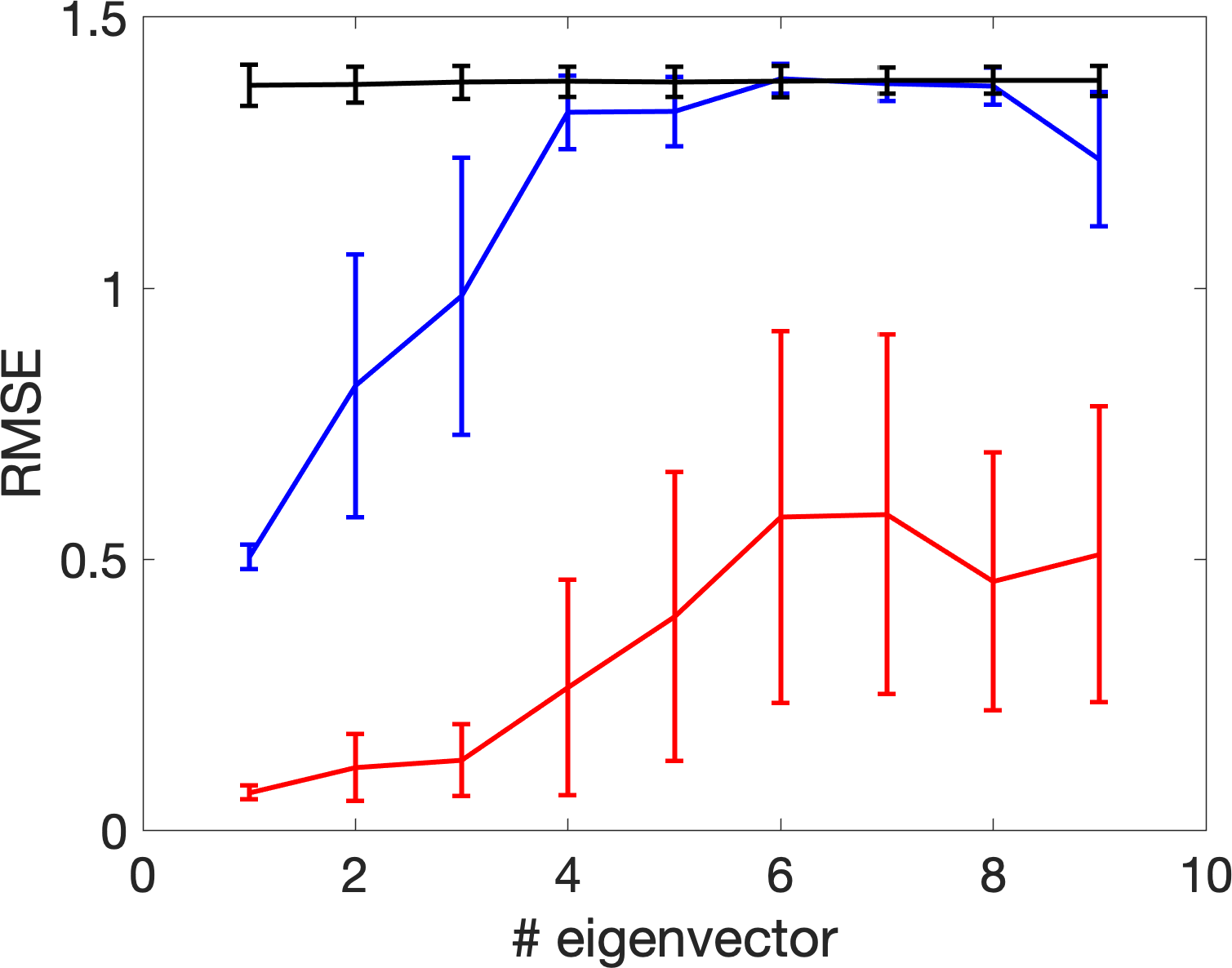}
	\caption{From left to right: Comparison of eigenvectors determined by different bandwidth choice schemes for $c_n=0.5,\ 1, 2$. We consider the Klein bottle parametrized in \eqref{parametrization KB} with $a=20\sqrt{p}.$ Here $n=800.$ The red, blue and black lines indicate the mean$\pm$standard deviation of RMSE when we compare the first nine eigenvectors evaluated from Clean and Adap, Clean and Medq, and Clean and $h=p$, where Clean means the matrix $\Wb_1$ with the bandwidth $p+\lambda$, Adap means using the bandwidth according to our proposed Algorithm \ref{alg:choice}, Medq means using the bandwidth via (\ref{eq_bandwithchoice}) with $\omega=0.5$ and $h=p$ means using the bandwidth $h=p.$} \label{fig_KB_comparisonRMSE}
\end{figure*}

The results of the above two examples support the potential of our proposed bandwidth selection algorithm, particularly when compared with two bandwidth selection methods commonly considered in the literature. Note that since only the eigenvalue information is used in Algorithm \ref{alg:choice}, only partial information in the kernel random matrix is utilized. We hypothesize that by taking the eigenvectors into account, we could achieve a better bandwidth selection algorithm. Since the study of eigenvector is out of the scope of this paper, we will explore this possibility in our future work. 
}

\section{Discussion and Conclusion}\label{sec: discussion and conclusion}
We provide a systematic study of the spectrum of GL under the {nonnull} setup with different SNR regions, particularly when $d=1$, and explore the impact of chosen bandwidths. Specifically, we show that under a proper SNR and bandwidth, the spectrum of $\Ab_1$ from the clean dataset can be extracted from that of $\Ab$ from the noisy dataset. We also provide a new algorithm to select the suitable bandwidth so that the number of ``outliers'' of the spectral bulk associated with the noise is maximized. 

Note that we need the assumption that the entries of $\mathbf{y}_i$ are independent because our arguments depend on established results in the literature, for example, \cite{principal}, which are proved using the assumption that the entries of $\mathbf{y}_i$ are independent. Nevertheless, we believe that this assumption can be removed with extra technical efforts. A natural strategy is to utilize the Gaussian comparison method in the literature of random matrix theory as in \cite{erdos2017dynamical}. This strategy contains two steps. In the first step, we establish the results for Gaussian random vectors where linear independence is equivalent to independence. In the second step, we prove the results hold for sub-Gaussian random vectors using the comparison arguments as in \cite{erdos2017dynamical} under certain moment matching conditions. Since this is not the main focus of the current paper and the extension needs to generalize many existing results in the literature, for example, \cite{principal}, we will pursue this direction in the future.

A comparison with principal component analysis (PCA) deserves a discussion. 
It is well known that when $d=1$ so that the signal is sampled from a linear subspace, we could easily recover this linear subspace from any linear methods like PCA. At the first glance, it seems that the problem is resolved. However, if the purpose is studying the geometric and/or topological structure of the dataset, we may need further analysis. For example, if we want to answer the question like ``is the $1$-dim signal supported on two disjoint subsets (or the $1$-dim linear manifold is disconnected)?'', then PCA cannot answer and we need other tools (for example, GL via the spectral clustering). The same fact holds when $d>1$, where the nonlinear manifold is supported in a $d$-dim subspace while its dimension is strictly smaller than $d$. In this case, while we could possibly recover the $d$-dim subspace by PCA, if we want to further study the nonlinear geometric and/or topological structure of the manifold, we need GL as the analysis tool. For readers who are interested in recovering the manifold structure via GL, see  \cite{gine2006empirical,MR4130541,dunson2019diffusion} and the literature therein.
Based on our results, we know that even for the 1-dimensional linear manifold (for example, $\zb_i$ in (\ref{eq_modelsteptwo})--(\ref{eq_sigdefn}) satisfies that $\eb_1^\top\zb_i$ is uniformly distributed on $[0,1]$), the linear and nonlinear methods are significantly different.

One significant difference that we shall emphasize is the phase transition phenomenon when the nonlinear method like GL is applied. While the discussion could be much more general, the essence is the same, so we continue the discussion with the 1-dimensional linear manifold below. 
For simplicity, we focus on the results in Section \ref{section main result} for the kernel affinity matrix $\Wb$ when the bandwidth is chosen so that $h \asymp p.$ 
From Theorem \ref{thm_affinity matrix} to Theorem \ref{thm_informativeregion}, we observe several phase transitions for eigenvalues depending on the signal strength $\lambda$. In the case when $\lambda$ is bounded, we observe three or four outlying eigenvalues according to the magnitude of $\lambda,$ where three of these outlying eigenvalues are from the kernel effect; see Lemma \ref{lem_karoui} for details. 
When $\lambda$ transits from the subcritical region $\lambda \leq \sqrt{c_n}$ to the supercritical region $\lambda>\sqrt{c_n}$, one extra outlier is generated due to the well-known {\em BBP transition} \cite{baik2005} phenomenon for the spiked covariance matrix model. 
Moreover, the rest of the non-outlying eigenvalues still obey a shifted MP law. As will be clear from the proof, in the bounded region, studying the affinity matrix is directly related to studying PCA of the dataset via the Gram matrix; see (\ref{eq_control}) for details. 
However, once the signal strength diverges, the spectrum of the affinity matrix behaves totally differently. In PCA, under the setting of (\ref{eq_sigdefn}), no matter how large $\lambda$ is, we can only observe a single outlier and its strength increases as $\lambda$ increases. Moreover, only the first eigenvalue is influenced by $\lambda$ and the rest eigenvalues satisfy the MP law, and this MP law is independent of $\lambda.$ Specifically, the second eigenvalue, i.e., the first non-outlying eigenvalue will stick to the right-most edge of the MP law. We refer the readers to Lemma \ref{lem_gramsummary} for more details and Figure \ref{fig_1} for an illustration.
In contrast, regarding the nonlinear kernel method, the values and amount of outlying eigenvalues vary according to the signal strength. That is, the magnitude of $\lambda$ has a possible impact on {\em all} eigenvalues of $\Wb$ through the kernel function. Heuristically, this is because PCA explores the point cloud in a linear fashion so that $\lambda$ will only have an impact on the direction which explains the most variance, i.e., $\lambda_1(\Cb)$, where $\Cb$ is the covariance matrix that is directed related to the Gram matrix $\Gb$. However, the kernel method deals with the data in a nonlinear way. As a consequence, all eigenvalues of $\Wb$ contain the signal information. In other words, unlike the bulk (non-outlier) eigenvalues of $\Gb$, all eigenvalues of $\Wb$ change when $\lambda$ change when $\lambda$ increases. Consider the following three cases.
First, as we will see in the proof below, the eigenvalue corresponding to the BBP transition in the supercritical region will increase when $\lambda$ increases, and this eigenvalue will eventually be close to $1$ when $\lambda$ diverges following (\ref{eq_conditionexponential}). Thus, we expect that the magnitude of this outlying eigenvalue follows an asymmetric bell curve as $\alpha$ increases. 
Second, the eigenvalues corresponding to the kernel effect (see (\ref{eq_sho})) will decrease when $\lambda$ increases since the kernel function $f$ is a decreasing function. 
Third, as we can see from Theorems \ref{lem_affinity_slowly} and \ref{thm_informativeregion}, the non-outlying eigenvalues become outlying eigenvalues when $\lambda$ increases.
In Figure \ref{fig_2}, we illustrate the above phenomena by investigating the behavior of some eigenvalues.

\begin{figure*}[!ht]
\centering
	\includegraphics[width=4cm]{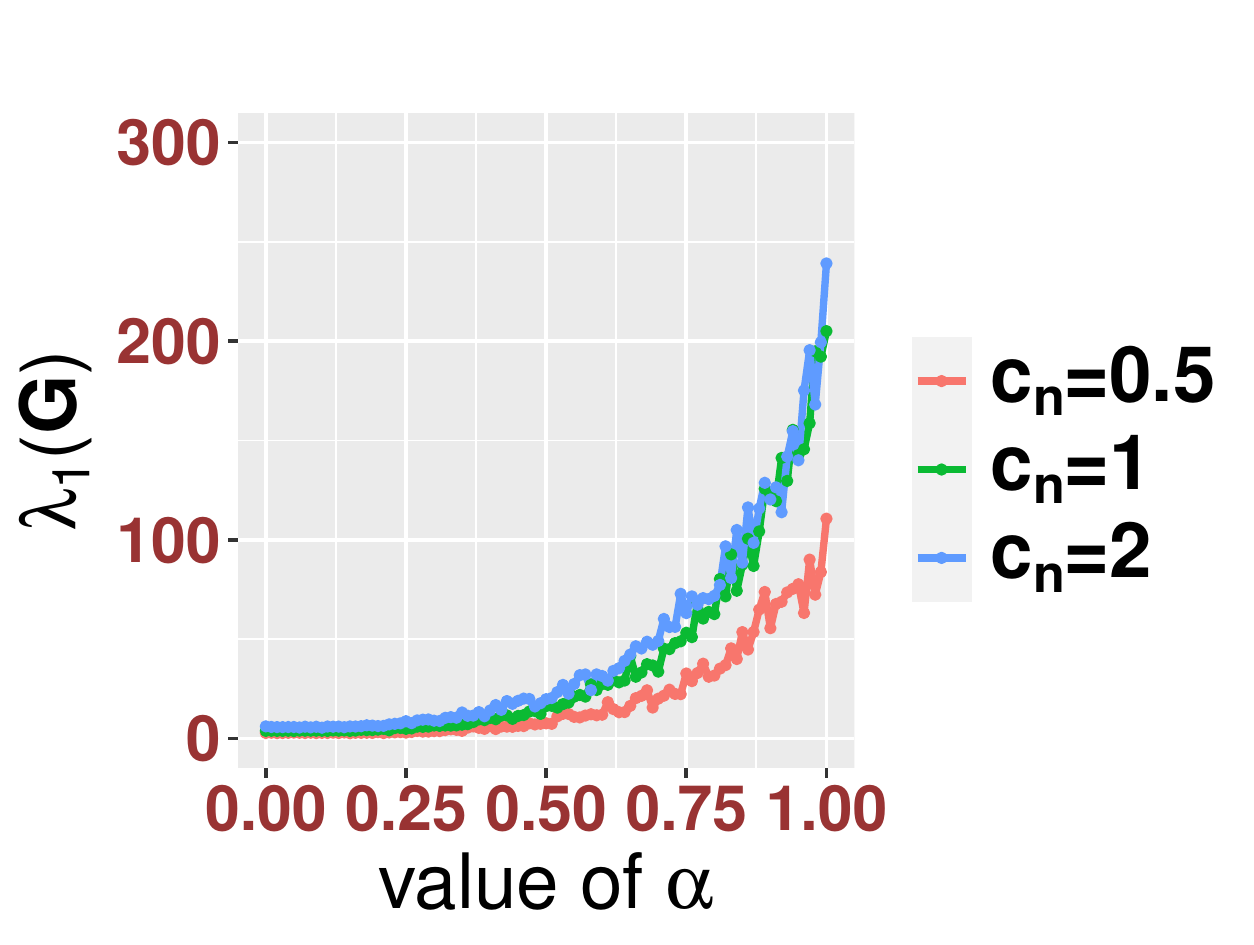}
		\includegraphics[width=4cm]{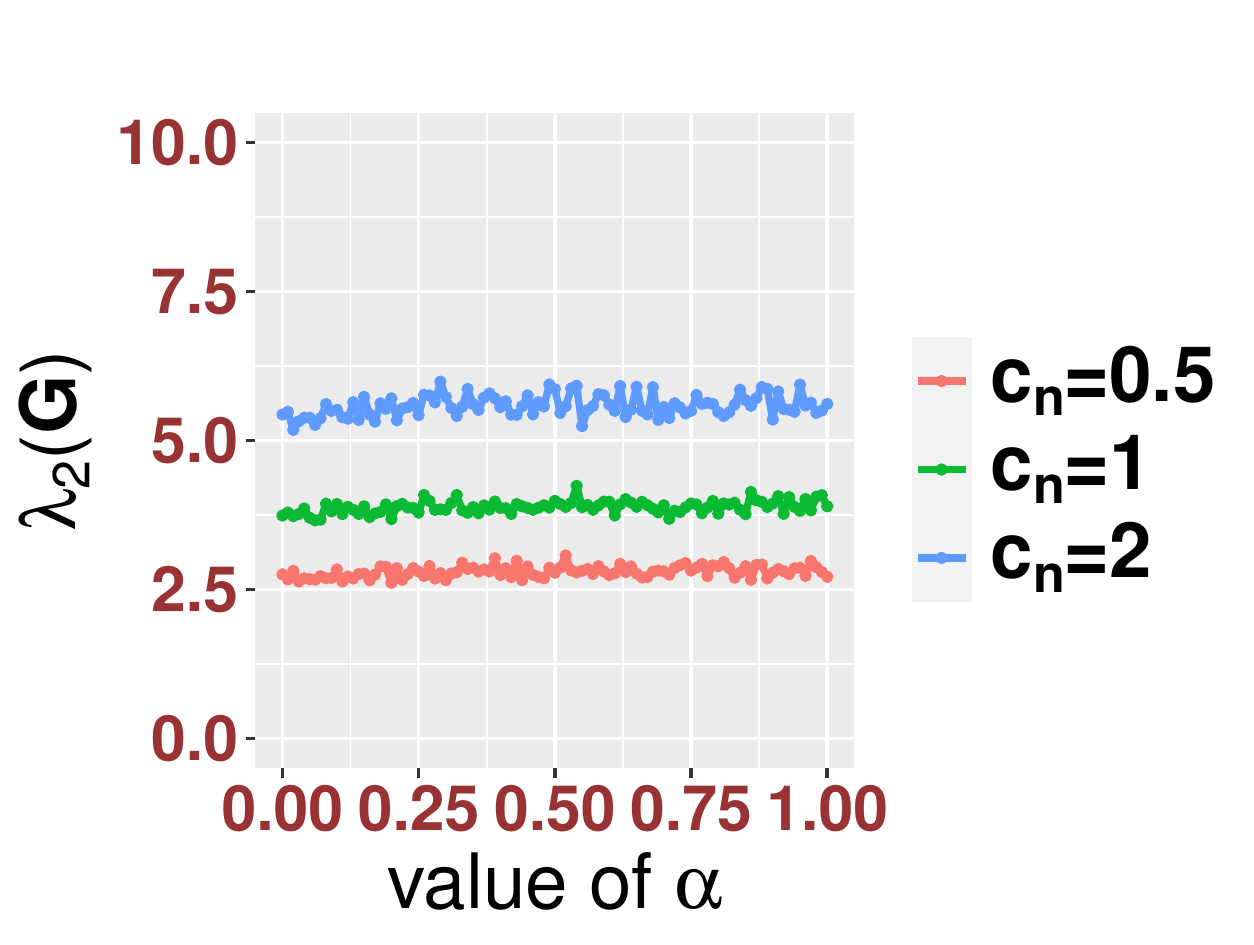}
	\caption{Outlier and the first non-outlying eigenvalues of $\Gb=\frac{1}{p} \Xb^\top \Xb.$  In this simulation,  we use the Gaussian random vectors with covariance matrix (\ref{eq_sigdefn}) with $\lambda=p^{\alpha}$ under the setting $n=200$ and $c_n=n/p.$ The left panel corresponds to the first eigenvalue, i.e., the outlying eigenvalue of $\Gb$, and the right panel corresponds to the second eigenvalue, i.e., the first non-outlying eigenvalue, of $\Gb.$} \label{fig_1}
\end{figure*}

\begin{figure*}[ht]
\centering
	\includegraphics[width=4cm]{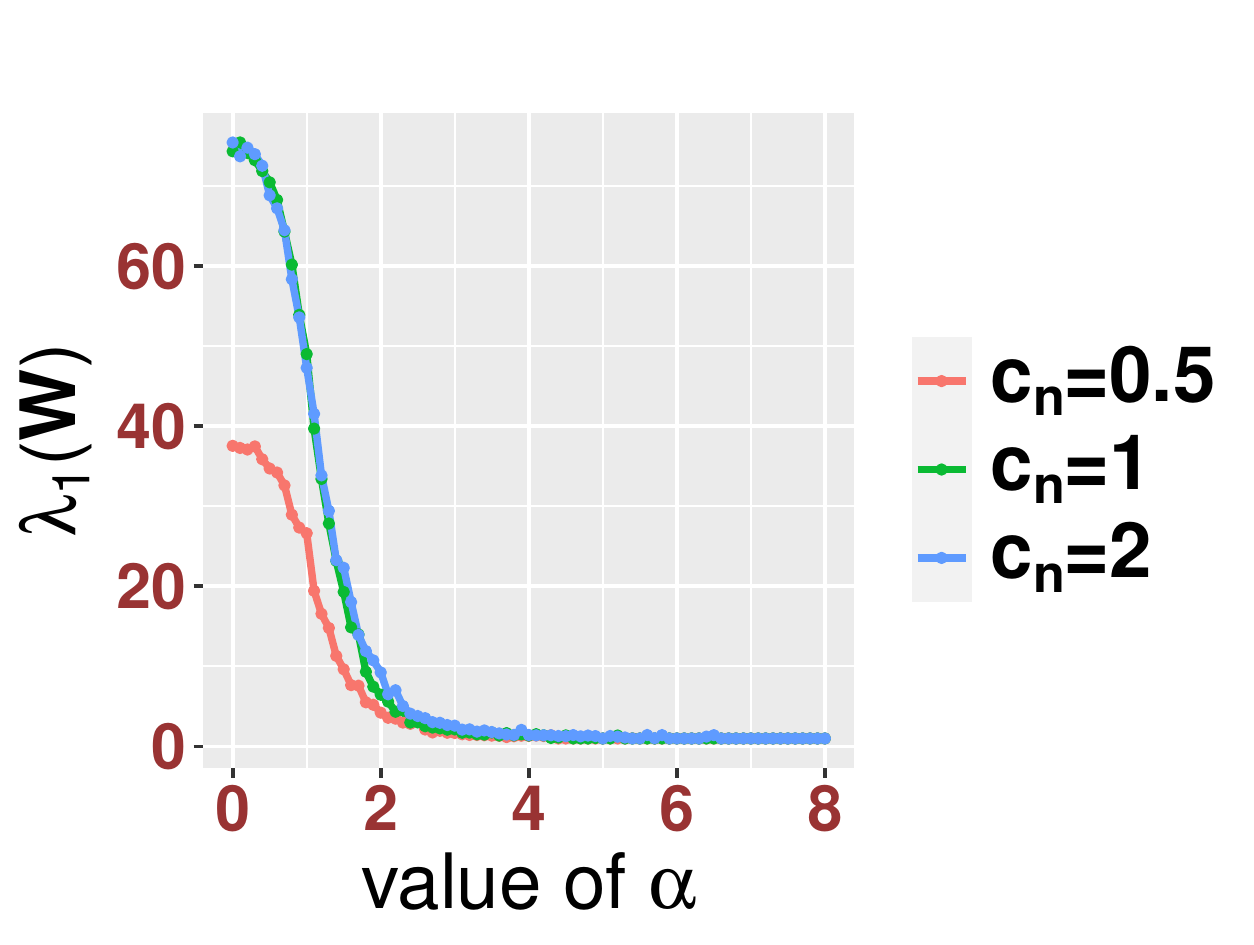}
		\includegraphics[width=4cm]{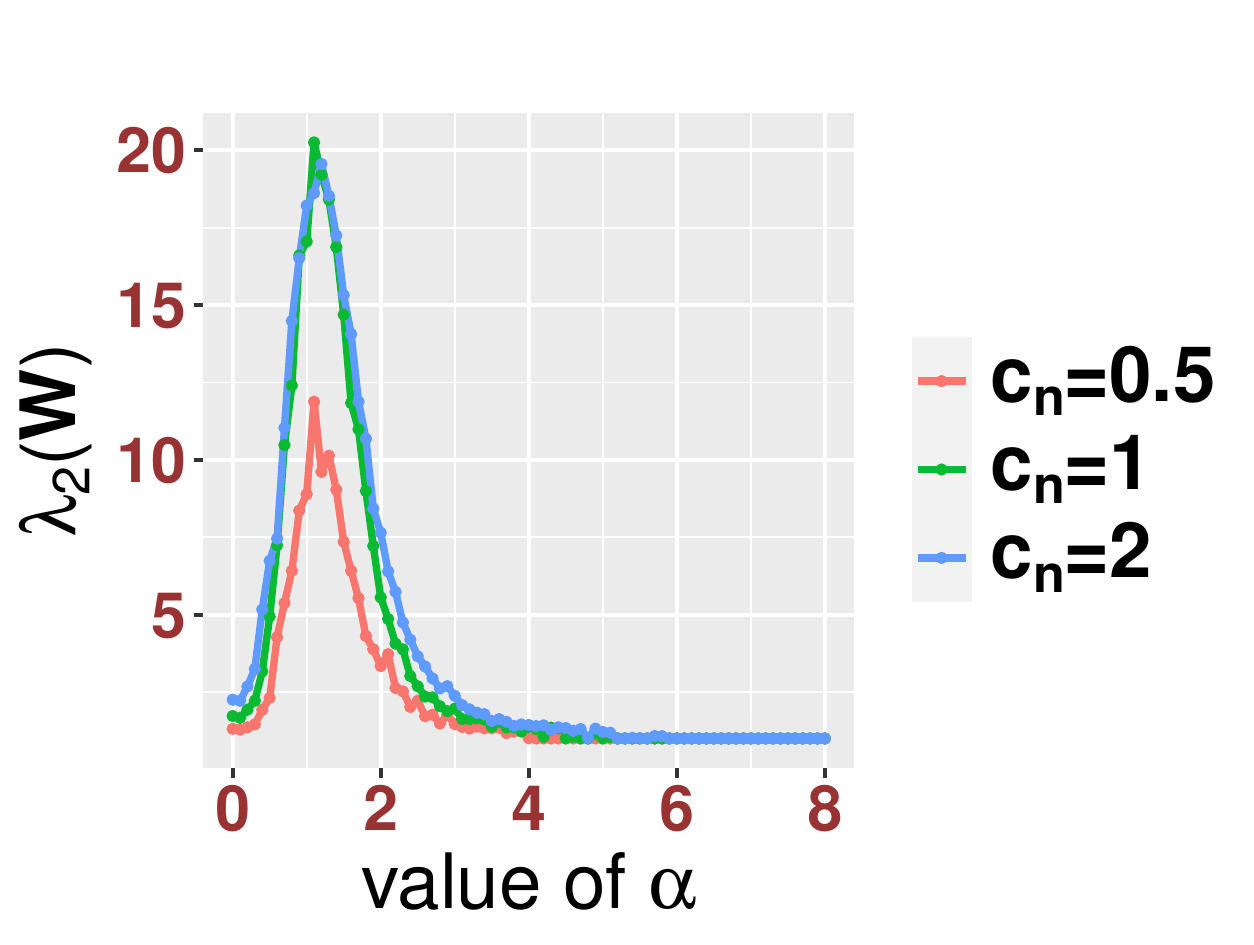}\\
			\includegraphics[width=4cm]{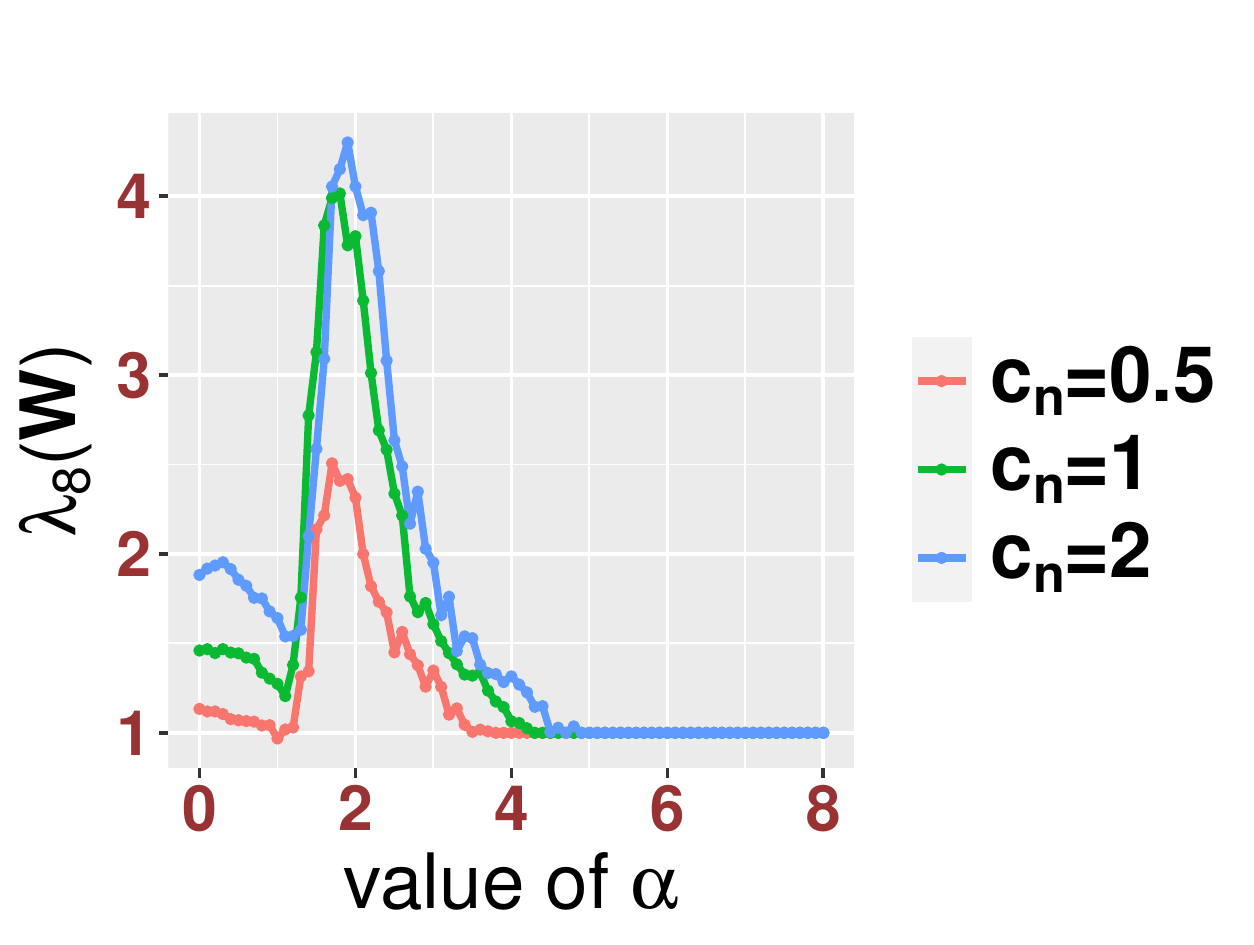}
		\includegraphics[width=4cm]{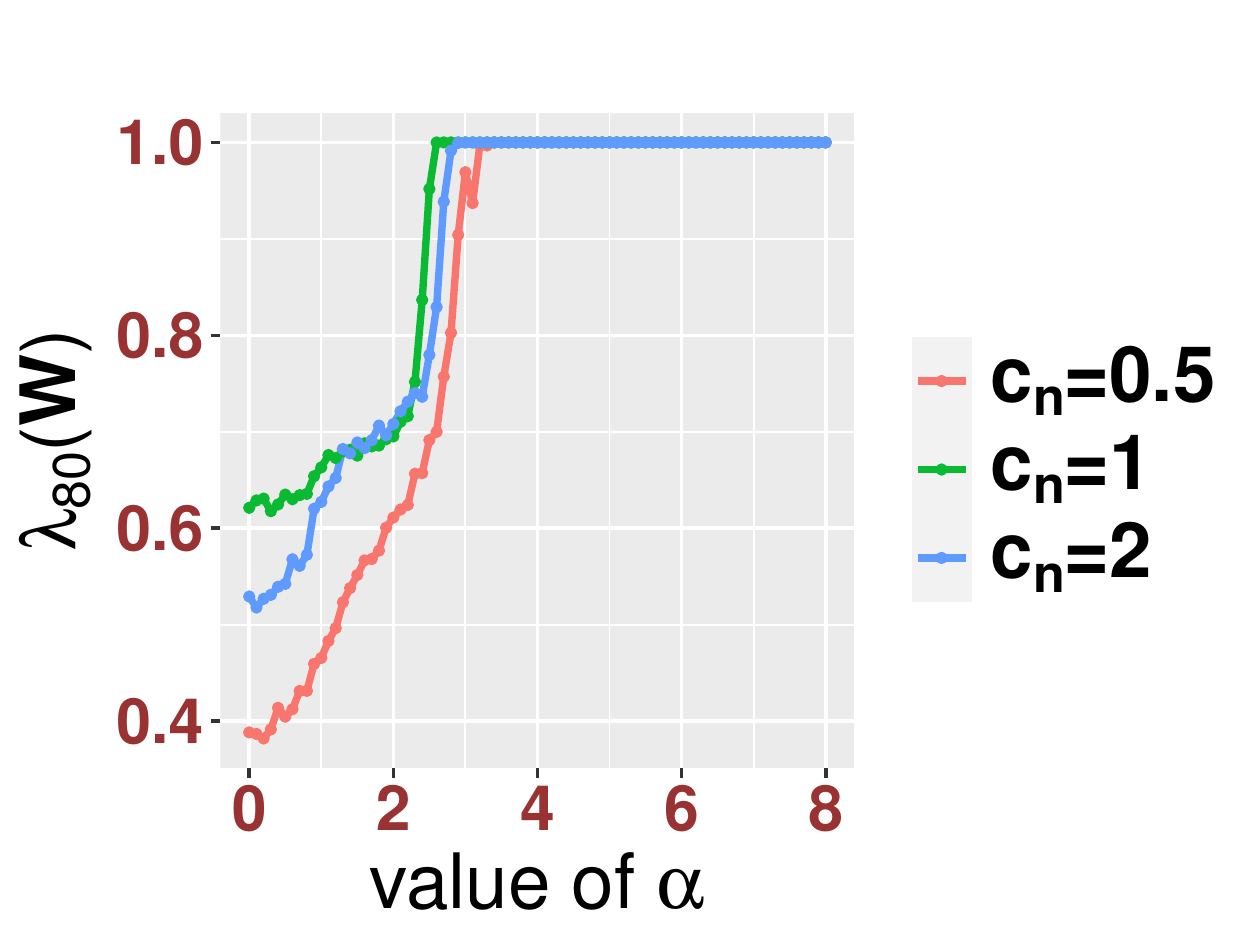}
	\caption{Eigenvalues of $\Wb$ for different $\alpha$ with the kernel $f(x)=\exp(-x/2)$ and bandwidth $h=p$. In this simulation,  we use the Gaussian random vectors with covariance matrix (\ref{eq_sigdefn}) with $\lambda=p^{\alpha}$ with different $c_n=n/p$ and  $n=200$. The top left panel corresponds to the outlying eigenvalue (i.e. $\lambda_1(\Wb)$) associated with $\mathrm{Sh}_0(\tau)$ in (\ref{eq_sho}), the top right panel stands for the outlying eigenvalue (i.e. $\lambda_2(\Wb)$) from the BBP transition effect, the bottom left panel is an eigenvalue (i.e. $\lambda_8(\Wb)$) which gradually detaches from the bulk spectrum and becomes an outlying eigenvalue when $\alpha$ increases. The bottom right panel is an eigenvalue sticking to the bulk spectrum  (i.e. $\lambda_{80}(\Wb)$).  In all cases, we see that the eigenvalues change when $\alpha$ increases.} \label{fig_2}
\end{figure*}

While our results pave the way towards a foundation for statistical inference of various kernel-based unsupervised learning algorithms for various data analysis purposes, like visualization, dimensional reduction, spectral clustering, etc, there are various problems we need to further study. First, as has been mentioned in Remark \ref{rmk_multipled}, there are some open problems; particularly, when $\alpha_i=1$ for all $i=1,\ldots,d$. Second, the behavior of eigenvectors of $\Ab$ from noisy dataset and its relationship with the eigenfunctions of the Laplace-Beltrami operator under the manifold model need to be explored so that the above-mentioned data analysis purposes can be justified under the manifold setup. Third, in practice, noise may have a fat tail, the kernel may not be Gaussian \cite{elkaroui20102,el2016graph} or the kernel might not be isotropic \cite{Wu_Wu:2017}, and the kernel function might be group-valued \cite{Singer_Wu:2012}. Fourth, the bandwidth selection algorithm is established under very nice assumptions, and its performance for real world databases needs further exploration. We will explore these problems in our future work.

\section*{Acknowledgment}
The authors would like to thank the associate editor and two anonymous reviewers for many insightful comments and
suggestions, which have resulted in a significant improvement of the paper.

\appendix

\section{Preliminary results}\label{section preliminary results}
In this section, we collect and prove some useful preliminary results, which will be used later in the technical proof.

\subsection{From manifold to spiked model}\label{sec_reducedproblem} 

In this subsection, we detail the claim in Section \ref{section:intro manifold model} and explain why the manifold model and (\ref{eq_modelsteptwo}) overlaps.  
Suppose $\{\zb^0_i\}$ are i.i.d. sampled from a $p$-dim random vector $Z$, and suppose the range of $Z$ is supported on a $m$-dimensional, connected, smooth and compact manifold $M$, where $m\geq 1$, isometrically embedded in $\mathbb{R}^p$ via $\iota$. Assume there exists $d\geq m$, where $d$ is independent of $n$ and $d\leq p$, so that the embedded $M$ is supported in a $d$-dim affine subspace of $\mathbb{R}^p$. Here we assume that $d$ is fixed. 
Since we consider the kernel distance matrix as in (\ref{eq_defnw}), without loss of generality, we can assume that $\mathbb{E}\zb^0_i=0$; that is, the embedded manifold is centered at $0$. Also, assume the density function on the manifold associated with the sampling scheme is smooth with a positive lower bound. See \cite{cheng2013local} for more detailed discussion of the manifold model and the relevant notion of density function. Suppose the noisy data is 
\begin{equation}\label{eq_reducedmodel0}
\xb_i=\sqrt{\lambda}\zb^0_i+\yb_i\,, 
\end{equation}
where $\lambda>0$ represents the signal strength, and $\yb_i$ is the independent noise that satisfies \eqref{eq_yassum1}. Denote $\zb_i:=\sqrt{\lambda}\zb^0_i$.
Thus, there exists a $p \times p$ orthogonal matrix $R_p$ such that 
\[
R_p \zb^0_i=(z^0_{i1}, z^0_{i2}, \cdots, z^0_{id},0, \cdots, 0)^\top\,.
\]
Since $M$ is compact, $z^0_{ij}$, $j=1,\ldots,d$, is bounded and hence sub-Gaussian with the variance controlled by $\mathcal K:=\max_{x,y\in M}\|\iota(x)-\iota(y)\|_{\mathbb{R}^p}$, which is independent of $p$ since the manifold is assumed to be fixed. On the other hand, since $M$ is connected, $z^0_{ij}$, $j=1,\ldots,d$, is a continuous random variable.
We can further choose another rotation $\bar{R}_p$ so that the first $d$ coordinates of $\zb^0_i$ are whitened; that is, $\bar{R}_p\zb^0_i$ has the covariance structure 
\[
\texttt{diag}(\lambda^0_1,\lambda^0_2,\ldots,\lambda^0_d,0,\ldots,0)\in \mathbb{R}^{p\times p}\,, 
\]
where by the assumption of the density function, we have $\lambda^0_i> 0$ for $i=1,\ldots, d$. Note that $\lambda^0_i$ is controlled by $\mathcal K$, and by the smoothness of the manifold, we could assume without loss of generality that $c\leq \lambda^0_i\leq 1/c$ for $c\in (0,1)$ for all $i=1,\ldots,d$. 
Then, multiply the noisy dataset in \eqref{eq_reducedmodel0} by $\bar R_p$ from the left and get 
\begin{equation}\label{eq_reducedmodel}
\bar{\xb}_i=\bar{\zb}_i+\bar{\yb}_i\,,
\end{equation}
where $\bar{\xb}_i:=\bar R_p \xb_i$, $\bar{\zb}_i:=\bar R_p \zb_i$ and $\bar{\yb}_i:=\bar R_p \yb_i$.
It is easy to see that since $\{\yb_i\}$ are isotropic (c.f. \eqref{eq_yassum1}), $\{\bar{\yb}_i\}$ are sub-Gaussian random vectors satisfying (\ref{eq_yassum1}). 
Thus, since $\bar{\zb}_i$ and $\bar{\yb}_i$ are still independent, the covariance of $\{\bar{\xb}_i\}$ becomes 
\[
\bar{\Sigma}_p:=\texttt{diag}(\lambda_1+1, \cdots, \lambda_d+1, 1, \cdots, 1)\in \mathbb{R}^{p\times p}\,,
\]
where $\lambda_l:=\lambda\lambda^0_l$ for $l=1,\ldots,d$. Note that in this case, $\lambda_i$ are of the same order.
By the above definitions and the invariance of the $\ell_2$ norm, we have
\begin{align*}
\| \bar{\xb}_i-\bar{\xb}_j  \| \,& =\| \xb_i-\xb_j \|\,, \\
\| \bar{\zb}_i-\bar{\zb}_j  \| \,& =\| \zb_i-\zb_j \|\,, \\
\| \bar{\yb}_i-\bar{\yb}_j  \| \,& =\| \yb_i-\yb_j \|\,,
\end{align*}
which means that the affinity matrices (\ref{eq_defnw}) and transition matrices (\ref{eq_singlematrix}) remain unchanged after applying the orthogonal matrix.
Thus, \eqref{eq_reducedmodel} is reduced to \eqref{eq_modelsteptwo} and it suffices to focus on model (\ref{eq_modelsteptwo}).

Under the above setting that a low dimensional manifold is embedded into an affine subspace with a  fixed dimension $d$, the nonlinear manifold model is thus closely related to the spiked covariance matrix model. Recall that according to Nash's isometric embedding theorem \cite{nash1956imbedding}, there exists an embedding so that $d$ is smaller than $m(3m+11)/2$, but there may exist embeddings so that the $d$ is higher than $m(3m+11)/2$. More complicated models might need $d$ to even diverge as $n\to \infty$. In these settings, the nonlinear manifold model will be reduced to other random matrix models, i.e., the divergent spiked model \cite{cai2020limiting} or divergent rank signal plus noise model \cite{dozier2007analysis,dozier2007empirical,ding2022edge,ding2022tracy}. We believe that the spectrum of the GL under these settings can also be investigated once the spectrum of these random matrix models can be well studied. Since this is not the focus of the current paper, we will pursue this direction in the future.

\subsection{Some linear algebra facts}

We record some linear algebraic results. The first one is for the Hadamard product from \cite{elkaroui2010}, Lemma A.5. 
\begin{lemma}\label{lem_hardamardproductbound}
Suppose $\Mb$ is a real symmetric matrix with nonnegative entries and $\Eb$ is a real symmetric matrix. Then we have that
\begin{equation*}
\sigma_1(\Mb \circ \Eb) \leq \max_{i,j}|\Eb(i,j)| \sigma_1(\Mb),
\end{equation*}
where $\sigma_1(\Mb)$ stands for the largest singular value of $\Mb.$ 
\end{lemma}

The following lemma is commonly referred to as the {\em Gershgorin circle theorem}, and its proof can be found in \cite[Section 6.1]{2012matrix}. 

\begin{lemma}\label{lem_circletheorem}

Let  $A=(a_{ij})$ be a real $ n\times n$ matrix. For  $1 \leq i \leq n,$ let  $R_{i}=\sum _{{j\neq {i}}}\left|a_{{ij}}\right| $ be the sum of the absolute values of the non-diagonal entries in the  $i$-th row. Let  $ D(a_{ii},R_{i})\subseteq \mathbb {R} $ be a closed disc with center $a_{ii}$ and radius  $R_{i}$ referred as the \emph{Gershgorin disc}. Every eigenvalue of $ A=(a_{ij})$ lies within at least one of the Gershgorin discs  $D(a_{ii},R_{i})$, where $R_i=\sum_{j\ne i}|a_{ij}|$.
\end{lemma}

We also collect some important matrix inequalities. For details, we refer readers to \cite[Lemma SI.1.9]{DW1}
\begin{lemma}\label{lem_collectioninequality} For two $n \times n$ symmetric matrices  $\Ab$ and $\Bb,$ we have that
\begin{equation*}
\sum_{i=1}^n |\lambda_i(\Ab)-\lambda_i(\Bb)|^2 \leq \operatorname{tr} \{(\Ab-\Bb)^2\}. 
\end{equation*}
Moreover, let $m_{\Ab}(z)$ and $m_{\Bb}(z)$ be the Stieltjes transforms of the ESDs of $\Ab$ and $\Bb$ respectively, then we have 
\begin{equation*}
|m_{\Ab}(z)-m_{\Bb}(z)| \leq \frac{\operatorname{rank}\{\Ab-\Bb\}}{n} \min \left\{ \frac{2}{\eta}, \frac{\norm{\Ab-\Bb}}{\eta^2} \right\}.
\end{equation*}
\end{lemma}

\subsection{Some concentration inequalities for sub-Gaussian random vectors}\label{sec_auxi}
We record some auxiliary lemmas for our technical proof. We start with the concentration inequalities for the sub-Gaussian random vector $\yb$ that satisfies 
\begin{equation*}
\mathbb{E}(\exp( \ab^\top \yb)) \leq \exp(\| \ab \|_2^2/2).
\end{equation*}
The first lemma establishes the concentration inequalities when $\lambda$ is bounded.
\begin{lemma}\label{lem_concentrationinequality}
Suppose (\ref{eq_modelsteptwo})-(\ref{eq_sigdefn}) hold. Assume $d\geq 1$ is fixed, $\lambda_l\asymp 1$ for $l=1,\ldots,d$ and write
\begin{equation}\label{notationalconvention}
\zb_i=[\sqrt{\lambda_1} z_{i1}, \ \sqrt{\lambda_2} z_{i2}, \ldots, \sqrt{\lambda_d} z_{id}, \ 0\ldots 0]
\end{equation}
for all $1 \leq i \leq n$, where $\var(z_{il})=1$ for all $l=1,\ldots,d$.
Then, for $i \neq j$ and $t>0,$ we have  
\begin{equation}\label{eq_offcontrol}
\mathbb{P}\left( \left|\frac{1}{p}\yb_i^\top \yb_j\right|> t \right) \leq \exp(-p t^2 /2)\,, 
\end{equation}
as well as 
\begin{equation}\label{eq_offcontrolx}
\mathbb{P}\left( \left|\frac{1}{p}\xb_i^\top \xb_j-\frac{1}{p} \zb_i^\top  \zb_j\right|> t \right) \leq \exp(-p t^2 /2)\,. 
\end{equation} 
For the diagonal terms, for $t>0,$ we have for some universal constants $C, C_1>0,$
\begin{align}\label{eq_diagonal}
 \mathbb{P} & \left( \left| \frac{1}{p} \| \yb_i \|_2^2-1 \right| >t\right) \leq 
\begin{cases}
2 \exp(-C_1 p t^2), & 0<t \leq C \\
2 \exp(-C_1 pt), & t >C\,,
\end{cases}
\end{align}
as well as 
\begin{align}\label{eq_diagona2}
\mathbb{P} & \left( \left| \frac{1}{p} \| \xb_i \|_2^2-\frac{1}{p} \| \zb_i \|_2^2 \right|>t \right) \leq 
\begin{cases}
2 \exp(-C_1 p t^2), & 0<t \leq C \\
2 \exp(-C_1 pt), & t >C. 
\end{cases}
\end{align} 
Especially, the above results imply that 
\begin{align}\label{eq_finalcontrolimplication1}
&\frac{1}{p}\left| \yb_i^\top \yb_j \right| \prec n^{-1/2}, \nonumber \\
& \frac{1}{p}\left| \xb_i^\top \xb_j\right| \prec n^{-1/2}\,, 
\end{align}
as well as 
\begin{align}\label{eq_finalcontrolimplication2}
\left| \frac{1}{p} \| \yb_i \|_2^2-1 \right| \,&\prec n^{-1/2} , \nonumber \\
 \left| \frac{1}{p} \| \xb_i \|_2^2- \left(1+\frac{\sum_{l=1}^d\lambda_l}{p}\right)\right| \,&\prec n^{-1/2}. 
\end{align} 
Note that since $p$ and $n$ are of the same order, the above results hold when $\yb_i$ is replaced by the vector $\bm{z}:=[z_1,\ldots,z_n]^\top\in \mathbb{R}^n$, which is a sub-Gaussian random vector.
\end{lemma}

\begin{proof}
We adapt (\ref{notationalconvention}) in the proof. When $i \neq j,$ (\ref{eq_offcontrol}) has been proved in \cite[Lemma A.2]{DW1}.  Observe that
\begin{equation}\label{eq_differenceexpansion}
 \xb_i^\top\xb_j-\zb_i^\top \zb_j=\yb_{i}^\top \yb_{j}+\sum_{l=1}^d\sqrt{\lambda_l} (z_{il}y_{jl}+z_{jl}y_{il}).
\end{equation}
Since $\lambda\asymp 1$ and $d$ is fixed, we find that $\yb_{i}^\top \yb_{j}$ is the leading order term. (\ref{eq_offcontrolx}) follows from (\ref{eq_offcontrol}) and (\ref{eq_differenceexpansion}).
When $i=j$,  (\ref{eq_diagonal}) has been proved in \cite[Corollary 2.8.3]{vershynin2018high}. (\ref{eq_diagona2}) follows from (\ref{eq_differenceexpansion}) and (\ref{eq_diagonal}).

(\ref{eq_finalcontrolimplication1}) ((\ref{eq_finalcontrolimplication2}) respectively) follows from (\ref{eq_offcontrol}) and (\ref{eq_offcontrolx}) ((\ref{eq_diagonal}) and (\ref{eq_diagona2}) respectively) for scalar random variables and the fact that $\lambda\asymp 1.$
\end{proof}

Then we provide the concentration inequalities when $\lambda$ is in the slowly divergent region. Indeed, in this region, the results of the diagonal parts of Lemma \ref{lem_concentrationinequality} still apply. 

\begin{lemma}\label{lem_concentrationslowlydivergent} 
Suppose (\ref{eq_modelsteptwo})-(\ref{eq_sigdefn}) 
hold. Assume $d \geq 1$ is fixed, $\lambda_l=n^{\alpha_l}$ with $0<\alpha_l<1$ for all $1 \leq l \leq d$ and recall (\ref{notationalconvention}).
Then when $i \neq j$, we have  
\begin{align}\label{eq_finalcontrolimplicationdivergent1}
\frac{1}{p}\left| \xb_i^\top \xb_j\right| &\,\prec \frac{\sum_{l=1}^d \lambda_l }{n}+\frac{1}{\sqrt{n}},  \nonumber \\
\left| \frac{1}{p} \| \xb_i \|_2^2- \left(1+\frac{\sum_{l=1}^d \lambda_l}{p}\right)\right| &\, \prec \frac{\sum_{l=1}^d \lambda_l}{n}+\frac{1}{\sqrt{n}}.  
\end{align} 
\end{lemma}
\begin{proof}
We only discuss the second term and the first item can be dealt with in a similar way. We define $\lambda_{fl}=\lfloor \lambda_l \rfloor$  as the floor of $\lambda$. We again adapt (\ref{notationalconvention}) in the proof. Note that using (\ref{eq_differenceexpansion}) we have
\begin{align}\label{eq_boundlambdadigerventone}
\sum_{l=1}^d \frac{\lambda_{fl}}{p} z_{il}^2 & +\frac{1}{p}\sum_{j=1}^p \yb_{ij}^2  \leq \frac{1}{p}\| \xb_i \|_2^2-\sum_{l=1}^d\frac{2\sqrt{\lambda_l}}{p} z_{il} \yb_{il}  \nonumber \\
& \leq \sum_{l=1}^d \frac{\lambda_{fl}+1}{p}z_{il}^2 +\frac{1}{p}\sum_{j=1}^p \yb_{ij}^2.
\end{align}
We study the upper bound of the sandwich inequality, and the lower bound follows by the same argument. On one hand, according to (\ref{eq_finalcontrolimplication2}), we have that
\begin{equation*}
\frac{1}{p} \sum_{j=1}^p \yb_{ij}^2=1+O_{\prec}(n^{-1/2}). 
\end{equation*}
Moreover, since $z_i$ is a sub-Gaussian random variable, we can apply (\ref{eq_finalcontrolimplication2}). This yields that
\begin{equation*}
\sum_{i=1}^d \frac{\lambda_{fl}+1}{p}z_{il}^2=\frac{\sum_{l=1}^d (\lambda_{fl}+1)}{p}+O_{\prec}(n^{{-1}+\alpha}),
\end{equation*}
where we used the fact that $\alpha_l<1$ for all $1 \leq l \leq d.$ Similarly, we have 
\begin{equation*}
\left|\frac{\sqrt{\lambda_l}}{p} z_{il} \yb_{il} \right| \prec \frac{\sqrt{\lambda_l}}{p}. 
\end{equation*} 
Using (\ref{eq_boundlambdadigerventone}), we readily obtain that
\begin{equation*}
\frac{1}{p}\| \xb_i \| \leq 1+\frac{\sum_{l=1}^d \lambda_l}{p}+O_\prec\left(\frac{\sum_{l=1}^d \lambda_l}{n}+\frac{1}{\sqrt{n}} \right).
\end{equation*}
\end{proof}

\subsection{Some results for Gram matrices}

Denote the Gram matrix of the point clouds $\{\xb_i\} \subset \mathbb{R}^p$ in the form of (\ref{eq_modelsteptwo}) as
\begin{equation}\label{eq_defngram}
\Gb_x=\frac{1}{p} \Xb^\top \Xb, \ \Xb=[\xb_1 \cdots \xb_n] \in \mathbb{R}^{p \times n}.
\end{equation}
The eigenvalues of $\Gb_x$ have been thoroughly studied in the literature; see \cite{principal,benaych2011eigenvalues,bdww,debashis}, among others. We summarize those results relevant to this paper in the following lemma. 
\begin{lemma}\label{lem_gramsummary} 
Suppose \eqref{eq_modelsteptwo}-\eqref{eq_ratio} hold true and $d\geq 1$ is fixed. Assume that there exists some $0\leq d' \leq d$ so that $\lambda_1 \geq \lambda_2 \geq \cdots \geq  \lambda_{d'}>\sqrt{c_n} \geq \lambda_{d'+1} \geq \cdots \geq  \lambda_d\geq0.$ Then, if $d'>0$, we have that for $1 \leq j \leq d'$, 
\begin{equation*}
|\lambda_j(\Gb_x)-(1+\lambda_j)(1+c_n \lambda_j^{-1})| \prec n^{-1/2} \sqrt{\lambda_j} (\lambda_j-\sqrt{c_n})^{1/2}.
\end{equation*}
If $d'<d$, we have that for $d'+1 \leq j \leq d$,
\begin{equation*}
|\lambda_j(\Gb_x)-\gamma_{\mu_{c_n,1}}(1)| \prec n^{-2/3}.
\end{equation*}
Moreover, for $1 \leq i \leq (1-\epsilon)n$, where $\epsilon>0$ is  a fixed small constant, we have
\begin{equation}\label{eq_rigidity}
\left| \lambda_{i+d}(\Gb_x)-\gamma_{\mu_{c_n,1}}(i) \right| \prec n^{-2/3} i^{-1/3}\,.
\end{equation}
\end{lemma}

\begin{proof}
We mention that the results under our setup have been originally proved in \cite{principal} (see Section 1.2 and Theorems 2.3 and 2.7 therein), and stated in the current form. 
\end{proof}

The following lemma provides a control for the  Hadamard product related to the Gram matrix. 
In our setup with the point cloud $\mathcal{X}$, as discussed around (\ref{eq_sigdefn}), $\xb_i$ itself is a sub-Gaussian random vector with a spiked $\Sigma$ as in (\ref{eq_sigdefn}). We thus can extract the probability and bounds by tracking the proof in \cite[Step (iv) on Page 21 of the proof of Theorem 2.1]{elkaroui2010}.

\begin{lemma}\label{lem_hardmard}Suppose \eqref{eq_yassum1}-\eqref{eq_ratio} hold true, $d\geq 1$, and $\lambda_l=n^{\alpha_l}$ for $l=1,\ldots,d$, where $0<\alpha_l<1$.  Let $\Gb_x$ be the Gram matrix associated with the point cloud $\mathcal{X}.$ Denote 
\begin{equation}\label{eq_widetildeGb}
\Pb_x:=\Gb_x-\text{diag}\{\Gb_x(1,1),\ldots, \Gb_x(n,n)\}\,.
\end{equation}
For some constant $C>0$, when $n$ is sufficiently large, with probability at least $1-O(n^{-1/2})$, we have
\begin{align}\label{Control:OxPx bound 2}
&\left\|\Pb_x \circ \Pb_x-\frac{ \sum_{l=1}^d (\lambda_l+1)^2+p-1}{p^2}(\mathbf{1} \mathbf{1}^\top-\mathbf{I}_n)  \right\| \nonumber  \\ 
\leq &\, C \max \left\{n^{-1/4}, \frac{ \sum_{l=1}^d \lambda_l}{p}\right\}. 
\end{align}
\end{lemma}

We also need the following lemma for the Gram matrix of noisy signals.
\begin{lemma}\label{lem_hardmard for Gx}
Suppose (\ref{eq_modelsteptwo})-(\ref{eq_lambdadefinition}) hold true, $d\geq 1$, $\lambda_l=n^{\alpha_l}$, where $0<\alpha_l \leq \alpha_{l-1}\leq \cdots \leq \alpha_1<1$, for $l=1,\ldots, d$.  
Let $\Pb_y$ be denoted as (\ref{eq_widetildeGb}) for the point cloud $\mathcal{Y}.$
Then we have
\begin{align}\label{eq_ppppyy}
& \left\|\Pb_x \circ \Pb_x-\Pb_y \circ \Pb_y  \right\| \prec  
\begin{cases}
  \frac{\sum_{l=1}^d \lambda_l}{\sqrt{n}},  & 0<\alpha_1<0.5 \\
  \frac{(\sum_{l=1}^d \lambda_l)^2}{n}, & \text{otherwise}. 
\end{cases}
\end{align}
\end{lemma}

\begin{proof}

When $i \neq j$, 
\begin{align}\label{Hadamard prod difference between x and y}
&\frac{1}{p^2} ((\xb_i^\top \xb_j)^2-(\yb_i^\top \yb_j)^2)\nonumber \\
=&\,\frac{\sum_{l=1}^d\sqrt{\lambda_l} (z_{il}y_{jl}+z_{jl}y_{il})+\zb_i^\top \zb_j}{p}\frac{\xb_i^\top \xb_j+\yb_i^\top \yb_j}{p}\,.
\end{align}
By the assumption that the standard deviations of the entries of $z_{il}$ and $y_{il}$ are of order $\sqrt{\lambda_l}$ and $1$ respectively, we have
\[
\frac{\sum_{l=1}^d\sqrt{\lambda_l} (z_{il}y_{jl}+z_{jl}y_{il})+\zb_i^\top \zb_j}{p}\prec \frac{\sum_{l=1}^d \lambda_l}{n}\,, 
\] 
and by (\ref{eq_finalcontrolimplicationdivergent1}), we have 
\[
\frac{\xb_i^\top \xb_j+\yb_i^\top \yb_j}{p}\prec \frac{1}{\sqrt{n}}
\]
when $\alpha<0.5$ and 
\[
\frac{\xb_i^\top \xb_j+\yb_i^\top \yb_j}{p}\prec \frac{\sum_{l=1}^d \lambda_l}{n}
\] 
when $0.5\leq\alpha<1$.
Therefore, using the Gershgorin circle theorem, we conclude the claimed bound. 
\end{proof}

\subsection{Some results for affinity matrices}\label{sec_affinitymatrixpreliminaryresults}

\begin{lemma}\label{lem_karoui}
Suppose \eqref{eq_defnw} and (\ref{eq_modelsteptwo})-(\ref{eq_lambdadefinition}) hold true, $d\geq 1$, $\alpha_l=0$ for $l=1,\ldots,d$ in \eqref{eq_lambdadefinition} (i.e., $\lambda$ is bounded), and $h=p$ in (\ref{eq_defnw}). Let $\Phi=(\phi_1,\ldots, \phi_n)$ with $\phi_i=\frac{1}{p}\| \xb_i \|-(1+\sum_{l=1}^d \lambda_l/p)$, $i=1,2,\cdots, n$. Denote 
\begin{align}\label{eq_kdtau}
\mathbf{K}_d \equiv \Kb_d(\tau)=&\,-2f'(\tau) p^{-1} \Xb^\top \Xb+\varsigma \mathbf{I}_n+\mathrm{Sh}_0(\tau) +\mathrm{Sh}_1(\tau)+\mathrm{Sh}_2(\tau),
\end{align} 
where $f(x)$ is a general kernel function satisfying the conditions in  Remark \ref{rmk_generalkernel}, $\varsigma$ is defined in \eqref{eq_varsigmalambda},  
\begin{align}
& \mathrm{Sh}_0(\tau):= f(\tau) \mathbf{1} \mathbf{1}^\top,\label{eq_sho} \\
& \mathrm{Sh}_1(\tau):=f'(\tau)[\mathbf{1} \Phi^\top+\Phi \mathbf{1}^\top ],  \nonumber \\
& \mathrm{Sh}_2(\tau):= \frac{f^{(2)}(\tau)}{2}\Big[ \mathbf{1}( \Phi \circ \Phi)^\top+(\Phi \circ \Phi) \mathbf{1}^\top+2 \Phi \Phi^\top \nonumber \\ 
&\qquad\qquad\quad +\frac{4}{p^2}\Big(\sum_{l=1}^d(\lambda_l+1)^2+p\Big) \mathbf{1} \mathbf{1}^\top
\Big], \label{eq_sh2}
\end{align}
and $\circ$ is the Hadamard product. 
Then for some small constant $\xi>0$, when $n$ is sufficiently large,  we have that with probability at least $1-O(n^{-1/2})$
\begin{equation}\label{eq_control}
\norm{\mathbf{W}-\Kb_d} \leq n^{-\xi} \,. 
\end{equation}
\end{lemma}
\begin{proof}
See \cite[Lemma A.10]{DW1} for the case $d=1$. For $d>1$, by modifying  the proof of  \cite[Lemma A.10]{DW1} using Lemmas \ref{lem_concentrationinequality}--\ref{lem_hardmard for Gx}  and the assumption that $d$ is fixed, we get the result.    
\end{proof}

We also need the following lemma, which is of independent interest.

\begin{lemma}\label{Lemma: W1 bound}
Suppose \eqref{eq_defnw} and (\ref{eq_modelsteptwo})-(\ref{eq_lambdadefinition})  hold true, $d\geq 1$ fixed
and $\alpha_1\geq \alpha_2\geq\ldots\geq\alpha_d \geq 1$. For the matrix $\Wb_1$ associated with the clean signal defined in \eqref{eq_singlematrix} and $\delta>0$, with high probability, we have for $C_{\delta}>0$ so that 
 \begin{equation}
 \| \Wb_1 \| \leq (n-C_{\delta} n^{\delta}) \exp(-\upsilon \gamma n^{\alpha_1-1-2(1-\delta)})+ C_{\delta} n^{\delta}\, , \label{proof bound W_1 moderate fast}
 \end{equation} 
where $\gamma$ is defined in (\ref{eq_ratio}). 
\end{lemma}

\begin{proof}
Assume $d=1$ and denote $\alpha=\alpha_1$. Throughout the proof, we adapt the notation \eqref{notationalconvention}.
For an arbitrarily small constant $\epsilon>0$ and a given $\delta>0,$ let $C _{\delta}>0$ be some constant depending on $\delta$, and denote the event $\mathcal{A}(\delta)$ as
\begin{equation*}
\mathcal{A}(\delta):=\left\{\text{There exist} \ C_{\delta} n^{\delta} \ z_{i}'s \ \text{such that} \ |z_{i}| \leq n^{- \epsilon} \right\}\,.
\end{equation*}
Note that $z_{i}$ is the signal part.
Denote $\ell:=\mathbb{P}(|z_{i}| \leq n^{-\epsilon}).$ 
Due to the independence, we find that 
\begin{equation}\label{eq_probabilitysetA}
\mathbb{P}(\mathcal{A}(\delta))={n \choose C_{\delta} n^{\delta}}  \ell^{C_{\delta} n^{\delta}}(1-\ell)^{n-C_{\delta} n^{\delta}}.
\end{equation}
Using Stirling's formula, when $n$ is sufficiently large, we find that 
\begin{align}\label{eq_bionomialcoefficient}
{n \choose C_{\delta} n^{\delta}}
=&\,\frac{n!}{ (n-C_{\delta} n^{\delta})! (C_{\delta} n^{\delta})!}  \\ 
\asymp &\,\sqrt{2\pi} n^{n+1/2} \exp(-n) (C_{\delta} n^{\delta})^{-C_{\delta}n^{\delta}-1/2} \exp(C_{\delta}n^{\delta})\nonumber\\
&\qquad\times (n-C_{\delta}n^{\delta})^{-n+C_{\delta}n^{\delta}-1/2} \exp(n-C_{\delta} n^{\delta}) \nonumber \\
\asymp &\,\sqrt{\frac{2\pi}{C_\delta}} \sqrt{\frac{n^{1-\delta}}{n-C_{\delta}n^{\delta}}} \left(1+\frac{C_{\delta} n^{\delta}}{n-n^{\delta}} \right)^{n-C_{\delta}n^{\delta}} \left( \frac{n}{C_{\delta}n^{\delta}} \right)^{C_{\delta} n^{\delta}} \nonumber \\
\asymp &\,\sqrt{\frac{2\pi}{C_\delta}}  n^{(1-\delta)C_{\delta} n^{\delta}-\delta/2} \exp( C_{\delta} n^\delta) \,.\nonumber
\end{align}
We next provide an estimate for $\ell.$ 
 Denote the probability density function (PDF) of $z_{i}^2$ as $\varrho$ and the PDF of $z_{i}$ is $\widetilde{\varrho}.$  
Recall the assumption that $\{z_{i}\}_{i=1}^n$ are continuous random variables with $\widetilde{\varrho}(0)\neq 0$.
 By using the fact $\varrho(y)=(\widetilde{\varrho}(\sqrt{y})+\widetilde{\varrho}(-\sqrt{y}))/(2\sqrt{y})$, we find that
$$
\varrho(y) \asymp O\left(\frac{1}{\sqrt{y}}\right)\,.
$$
Consequently, we have that for any {small} $y>0,$
\begin{equation} \label{eq_finallbound}
\mathbb{P}(z_{i}^2 \leq y) \asymp O( \sqrt{y})\,.  
\end{equation}
By (\ref{eq_finallbound}), we immediately have that \begin{equation} \label{eq_lcontrolbound}
\ell \asymp n^{-\epsilon}\,.  
\end{equation} 
Since $\ell \rightarrow 0$  as $n \rightarrow \infty$, we have
\begin{align}\label{eq_othercoefficientestimate}
\ell^{C_{\delta} n^{\delta}}(1-\ell)^{n-C_{\delta}n^{\delta}} \asymp \exp(-\ell n) \ell^{C_{\delta}n^{\delta}}, \ n \rightarrow \infty\,.
\end{align}
By plugging (\ref{eq_bionomialcoefficient}) and (\ref{eq_othercoefficientestimate}) into  (\ref{eq_probabilitysetA}), we obtain
\begin{align}\label{eq_adeltabound}
 \mathbb{P}(\mathcal{A}(\delta)) 
 \asymp &\, \exp(-\ell n+C_{\delta}(1-\delta)n^{\delta} \log n+C_{\delta}n^{\delta} \log \ell ) \nonumber  \\
 \asymp &\,\exp\left(-n^{1-\epsilon}+C_{\delta}(1-\delta)n^{\delta} \log n-C_{\delta}\epsilon n^{\delta} \log n +C_{\delta} n^{\delta}\right) \nonumber\\
 = &\, \exp\left(- n^{1-\epsilon}+C_{\delta}\left[1-\delta-\epsilon+1/\log(n) \right]n^{\delta} \log n\right),
\end{align}
where the second asymptotic comes from plugging (\ref{eq_lcontrolbound}). In light of (\ref{eq_adeltabound}), to make $\mathcal{A}(\delta)$ a high probability event, we  
may take $\epsilon+\delta=1$, which leads to  
\begin{equation*}
\mathbb{P}(\mathcal{A}(\delta))  \asymp \exp \left((C_{\delta}-1) n^{\delta} \right).
\end{equation*}
We can choose $1<C_{\delta} \leq 2$ such that for any large constant $D>0,$  when $n$ is large enough, we have that 
 \begin{equation}\label{eq_highprobabilityevent}
 \mathbb{P}(\mathcal{A}(\delta)) \geq 1-n^{-D}\,.
 \end{equation}
Note that $\Wb_1(i,j)=\exp\left(-\upsilon \lambda(z_i-z_j)^2/p\right)$. On the event $\mathcal{A}(\delta)$, by the Gershgorin circle theorem and the definition of $\Wb_1$, we find that 
 \begin{equation}
 \| \Wb_1 \| \leq (n-C_{\delta} n^{\delta}) \exp(-\upsilon \gamma n^{\alpha-1-2(1-\delta)})+C_{\delta} n^{\delta}\,,  
 \end{equation} 
where in the inequality we consider the worst scenario such that all the elements in $\mathcal{A}(\delta)$ are either on the same row or column. This finishes the claim with high probability.

To show the case when $d>1$, we need a slight modification of (\ref{eq_probabilitysetA}). Suppose $\alpha_1\geq\alpha_2\geq\ldots\geq\alpha_d\geq 1$. Since $z_{ik}$, $k=1,\ldots,d$ are in general dependent, (\ref{eq_probabilitysetA}) has to be modified to accommodate this additional dependence. In fact, our results hold true by replacing $\alpha$ in (\ref{proof bound W_1 moderate fast})  with $\alpha_1$ since we have
$\Wb_1(i,j)=\prod_{k=1}^d\exp\left(-\upsilon \lambda_1(z_{ik}-z_{jk})^2/p\right)  \leq \exp\left(-\upsilon \lambda_1(z_{i1}-z_{j1})^2/p\right)$ which reduces our discussion to the case $d=1.$
\end{proof}

\subsection{Orthogonal polynomials and kernel expansion}\label{sec_mehler}

We first recall the celebrated Mehler's formula (for instance, see equation (5) in \cite{hermite} or \cite{FOATA1978367})
\begin{align}\label{eq_melherformula}
\frac{1}{\sqrt{1-t^2}} & \exp\left(\frac{2txy-t^2(x^2+y^2)}{2(1-t^2)}\right) =\sum_{n=0}^{\infty} \frac{t^n}{n!} \widetilde{H}_n(x) \widetilde{H}_n(y),
\end{align}
where $\widetilde{H}_m(x)$ is the {\em scaled Hermite polynomial} defined as 
\begin{equation*}
\widetilde{H}_m(x)=\frac{H_m(x/\sqrt{2})}{\sqrt{2^n}}
\end{equation*} 
and $H_m(x)$ is the {\em standard Hermite polynomial} defined as 
\begin{equation*}
H_m(x)=(-1)^m \exp(x^2)  \frac{\mathrm{d}^m \exp(-x^2) }{\mathrm{d} x^m}. 
\end{equation*}
We mention that $\widetilde{H}_m(x)$ is referred to as the probabilistic version of the Hermite polynomial. We will see later that in our proof, the above Mehler's formula provides a convenient way to handle the interaction term when we write the affinity matrix as a summation of rank one matrices.

\subsection{More remarks}\label{sec_generalizationrevision}
\subsubsection{Zeroing-out technique}\label{sec_generalizationrevision1} Here we discuss the trick of zeroing-out diagonal terms proposed in \cite{el2016graph}. First of all, we summarize the idea and restate the results in \cite{el2016graph}. Then we modify it to our setting (\ref{eq_modelsteptwo})--(\ref{eq_sigdefn}). To simplify the discussion, we focus our discussion on the setting $d=1$ with the signal strength 
$\lambda \asymp n^{\alpha}$, where $\alpha \geq 0$. 
The zeroed out affinity matrix is defined as 
\begin{equation*}
\mathring{\mathbf{W}}=\mathbf{W} \circ \left( \mathbf{1} \mathbf{1}^\top-\mathbf{I}_n \right),
\end{equation*}
where $\mathbf{1} \in \mathbb{R}^n$ is the vector with $1$ in all entries. Denote the associated degree matrix as $\mathring{\mathbf{D}}.$ Consequently, the modified transition matrix is
\begin{equation}\label{alequal}
\mathring{\mathbf{A}}=\mathring{\mathbf{D}}^{-1} \mathring{\mathbf{W}}. 
\end{equation}
Recall that the transition matrix for the signal part is defined as 
$\mathbf{A}_1=\mathbf{D}_1^{-1} \mathbf{W}_1$.
It can be concluded from \cite{el2016graph} that with high probability the spectrum of $\mathring{\mathbf{A}}$ is close to that of $\mathbf{A}_1$ in that 
\begin{equation}\label{eq_closecloseclose}
\|\mathring{\mathbf{A}}-\mathbf{A}_1  \|=o_{\mathbb{P}}(1), 
\end{equation}
provided the following two conditions are satisfied: \\

\noindent{\bf (1).} The signal strength satisfies
\begin{equation*}
\alpha> \frac{1}{2}. 
\end{equation*}  
\noindent{\bf (2).} The off-diagonal entries of the signal kernel affinity matrix should satisfy that 
\begin{equation}\label{eq_keyassumption}
\inf_i \sum_{j \neq i} \frac{\mathbf{W}_1(i,j)}{n} \geq \gamma>0,
\end{equation}
for some universal constant $\gamma>0.$ Even though \cite{el2016graph} did not provide a detailed discussion on the bandwidth selection, the assumption (\ref{eq_keyassumption}) imposes a natural condition. 

We now explain how the zeroing-out trick is related to our approach in the large signal-to-noise ratio region.  Recall that the transition matrix for  the observation is defined as 
$\mathbf{A}=\mathbf{D}^{-1} \mathbf{W}$.
As shown in part 2) of Theorem \ref{thm_adaptivechoiceofc}, when the signal strength is $\alpha>1$ and the bandwidth is either $h \asymp \lambda$ or selected adaptively according to the method proposed in Section \ref{sec_bandwidthselectionalgo},  the spectrum of $\mathbf{A}$ will be close to that of $\mathbf{A}_1$ with high probability, i.e., 
\begin{equation*}
\|\mathbf{A}-\mathbf{A}_1  \|=o_{\mathbb{P}}(1). 
\end{equation*}
We emphasize that when $\alpha>1$ and $h \asymp \lambda$ or selected adaptively according to the method proposed in Section \ref{sec_bandwidthselectionalgo}, it can be concluded from the proof of Corollary \ref{coro_adaptivechoiceofc} that (\ref{eq_keyassumption}) holds with high probability. Consequently, together with (\ref{eq_closecloseclose}), we can conclude that when the signal-to-noise ratio is large in the sense $\alpha>1$ and the bandwidth is selected properly as in Section \ref{sec_bandwidthselectionalgo}, the spectrum of $\mathbf{A}$ is asymptotically the same as the zeroing-out matrix $\mathring{\mathbf{A}}.$
We mention that although in this setting our approach and results are asymptotically equivalent to the zeroing-out trick, our method provides an adaptive and theoretically justified method to select the bandwidth instead of choosing a fixed bandwidth according to (\ref{eq_keyassumption}). In fact, in the simulation of \cite{el2016graph}, the authors used a similar approach empirically.  

When the signal-to-noise ratio is ``smaller'' so that $1/2<\alpha \leq 1,$ the zeroing-out trick has a significant impact on the spectrum. In this region, the proposed bandwidth selection algorithm will choose a bandwidth satisfying $h \asymp p$, and the spectral behavior of noisy GL is recorded in Corollary \ref{thm_normailizedaffinitymatrix}, part 1) of Theorem \ref{thm_adaptivechoiceofc} and Corollary \ref{coro_adaptivechoiceofc}. We see that $\mathbf{A}$ is no longer close to the spectrum of $\mathbf{A}_1$ under our setup. However, the zeroing-out trick works provided (\ref{eq_keyassumption}) holds. Therefore, we can see that when $1/2<\alpha \leq 1$ and the bandwidth is properly selected, the result can be improved using the zeroing-out trick in the sense of (\ref{eq_closecloseclose}).

Finally, when the SNR is ``very small'' so that $\alpha \leq 1/2,$ both $\mathring{\mathbf{A}}$ and $\mathbf{A}$ are dominated by the noise. In this case, we are not able to extract useful information of the signal, even with the zeroing-out trick. For an illustration, in Figure \ref{fig_zero}, we show the performance of $\mathring{\mathbf{A}}$ and $\mathbf{A}$ in different SNR regions by comparing some of their eigenfunctions (eigenvectors) with those of the clean signal matrix $\mathbf{A}_1.$ We can conclude that the zeroing-out trick can be beneficial when $0.5< \alpha \leq 1$ provided the bandwidth is carefully selected.

\begin{figure*}[!ht]
\includegraphics[width=4cm]{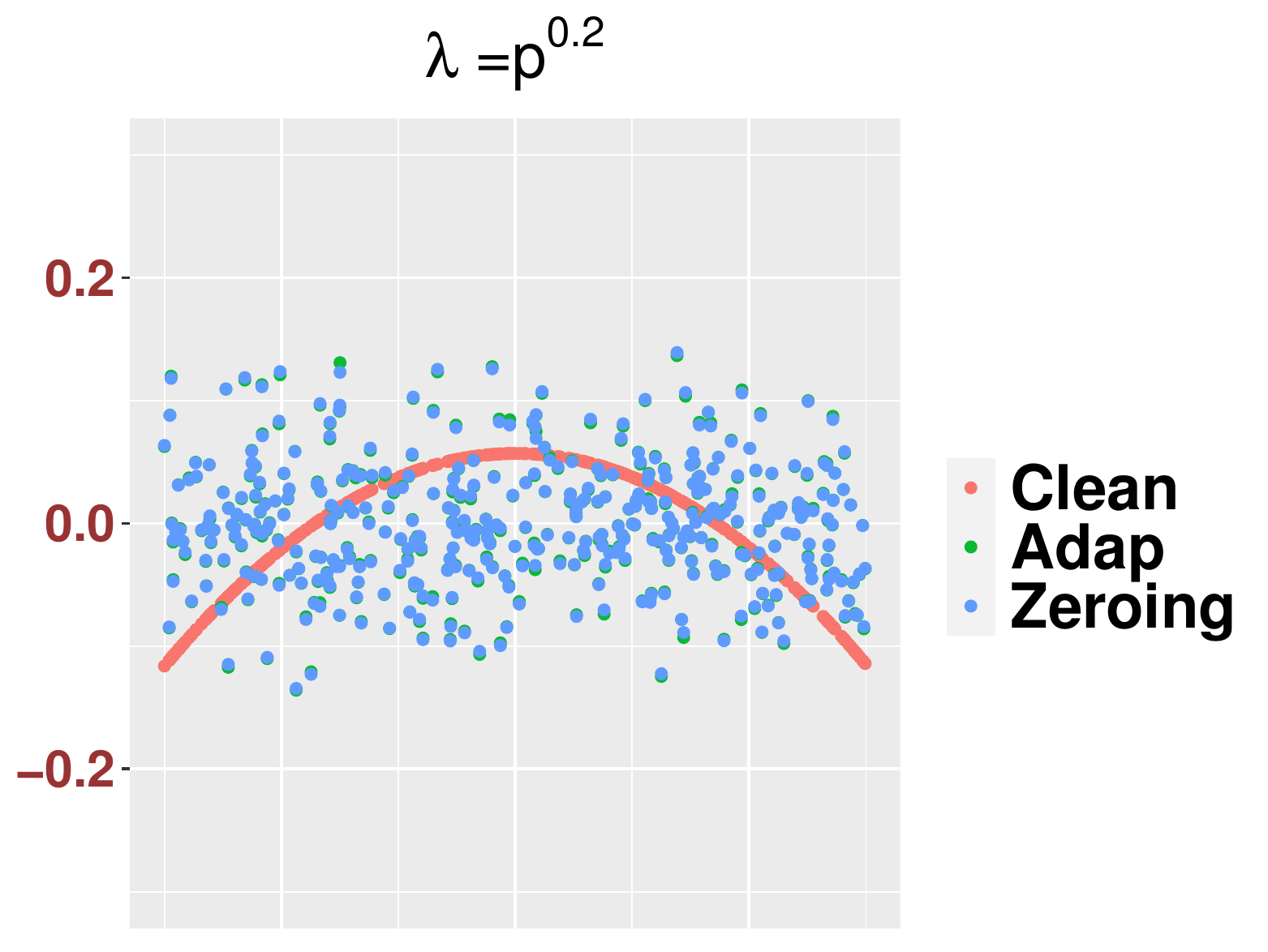}
\includegraphics[width=4cm]{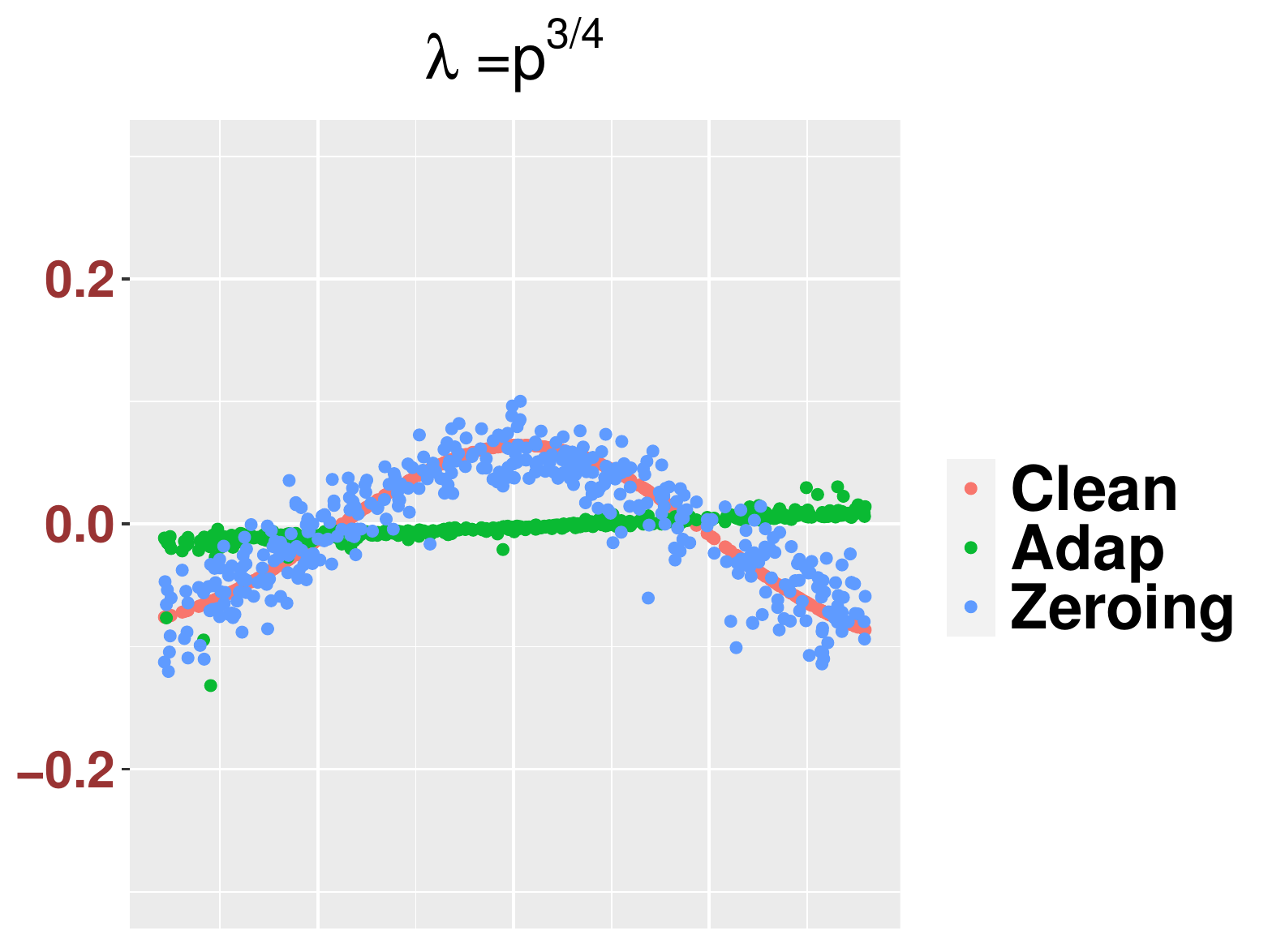}
	\includegraphics[width=4cm]{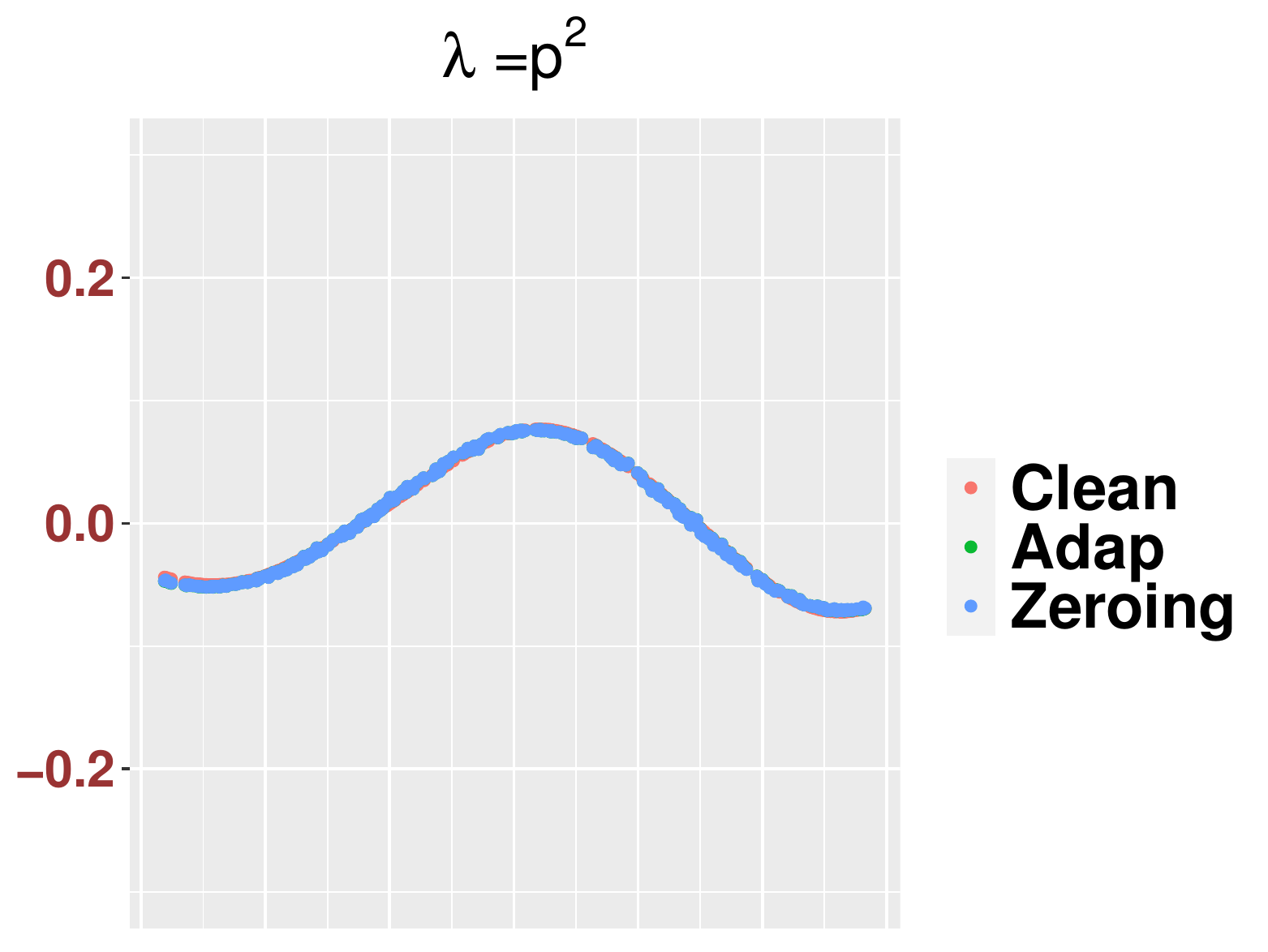}
	\caption{Comparison of different GLs. We consider the comparison of the 3rd eigenfunctions (eigenvectors) of $\Ab_1,$ $\Ab$ with bandwidth constructed using Algorithm \ref{alg:choice} and  $\mathring{\mathbf{A}}$ with $h=35.$ Here $p=200, n=400$ and the random variables are Gaussian. For the legend, Clean means $\Ab_1,$ Adap means $\Ab$ constructed using Algorithm \ref{alg:choice} and Zeoring means $\mathring{\mathbf{A}}$ with $h=35.$ We can see that the zeroing-out trick outperforms our method when $0.5<\alpha \leq 1$ when the bandwidth is chosen as $h=35.$} \label{fig_zero}
\end{figure*}

\subsubsection{Mor remark on $d>1$}\label{sec_generalizationrevision2} 
We continue the discussion in Remark \ref{rmk_multipled} and provide some simulations with $d=2$. We assume that $\alpha_1 \geq \alpha_2 \geq 0$, and discuss two cases with different $\alpha_1$ and $\alpha_2.$ 

First, we discuss the setting when the bulk eigenvalues are governed by the MP law, i.e., in the region $0 \leq \alpha_2 \leq \alpha_1<1.$  Recall that the shifted MP law $\nu_{\lambda}$ defined in (\ref{eq_nudefinition}) depends on the signal level via $\tau \equiv \tau(\lambda)$ in (\ref{eq_defntau}), and
when $\alpha<1,$ $\tau \rightarrow 2$ as $n \rightarrow \infty,$ which is independent of $\lambda$ asymptotically. We can thus always set $\lambda=0$ in (\ref{eq_defntau}) and use $\nu_0$ for the MP law as in (\ref{eq_formrigidity}) and (\ref{eq_formrigidity2}). Therefore, when $0 \leq \alpha_2 \leq \alpha_1<1,$ the bulk distribution is the same as that in (\ref{eq_formrigidity}) and (\ref{eq_formrigidity2}), which is characterized by the shifted MP law, 
\begin{equation*}
\nu_0=T_{\zeta(0)} \mu_{c_n, -2f'(2)},
\end{equation*} 
where $T$ is the shift operator defined in (\ref{eq_shiftoperator}), $\mu_{c_n, -2f'(2)}$ is the MP law defined in (\ref{eq_mp}) with  $\sigma^2$ replaced by $ -2f'(2)$, and $\zeta(0)$ is defined by inserting $\lambda=0$ (or equivalently $\tau(0)=2$) in (\ref{eq_varsigmalambda}); that is,
\begin{equation*}
\zeta(0)=f(0)+2f'(2)-f(2)\,. 
\end{equation*} 
In general, although the bulk distribution is the same for different finite $d,$ the number of outliers can vary according to $\lambda_i$, $1 \leq i \leq d.$ On the technical level, the proofs in Appendices \ref{sec_sub_subsuper} and \ref{sec_sub_slowly} follow after some minor modifications, especially in the Taylor expansion. For example, in (\ref{eq_taylorwx}), the key parameter $\tau$ should be defined as $\tau=2((\lambda_1+\lambda_2)/p+1)$ when $d=2$.  For an illustration, in Figure \ref{fig_d1d2icasei}, we show how the bulk eigenvalues are asymptotically identical for the setting $d=1$ and $d=2$ when the exponents are less than one for various settings of $c_n=0.5,1,2$.  

\begin{figure*}[!ht]
\includegraphics[width=4cm]{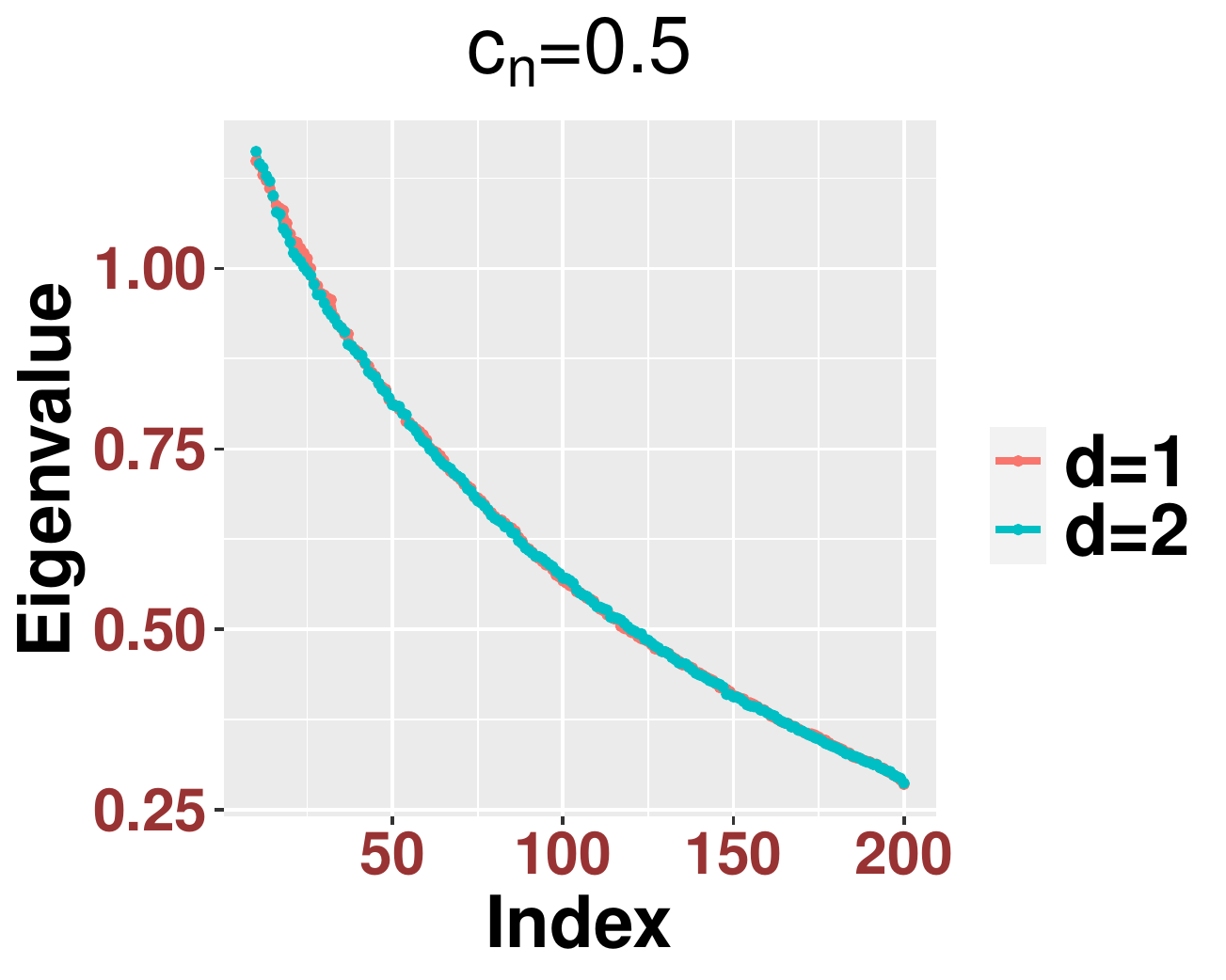}
\includegraphics[width=4cm]{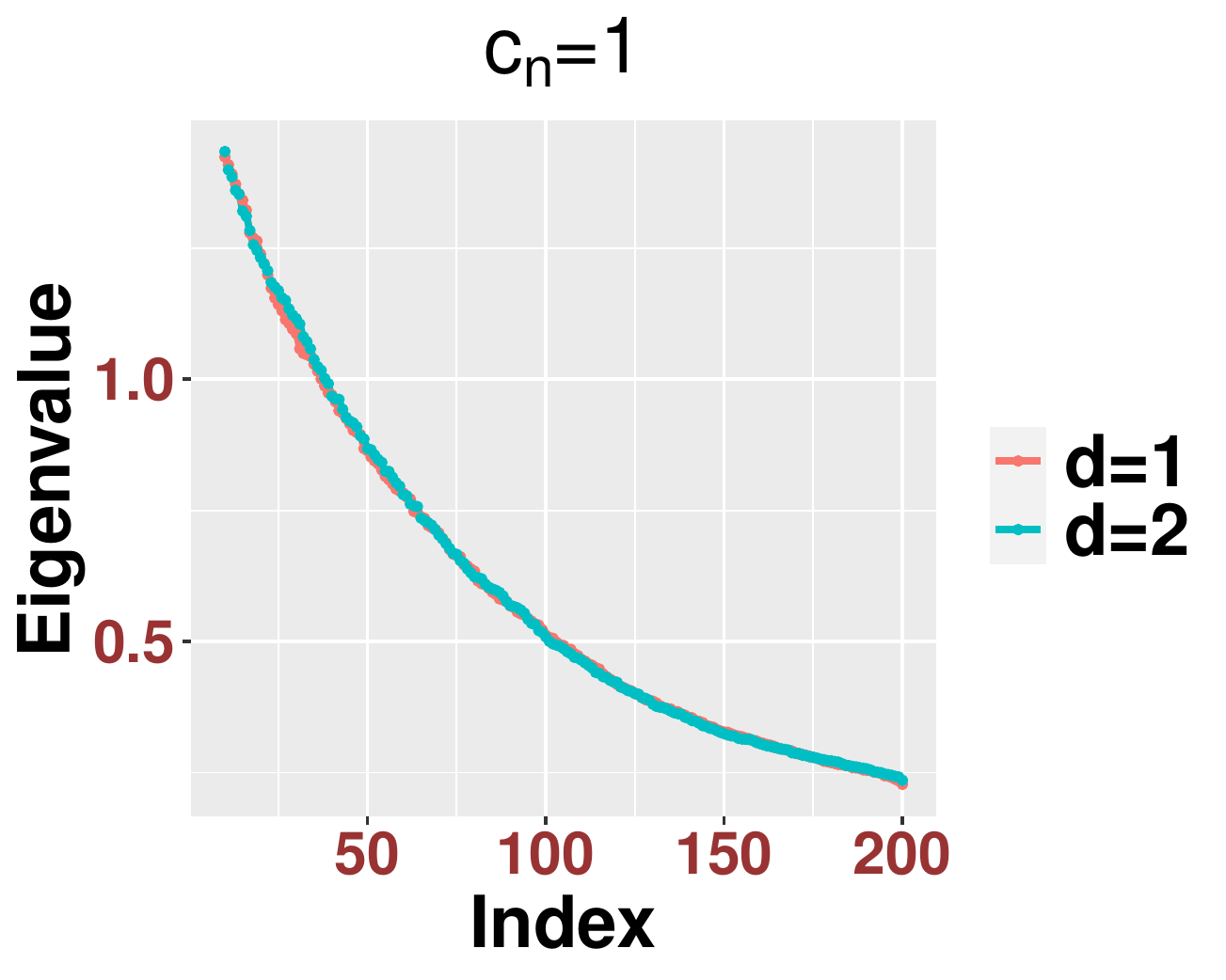}
	\includegraphics[width=4cm]{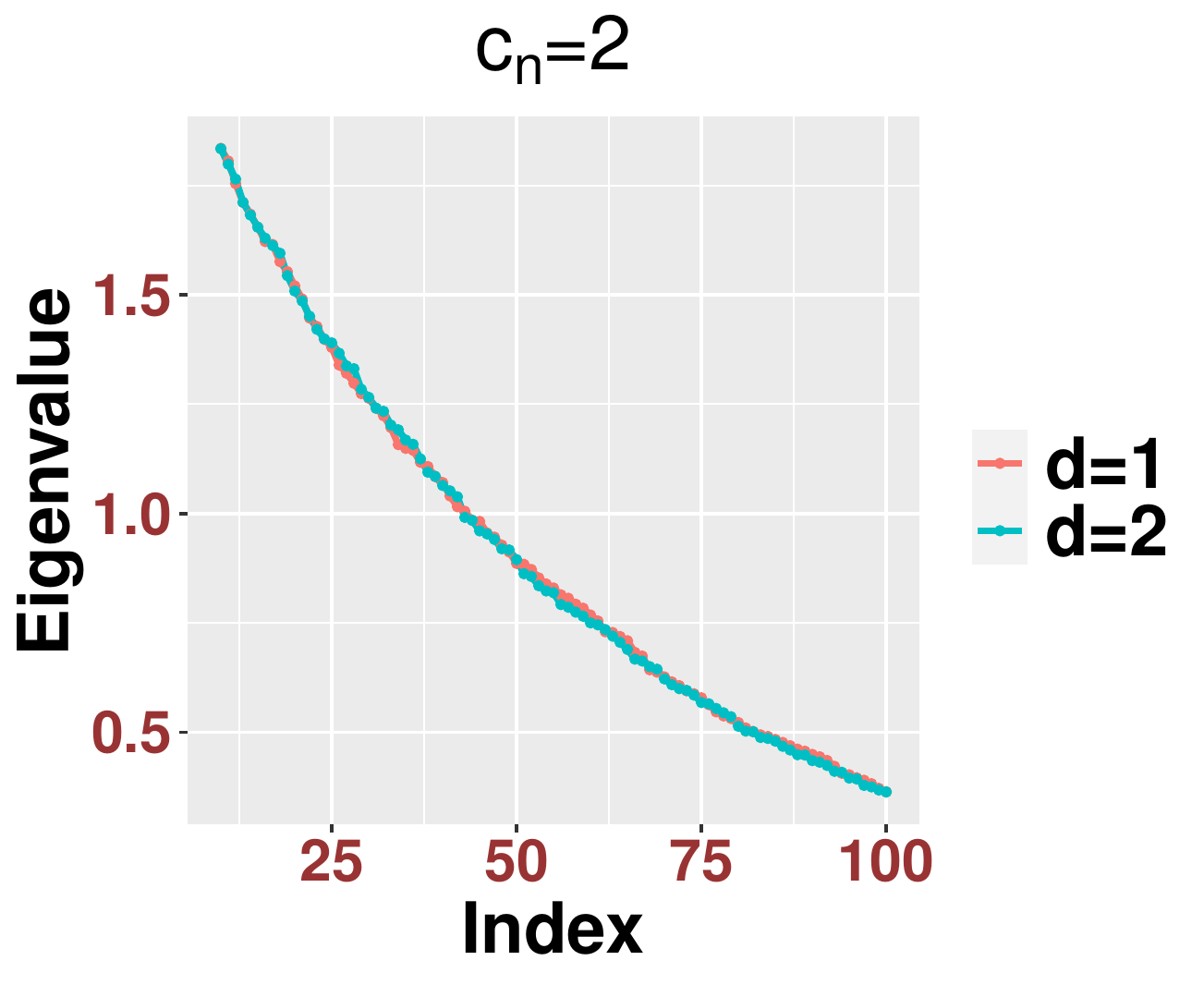}
	\caption{Bulk eigenvalues of $\Wb$ for different intrinsic dimension $d\in \mathbb{N}$ in the low SNR region. For the construction of $\Wb,$ we consider $h=p.$ We compare the bulk eigenvalues for two kernel affinity matrices constructed from two point clouds with different values of $d$. In the simulations, we use the Gaussian random vectors with covariance matrix (\ref{eq_sigdefn}). Specifically, for the setting $d=1,$ we choose $\lambda_1=p^{0.4}, \lambda_2=\cdots=\lambda_p=1,$ and for the setting $d=2,$ we choose $\lambda_1=p^{0.4}, \lambda_2=p^{0.1}, \lambda_3=\cdots\lambda_p=1.$ We also consider the comparison for three different settings of $c_n=n/p$ with $n=200.$ We can see that the bulk eigenvalues for these two settings are fairly close to each other. For the bulk eigenvalues, for definiteness and simplicity, we start from $\lambda_{10}(\Wb).$} \label{fig_d1d2icasei}
\end{figure*}

Second, we discuss the region when $\alpha_2 \geq 1$. On the one hand, as in Theorem \ref{thm_informativeregion}, we can show that the spectrum of $\mathbf{W}$ is close to those of matrices defined in (\ref{eq_wba1}) and (\ref{eq_defntildewa1}) that depend on the clean signal part $\mathbf{W}_1$. As in the case when $d=1$, the spectrum of $\mathbf{W}_1$ may not follow the MP law and depends on both $\lambda_1$ and $\lambda_2$ and the chosen bandwidth. This dependence suggests that the spectrum of $\mathbf{W}$ might be different from that when $d=1$. In Figure \ref{fig_d1d2caseii}, we show numerically how the bulk under the setup $d=2$ and $h=p$ is different from that when $d=1$ and $h=p$. 

Third, when $\alpha_1 \geq 1 >\alpha_2 \geq 0,$ the spectrum of $\mathbf{W}$ will be close to some matrices that depend only on $\lambda_1.$ In this case, the spectrum of $\mathbf{W}$ is close to that of $\mathbf{W}_1$ when $d=1$ and the signal strength is $\lambda_1.$ See Figure \ref{fig_d1d2caseiii} for an illustration, where we see that the bulks are fairly close to each other, while they may not necessarily follow the MP law.

\begin{figure*}[!ht]
\includegraphics[width=4cm]{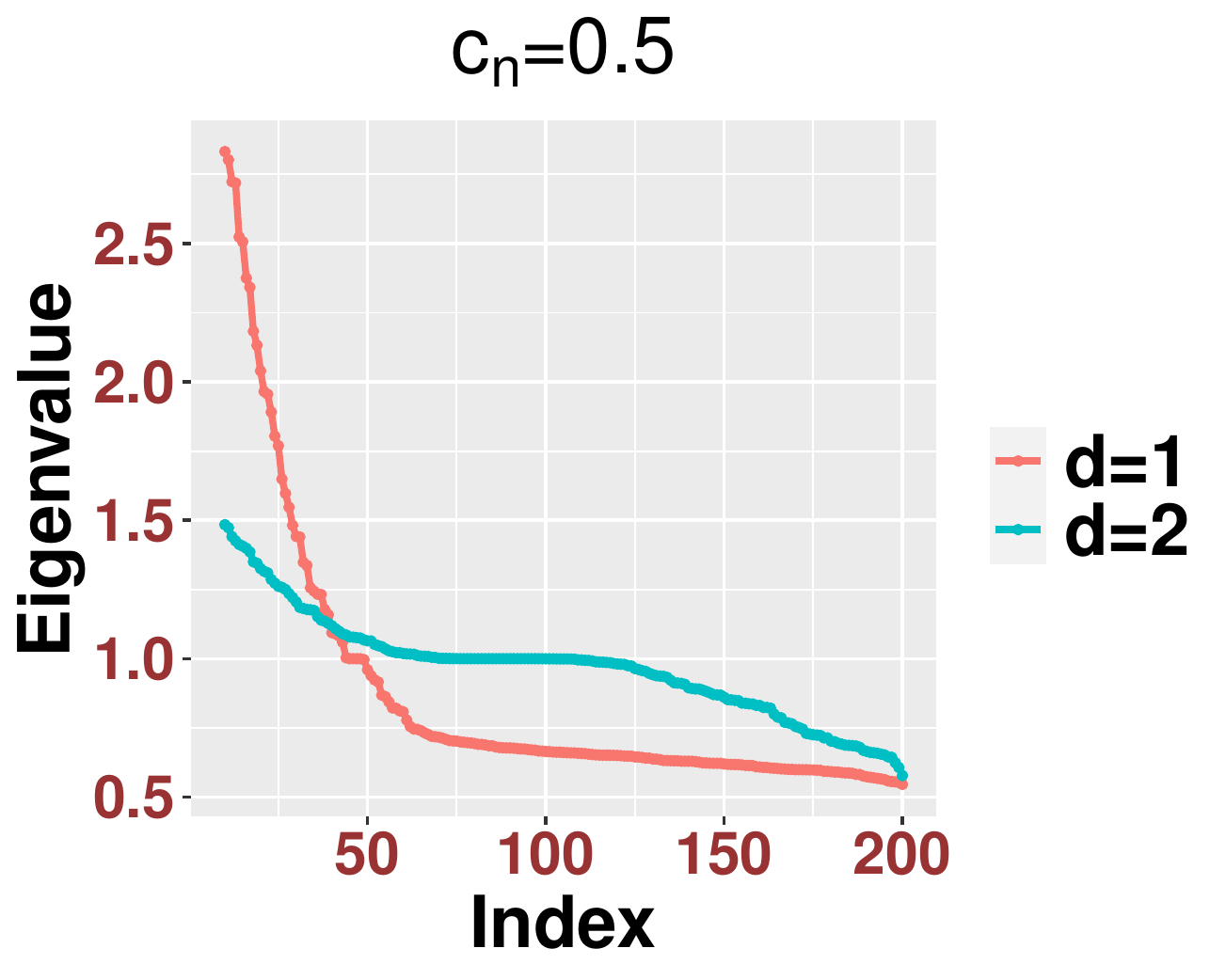}
\includegraphics[width=4cm]{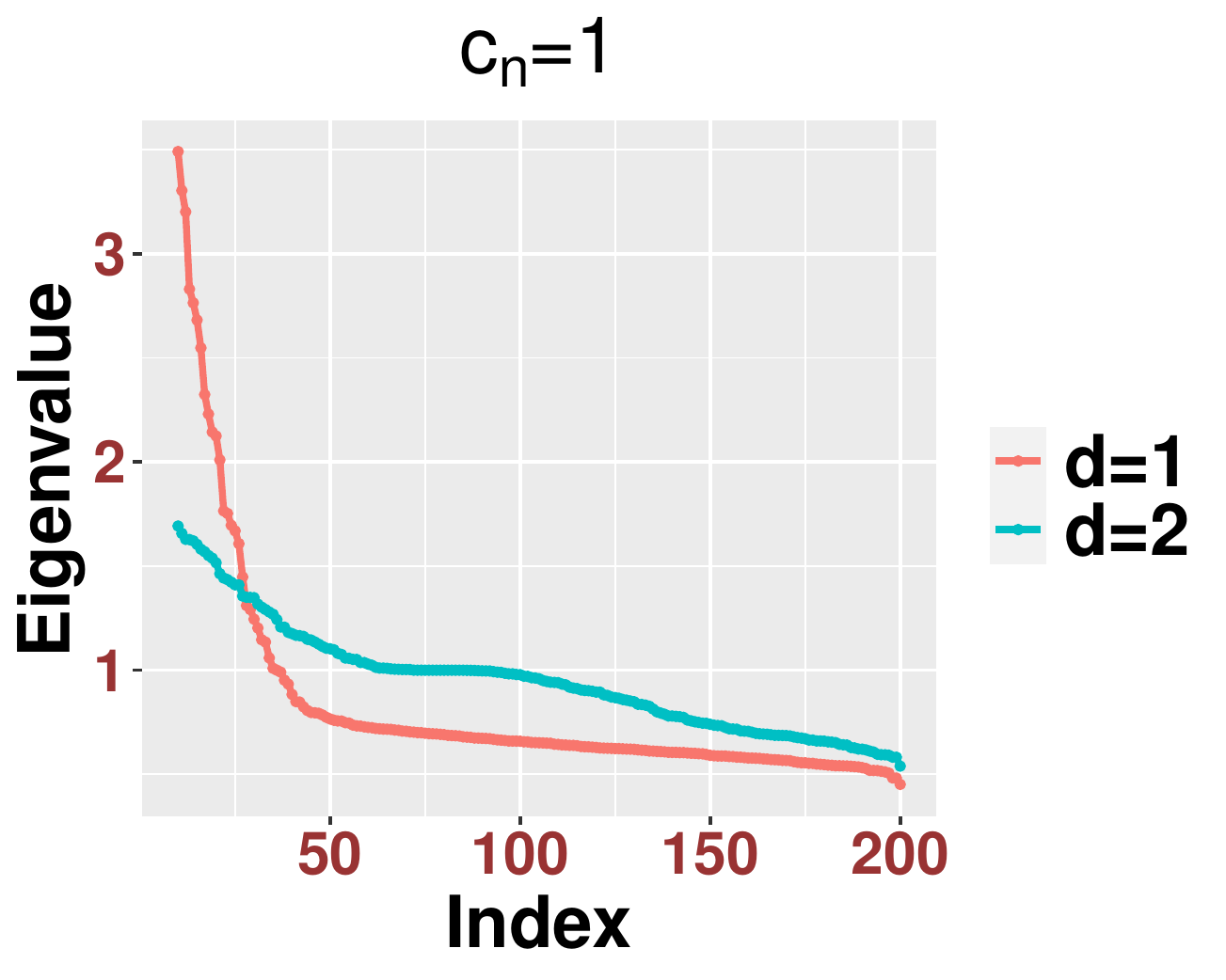}
	\includegraphics[width=4cm]{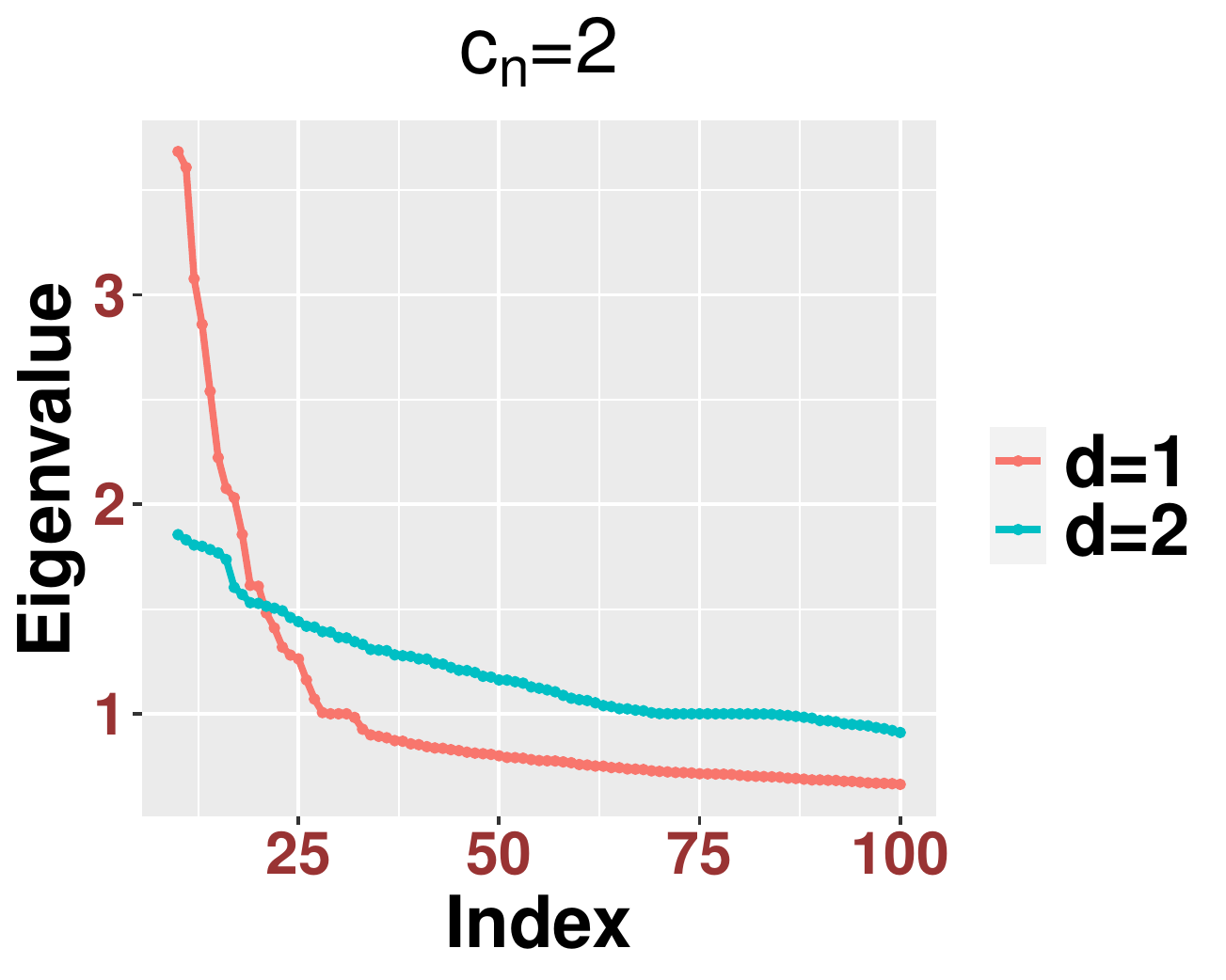}
	\caption{Comparison of bulk eigenvalues of $\Wb$ for different intrinsic dimensions, $d=1$ and $d=2$. When $d=2$, we have two large signals; that is, $\alpha_1\geq \alpha_2\geq 1$.  For the construction of $\Wb,$ we consider $h=p.$ In this simulation, we use the Gaussian random vectors with the covariance matrix (\ref{eq_sigdefn}). When $d=1,$ we choose $\lambda_1=p^{2}, \lambda_2=\cdots=\lambda_p=1$; when $d=2,$ we choose $\lambda_1=p^{2}, \lambda_2=p^{1.6}, \lambda_3=\cdots\lambda_p=1.$ We consider the comparison for three different settings of $c_n=n/p$ with $n=200$. For a better visualization of the bulk eigenvalues, we start from $\lambda_{10}(\Wb).$ We can see that the bulk eigenvalues for these two settings are different and not necessarily follow the MP law. } \label{fig_d1d2caseii}
\end{figure*}

  \begin{figure*}[!ht]
\includegraphics[width=4cm]{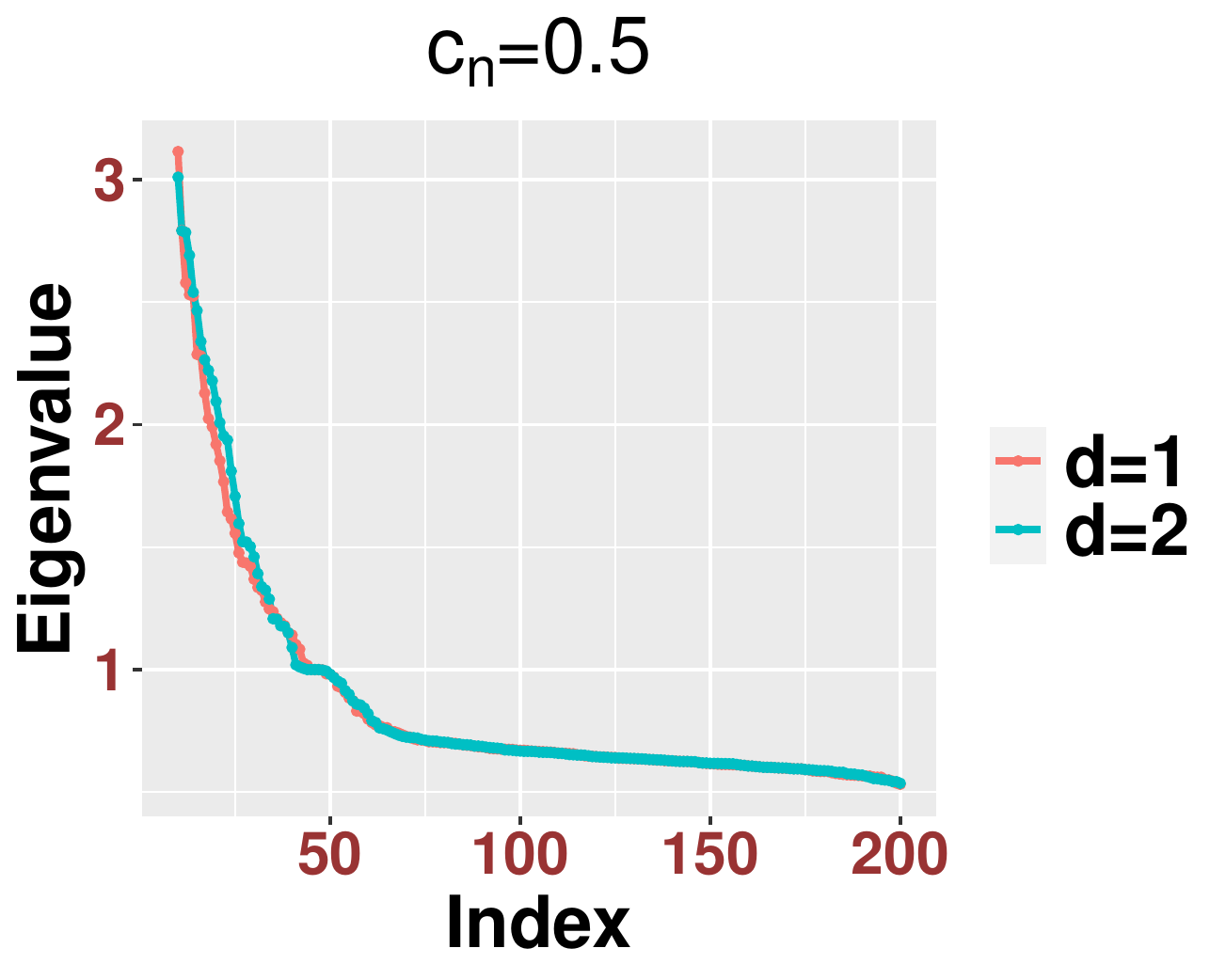}
\includegraphics[width=4cm]{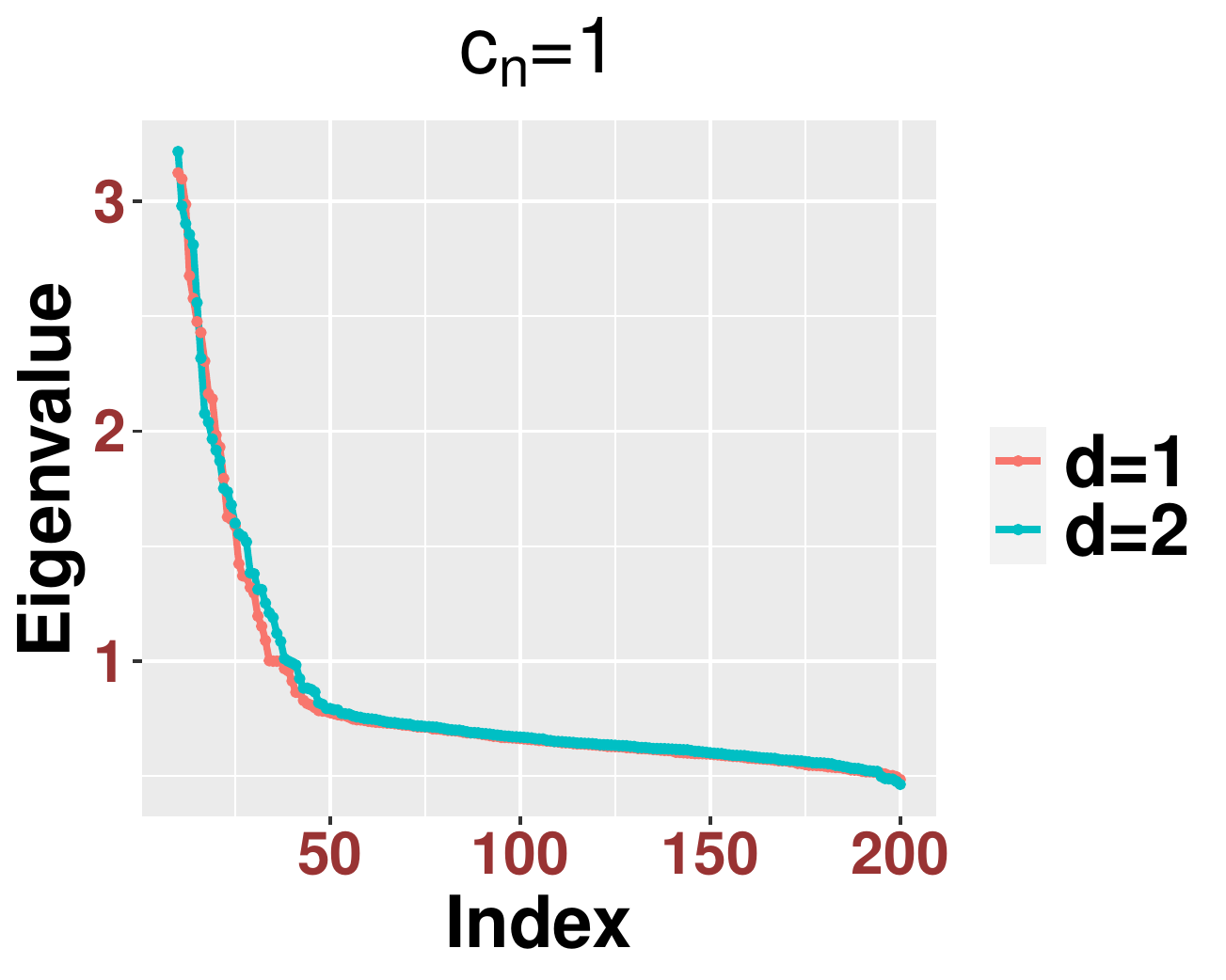}
	\includegraphics[width=4cm]{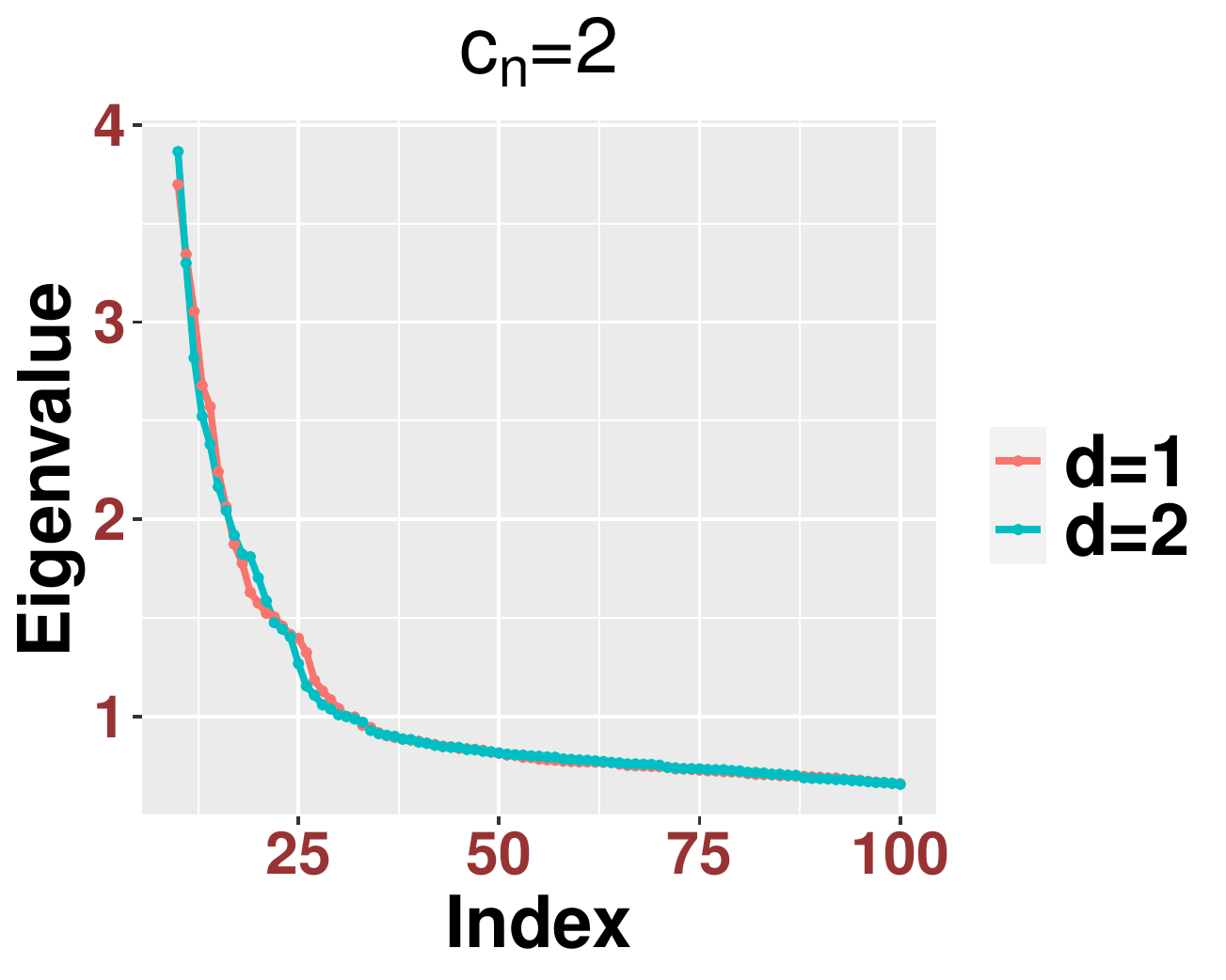}
	\caption{Comparison of bulk eigenvalues of $\Wb$ for different intrinsic dimensions, $d=1$ and $d=2$. When $d=2$, we have one large signal and one small signal; that is, $\alpha_1 \geq 1 >\alpha_2 \geq 0$. For the construction of $\Wb,$ we consider $h=p.$  In the simulations, we use the Gaussian random vectors with the covariance matrix (\ref{eq_sigdefn}). When $d=1,$ we choose $\lambda_1=p^{2}, \lambda_2=\cdots=\lambda_p=1$; when $d=2,$ we choose $\lambda_1=p^{2}, \lambda_2=p^{0.4}, \lambda_3=\cdots\lambda_p=1.$ We consider three different settings of $c_n=n/p$ with $n=200.$ For a better visualization of the bulk eigenvalues, we start from $\lambda_{10}(\Wb).$  We see that the bulk eigenvalues are fairly close to each other.} \label{fig_d1d2caseiii}
\end{figure*}

\section{Proof of main results in Section \ref{section main result}}\label{sec_proofs2} 

In this section, we prove the main results in Section \ref{section main result}.

\subsection{Proof of Theorem \ref{thm_affinity matrix}}\label{sec_sub_subsuper}

Recall $\tau$ defined in (\ref{eq_defntau}).  Since the proof holds for general kernel function described in Remark \ref{rmk_generalkernel}, we will carry out our analysis with such general kernel function $f(x)$. 

\begin{proof}

We start from simplifying $\Wb$. 
Denote $\delta_{ij}$ to be the Kronecker delta. By the Taylor expansion, when $i \neq j$, we have that
\begin{align} \label{eq_taylorwx}
\Wb(i,j)=&f(\tau)+f'(\tau)\left[\Ob_x(i,j)-2 \Pb_x(i,j)\right] \nonumber \\
&+\frac{f^{(2)}(\tau)}{2}\left[\Ob_x(i,j)-2\Pb_x(i,j)\right]^2 \nonumber \\
&+\frac{f^{(3)}(\xi_x(i,j))}{6}\left[\Ob_x(i,j)-2\Pb_x(i,j) \right]^3,
\end{align} 
where $\Pb_x(i,j)$ is defined in \eqref{eq_widetildeGb}, $\Ob_x(i,j)$ is defined as 
\begin{equation}\label{eq_defnop}
\Ob_x(i,j)=(1-\delta_{ij})\left(\frac{\| \xb_i \|_2^2+\| \xb_j\|_2^2}{p}-\tau \right), 
\end{equation}
and $\xi_x(i,j)$ is some value between $(\| \xb_i \|_2^2+\| \xb_j \|_2^2)/p$ and $\tau=2(\lambda/p+1)$ is defined in \eqref{eq_defntau}. 
Consequently, we find that $\Wb$ can be rewritten as 
\begin{align}\label{eq_matrixapproximation}
\Wb=&\, f(\tau) \mathbf{1} \mathbf{1}^\top-\frac{2f'(\tau)}{p} \Xb^\top \Xb  + \varsigma(\lambda) \mathbf{I}\nonumber \\ 
& +\Big[f'(\tau) \Ob_x+\frac{f^{(2)}(\tau)}{2} \Hb_x 
+\frac{f^{(3)}(\xi_x(i,j))}{6} \Qb_x \Big] \nonumber \\
& +2f'(\tau)\left(\frac{1}{p}\text{diag}(\|\xb_1\|^2,\ldots,\|\xb_n\|^2)-1\right),
\end{align}
where we used the shorthand notations 
\begin{align*}
&\Hb_x(i,j)= \left[\Ob_x(i,j)-2\Pb_x(i,j)\right]^2 \\  & \Qb_x(i,j)=\left[\Ob_x(i,j)-2\Pb_x(i,j) \right]^3.
\end{align*}
With this expansion, we immediately obtain  
\begin{align}\label{eq_wx005}
\Wb=& \,f(\tau) \mathbf{1} \mathbf{1}^\top-\frac{2f'(\tau)}{p} \mathbf{X}^\top \Xb+\varsigma(\lambda) \mathbf{I}+f'(\tau) \Ob_x \nonumber \\
& +\frac{f^{(2)}(\tau)}{2} \Hb_x+O_{\prec}(n^{-1/2})\,,  
\end{align}
where the error quantified by $O_{\prec}$ is in the operator norm, the term $\frac{1}{p}\text{diag}(\|\xb_1\|^2,\ldots,\|\xb_n\|^2)-1$ is controlled by Lemma \ref{lem_concentrationslowlydivergent}, and the term $\Qb_x$ is controlled by the facts that $\Ob_x(i,j) \prec (1-\delta_{ij}) n^{-1/2}$, $\Pb_x(i,j) \prec (1-\delta_{ij}) n^{-1/2}$ and the Gershgorin circle theorem. 

Next, we control $\Ob_x$ and $\Hb_x$. Since $\Ob_x =\mathbf{1} \Phi^\top+\Phi \mathbf{1}^\top- 2\text{diag}\{\phi_1, \cdots, \phi_n\}$, we could approximate $\Ob_x$ by $\mathbf{1} \Phi^\top+\Phi \mathbf{1}^\top$ via
\begin{align}\label{eq_oxapproximation}
\| \Ob_x -(\mathbf{1} \Phi^\top+\Phi \mathbf{1}^\top)\|&=\| 2\text{diag}\{\phi_1, \cdots, \phi_n\} \|  \prec n^{-1/2}\,,
\end{align} 
where $\Phi=(\phi_1,\ldots, \phi_n)$ with $\phi_i=\frac{1}{p}\norm{\xb_i}_2^2-(1+\lambda/p)$, $i=1,2,\cdots, n$, is defined in \eqref{eq_kdtau}, and the last bound comes from Lemma \ref{lem_concentrationslowlydivergent}. 

For $\Hb_x$, we write $\Hb_x(i,j)=[\Ob_x(i,j)-2\Pb_x(i,j)]^2=\Ob_x(i,j)^2+4\Pb_x(i,j)^2-4\Ob_x(i,j)\Pb_x(i,j)$ and focus on the term $\Ob_x(i,j) \Pb_x(i,j).$  
Since $\mathbf{1} \Phi^\top\circ \Pb_x=\text{diag}\{\phi_1, \cdots, \phi_n\}\Pb_x$, we have (also see the proof of \cite[Theorem 2.2]{elkaroui2010})
\begin{align}\label{eq_defncx}
\Ob_x\circ \Pb_x  =\,&\Pb_x\text{diag}\{\phi_1, \cdots, \phi_n\} +\text{diag}\{\phi_1, \cdots, \phi_n\}\Pb_x\,.
\end{align} 
Moreover, construct $\Pb_y$ from $\mathcal{Y}$ in the same way as \eqref{eq_widetildeGb} and write 
\begin{equation}\label{eq_fgdefinition}
\Pb_x-\Pb_y=\frac{1}{p} \left(\bm{z} \bm{z}^\top+\bm{z} \bm{y}^\top+\bm{y} \bm{z}^\top\right)\circ(\boldsymbol{1}\boldsymbol{1}^\top-\Ib_n)\,, 
\end{equation}
where $\bm{y}:=(\yb_{11},\cdots, \yb_{n1})^\top, \bm{z}:=(\zb_{11},\cdots, \zb_{n1})^\top.$   We find that 
\begin{equation*}
\|\Pb_x-\Pb_y \| \leq \frac{1}{p}(\bm{z}^\top \bm{z}+2 \bm{z}^\top \bm{y}+\|(\bm{z} \bm{z}^\top+\bm{z} \bm{y}^\top+\bm{y} \bm{z}^\top)\circ \Ib_n\|). 
\end{equation*}
Note that $\|\bm{z} \bm{z}^\top+\bm{z} \bm{y}^\top+\bm{y} \bm{z}^\top\|/p\leq (\|\bm z+\bm y\|^2+\|\bm y\|^2)/p\leq (2\|\bm z\|^2+3\|\bm y\|^2)/p\prec \lambda+1$ by (\ref{eq_finalcontrolimplication1}) and (\ref{eq_finalcontrolimplicationdivergent1}), and $p^{-1} \|(\bm{z} \bm{z}^\top+\bm{z} \bm{y}^\top+\bm{y} \bm{z}^\top)\circ \Ib_n\|\prec (\lambda+1)/p$ by the fact that $|\bz_{i1}|^2\prec \lambda$ and $|\bz_{i1}\by_{i1}|\prec \sqrt{\lambda}$, so we have $\| \Fb_g \| \prec \lambda+1. $ By Lemma \ref{lem_gramsummary} and (\ref{eq_finalcontrolimplication2}), we find that $\| \Pb_y \| \prec 1$. On the other hand, by the fact that $\|\Gb_x\|=\|(\Zb^\top\Zb+\Zb^\top \Yb+\Yb^\top\Zb +\Yb^\top\Yb)/p\|\leq (\|\bm z\|^2+2\|\bm z\|\|\bm y\|)/p+\|\Yb^\top\Yb\|/p\prec \lambda+1$ and \eqref{eq_finalcontrolimplicationdivergent1}, we have
\begin{equation}\label{norm bound of Px alpha<1}
\|\Pb_x\| \prec \lambda+1\,. 
\end{equation}
By (\ref{eq_oxapproximation}), we conclude that 
\begin{equation*}
\| \Ob_x\circ \Pb_x \| \prec \frac{\lambda+1}{\sqrt{n}}. 
\end{equation*}
With the above preparation, 
$\Wb$ is reduced from \eqref{eq_wx005} to  
\begin{align}\label{eq_wx005expansionfinal}
\Wb=\,&f(\tau) \mathbf{1} \mathbf{1}^\top-\frac{2f'(\tau)}{p} \Xb^\top \Xb+\varsigma(\lambda) \mathbf{I}_n+f'(\tau) \Ob_x \nonumber \\
&+\frac{f^{(2)}(\tau)}{2} \Ob_x \circ \Ob_x 
+2f^{(2)}(\tau) \Pb_x \circ \Pb_x \nonumber \\
&+O_{\prec}\left(\frac{\lambda+1}{\sqrt{n}} \right).
\end{align} 
We further simplify $\Wb$. Let $\Pb_y$ be constructed in the same way as $\Pb_x$ in (\ref{eq_defnop}) using the point cloud $\mathcal{Y}.$ By Lemma \ref{lem_hardmard for Gx}, $\Pb_x \circ \Pb_x$ can be replaced by $\Pb_y \circ \Pb_y$.
Moreover, by a discussion similar to (\ref{eq_oxapproximation}), we can control $\Ob_x \circ \Ob_x$ by 
\begin{align*}
\|& (\mathbf{1} \Phi^\top+\Phi \mathbf{1}^\top) \circ (\mathbf{1} \Phi^\top+\Phi \mathbf{1}^\top)-\Ob_x \circ \Ob_x \|  =\|2 \text{diag}\{\phi_1^2,\cdots, \phi_n^2\} \| \prec n^{-1}.
\end{align*}
 Combining all the above results,  and applying Lemma \ref{lem_hardmard}, we have simplified $\Wb$ as
 \begin{align}\label{eq_wxdecompositionslowlydivergent}
\Wb= \mathsf{W}
 +\varsigma(\lambda) \mathbf{I}_n+O\left(n^{\epsilon}\frac{\lambda+1}{\sqrt{n}}+n^{-1/4}\right)\,,
\end{align}
where
\begin{align}\label{eq_reddddd}
\mathsf{W}:=\,&(f(\tau)+2f^{(2)}(2)p^{-1}) \mathbf{1} \mathbf{1}^\top-\frac{2f'(\tau)}{p} \Xb^\top \Xb \nonumber \\
&+f'(\tau)(\mathbf{1} \Phi^\top+\Phi \mathbf{1}^\top) +\frac{f^{(2)}(\tau)}{2} (\mathbf{1} \Phi^\top+\Phi \mathbf{1}^\top) \circ (\mathbf{1} \Phi^\top+\Phi \mathbf{1}^\top)\,,
\end{align}
with probability at least $1-O(n^{-1/2})$ for some small $\epsilon>0$,

With the above simplification, we discuss the outlying eigenvalues. 
 Invoking (\ref{eq_wxdecompositionslowlydivergent}), since $\varsigma(\lambda) \mathbf{I}_n$ is simply an isotropic shift,  the outlying eigenvalues of $\Wb$ can only come from $ \mathsf{W}$. 
Notice that by the identity  $\ab \mathbf{b}^\top \circ \ub \vb^\top=(\ab \circ \ub)(\mathbf{b} \circ \vb)^\top$, we find that  
\begin{align}\label{eq_hardmardexpansion}
(\mathbf{1} \Phi^\top &+\Phi \mathbf{1}^\top) \circ (\mathbf{1} \Phi^\top+\Phi \mathbf{1}^\top) \nonumber \\
& =\mathbf{1}(\Phi \circ \Phi)^\top+(\Phi \circ \Phi) \mathbf{1}^\top+2 \Phi \Phi^\top\,,
\end{align}
which leads to a rearrangement of $\mathsf{W}$ to
\begin{align}
\mathsf{W}=\,&\mathbf{1}\Big[\frac{1}{2}(f(\tau)+2f^{(2)}(2)p^{-1})\mathbf{1}^\top+f'(\tau)\Phi^\top\nonumber\\
&\qquad+\frac{f^{(2)}(\tau)}{2}(\Phi \circ \Phi)^\top \Big]+\,\Big[\frac{1}{2}(f(\tau)+2f^{(2)}(2)p^{-1})\mathbf{1}+f'(\tau)\Phi\nonumber\\
&\qquad+\frac{f^{(2)}(\tau)}{2}(\Phi \circ \Phi) \Big]\mathbf{1}^\top +f^{(2)}(\tau) \Phi \Phi^\top-\frac{2f'(\tau)}{p} \Xb^\top \Xb \nonumber \\
 :=\,&\mathsf{O}-\frac{2f'(\tau)}{p} \Xb^\top \Xb.
\label{eq_spikedpartofmatrix}
\end{align}
Note that $\mathsf{O}$ is of rank at most three since the first two terms of (\ref{eq_spikedpartofmatrix}) form a matrix of rank at most $2$ and $\Phi \Phi^\top$ is a rank-one matrix with the spectral norm of order $O(\sqrt{n})$. With (\ref{eq_wxdecompositionslowlydivergent}), we can therefore conclude our proof using Lemma \ref{lem_gramsummary}.

\end{proof}

\subsection{Proof of Theorem \ref{lem_affinity_slowly}}  \label{sec_sub_slowly}

Since the proof in this subsection hold for general kernel function described in Remark \ref{rmk_generalkernel}, we will carry out our analysis with such general kernel function $f(x)$. 
Note that in this case, $\tau$ defined in (\ref{eq_defntau}) is still bounded from above. So for a fixed $K\in \mathbb{N}$, the first $K$ coefficients in the Taylor expansion can be well controlled under the smoothness assumption, i.e., $f^{(k)}(\tau) \asymp 1$, for $k=1,2,\ldots,K$ for $K\in \mathbb{N}$. However, Lemma \ref{lem_karoui} is invalid since $\alpha>0$. 
On the other hand, in this region, although the concentration inequality (c.f. Lemma \ref{lem_concentrationslowlydivergent}) still works, its rate becomes worse as $\lambda$ becomes larger. In \cite{elkaroui20102}, the author only needs to conduct the Taylor expansion up to the third order since $\lambda$ is fixed. In our setup, to handle the divergent $\lambda$, we need a high order expansion that is adaptive to $\lambda.$ 
Thus, due to the nature of convergence rate in Lemma \ref{lem_concentrationslowlydivergent}, we will employ different proof strategies for the cases $0<\alpha<0.5$ and $0.5 \leq \alpha<1.$ When $\alpha$ satisfies $0<\alpha<0.5$,  the proof of Theorem \ref{thm_affinity matrix} still holds. When $\alpha$ satisfies $0.5 \leq \alpha<1$, we need a higher order Taylor expansion to control the convergence.  
This comes from the second term of (\ref{eq_finalcontrolimplicationdivergent1}), where the concentration inequalities regarding $\xb_i^\top \xb_j$, where $i \neq j$, have different upper bounds with different $\alpha$. 

\begin{proof}[\bf Proof of case (1), $0<\alpha < 0.5$]

 By (\ref{eq_taylorwx}) and Lemma \ref{lem_concentrationslowlydivergent}, we find that when $i \neq j$, 
\begin{align*}
\Wb(i,j)=&\,f(\tau)+f'(\tau)\left[\Ob_x(i,j)-2 \Pb_x(i,j)\right]\\
&+\frac{f^{(2)}(\tau)}{2}\left[\Ob_x(i,j)-2\Pb_x(i,j)\right]^2+O_{\prec}(n^{-3/2}),
\end{align*}
{where we used the fact that $f^{(3)}(\xi_x(i,j))$ is bounded.} By a discussion similar to (\ref{eq_matrixapproximation}) and the Gershgorin circle theorem, we find that (\ref{eq_wx005}) also holds true. The rest of the proof follows lines of the proof of Theorem \ref{thm_affinity matrix} using Lemmas \ref{lem_concentrationslowlydivergent} and \ref{lem_gramsummary}. We omit the details here. 

\end{proof}

\begin{proof}[\bf Proof of Case (2), $0.5 \leq \alpha < 1$] For simplicity, we introduce 
\begin{equation}\label{Definition:Lx}
\Lb_x:=\Ob_x-\Pb_x.
\end{equation}
By Lemma \ref{lem_concentrationslowlydivergent} and notations defined in \eqref{eq_defnop}, we have
\begin{equation}\label{Control:OxPx bound}
|\Pb_x(i,j)|=O_\prec(\lambda/n)\ \mbox{ and }\ |\Ob_x(i,j)|=O_\prec(\lambda/n)\,.
\end{equation}
By the Taylor expansion, when $i \neq j$, we have 
\begin{align}\label{eq_wbx051expansion}
\Wb(i,j)=&\,\sum_{k=0}^{\fd-1} \frac{f^{(k)}(\tau)}{k!}\Lb_x(i,j)^k +\frac{f^{(\fd)}(\xi_x(i,j))}{\fd!}\Lb_x(i,j)^\fd\,,
\end{align}
where $\fd$ is defined in (\ref{eq_defnd}) and $\xi_x(i,j)$ is some value between $(\| \xb_i \|_2^2+\| \xb_j \|_2^2)/p$ and $\tau.$  
Consider $\widetilde{\Wb},\,\Rb_\fd\in \mathbb{R}^{n\times n}$ defined as
\begin{align*}
\widetilde{\Wb}(i,j):=\sum_{k=3}^{\mathfrak{d}-1} \frac{f^{(k)}(\tau)\Lb_x(i,j)^k}{k!}\,,\\
\Rb_\fd(i,j)=\frac{f^{(\fd)}(\xi_x(i,j))}{\fd!}\Lb_x(i,j)^\fd, 
\end{align*}
so that $\Wb=\sum_{k=0}^{2} \frac{f^{(k)}(\tau)}{k!}\Lb_x(i,j)^k+\widetilde{\Wb}+\Rb_\fd$. We start from claiming that 
\begin{equation}\label{eq_errorbound051}
\left| \Rb_\fd(i,j) \right| \prec p^{\mathcal{B}(\alpha)-1}\,, 
\end{equation}
where $\mathcal{B}(\alpha)={(\alpha-1)\left\lceil \frac{1}{1-\alpha} \right\rceil +\alpha<0}$ is defined in (\ref{eq_defnmathcalbalpha}).
To see \eqref{eq_errorbound051}, we use (\ref{eq_finalcontrolimplicationdivergent1}) and the fact that $d$ is finite to get
\begin{equation*}
\Lb_x(i,j)^\fd \prec n^{\fd(-1+\alpha)}=n^{(\alpha-1)\left\lceil \frac{1}{1-\alpha} \right\rceil+\alpha-1}=n^{\mathcal{B}(\alpha)-1}\,.
\end{equation*}
Together with (\ref{eq_wbx051expansion}), by the Gershgorin circle theorem, we have 
\begin{equation*}
\|\Rb_\fd\|=\left\| \Wb-\sum_{k=0}^{\fd-1} \frac{f^{(k)}(\tau)}{k!}\Lb_x^{\circ k} \right\| \prec n^{\mathcal{B}(\alpha)},
\end{equation*}
where we set
\begin{equation*}
\Lb_x^{\circ k}:= \underbrace{\Lb_x \circ \cdots\circ  \Lb_x}_{k \ \text{times}} .
\end{equation*}
Similar definition applies to $\Ob_x^{\circ k}$ and $\Pb_x^{\circ k}.$  

Next, we study $\widetilde{\Wb}.$ Recall the definition of $\Phi$ in (\ref{eq_sho}). To simplify the notation, we denote $\mathsf{F}_1:=\mathbf{1} \Phi^\top+\Phi \mathbf{1}^\top$ and $\mathsf{F}_2:=-2\operatorname{diag}\{\phi_1, \cdots, \phi_n\}$, and obtain 
\begin{equation}\label{eq_oxexpansion}
\Ob_x=\mathsf{F}_1+\mathsf{F}_2.
\end{equation} 
Clearly, $\operatorname{rank}(\mathsf{F}_1)\leq 2$, and by Lemma \ref{lem_concentrationslowlydivergent}, we have
\begin{equation}\label{eq_boundbound}
\| \Ob_x-\mathsf{F}_1 \| \prec \frac{\lambda}{p}. 
\end{equation}
For any $3 \leq k \leq \fd-1$, in view of the expansion 
\[
\Lb_x^{\circ k}=\sum_{l=0}^k {k \choose l} \Ob_x^{\circ l} \circ (-\Pb_x)^{\circ (k-l)}\,,
\] 
below we examine $\Ob_x^{\circ l} \circ (-\Pb_x)^{\circ(k-l)}$ term by term.

First, when $l=0$, we only have the term $\Pb^{\circ k}_x$.
We focus on the discussion when $k=3$, and the same argument holds when $k>3$. We need the following identity. For any $n \times n$ matrix $\Eb$ and vectors $\ub, \vb \in \mathbb{R}^n$, 
\begin{equation}\label{ki}
\Eb \circ \ub \vb^\top=\text{diag}(\ub)\Eb \text{diag}(\vb). 
\end{equation} 
Note that the expansion in (\ref{eq_fgdefinition}) still holds, and to further simplify the notation, we denote  
\begin{align}\label{eq_keyexpansion}
\Pb_x-\Pb_y=\mathsf{T} \circ \mathsf{Q}\,,
\end{align}
where $\mathsf{T}:=\frac{1}{p}(\bm{z} \bm{y}^\top+\bm{y} \bm{z}^\top+\bm{z} \bm{z}^\top)$ and $\mathsf{Q}:=\mathbf{1} \mathbf{1}^\top -\Ib_n$. 
We thus have that
\begin{align}\label{eq_expansionorder}
\Pb_x^{\circ 3}-\Pb_{y}^{\circ 3}  &=(\Pb_x-\Pb_y) \circ (\Pb_x^{\circ 2}+\Pb_x \circ \Pb_y+\Pb_y^{\circ 2}) \nonumber \\
&=(\Pb_x-\Pb_y) \circ \Pb_x^{\circ 2}+(\Pb_x-\Pb_y) \circ \Pb_x \circ \Pb_y +(\Pb_x-\Pb_y) \circ \Pb_y^{\circ 2}. 
\end{align}
We control the first term, and the other terms can be controlled by the same way.
Since $\Pb_y=\Gb_y\circ \mathsf{Q}$ and $\Gb_y=p^{-1} \Yb^\top \Yb$, we have that
\begin{equation*}
\Pb_x^{\circ 2}=\left( \mathsf{T}+\Gb_y \right) \circ \left( \mathsf{T}+\Gb_y \right) \circ \mathsf{Q}. 
\end{equation*}
Together with (\ref{eq_keyexpansion}) and the fact that $\mathsf{Q}^{\circ 2}=\mathsf{Q}$, we obtain 
\begin{align*}
(\Pb_x-\Pb_y) \circ \Pb_x^{\circ 2} & =\left( \mathsf{T}+\Gb_y \right) \circ \left( \mathsf{T}+\Gb_y \right) \circ \mathsf{T} \circ \mathsf{Q} \\
&=\left( \mathsf{T} \circ \mathsf{T}+2 \mathsf{T} \circ \Gb_y+\Gb_y \circ \Gb_y \right) \circ \mathsf{T} \circ \mathsf{Q} \\
& =\mathsf{T}^{\circ 3} \circ \mathsf{Q}+\mathsf{R}_1+\mathsf{R}_2\,,
\end{align*}
where $\mathsf{R}_1:=2 \mathsf{T}^{\circ 2} \circ \Gb_y \circ \mathsf{Q}$ and $\mathsf{R}_2:=\Gb_y^{\circ 2}\circ \mathsf{T} \circ \mathsf{Q}$.
Now we discuss the above three terms one by one. First, using (\ref{ki}), we have 
\begin{align*}
\mathsf{T}^{\circ 3} \circ \mathsf{Q}=\mathsf{T}^{\circ 3}-\left[\mathsf{T}\circ \Ib_n\right]^3=\mathsf{T}^{\circ 3}+O_{\prec}\left( \frac{\lambda^3}{p^3} \right),
\end{align*}
where in the second equality we used the fact that $\mathsf{T}(i,i) \prec \lambda /p. $ Second, we have
\begin{align*}
\mathsf{R}_1& =2 \mathsf{T}^{\circ 2} \circ \Gb_y-2\left[\mathsf{T}\circ \Ib_n\right]^2 [\Gb_y\circ \Ib_n] \\
& =2 \mathsf{T}^{\circ 2} \circ \Gb_y+O_{\prec}\left( (\lambda/p)^2 \right).
\end{align*}
On one hand, we can use (\ref{ki}) to write
\begin{align*}
\mathsf{T}^{\circ 2} \circ \Gb_y 
 =&\,\frac{1}{p^2} \left[\operatorname{diag}(\bm{z})\right]^2 \Gb_y \left[\operatorname{diag}(\bm{z})\right]^2\\
 &+\frac{2}{p^2} \left[\operatorname{diag}(\bm{z})\right]^2 \Gb_y \left[\operatorname{diag}(\bm{z})\right]\left[\operatorname{diag}(\bm{y})\right] \\
&+\frac{2}{p^2}\left[\operatorname{diag}(\bm{z})\right]\left[\operatorname{diag}(\bm{y})\right]  \Gb_y \left[\operatorname{diag}(\bm{z})\right]^2\\
&+\frac{1}{p^2} \left[\operatorname{diag}(\bm{z})\right]^2 \Gb_y \left[\operatorname{diag}(\bm{y})\right]^2 \\
&+\frac{2}{p^2}\left[\operatorname{diag}(\bm{z})\right]\left[\operatorname{diag}(\bm{y})\right]  \Gb_y \left[\operatorname{diag}(\bm{z})\right]\left[\operatorname{diag}(\bm{y})\right] \\
& +\frac{1}{p^2} \left[\operatorname{diag}(\bm{y})\right]^2 \Gb_y \left[\operatorname{diag}(\bm{z})\right]^2. 
\end{align*}
The first term of the above equation is the leading order term which can be bounded as follows
\begin{equation*}
\left\|\frac{1}{p^2} \left[\operatorname{diag}(\bm{z})\right]^2 \Gb_y\left[\operatorname{diag}(\bm{z})\right]^2 \right\| \prec (\lambda/p)^2,
\end{equation*}
where we used the fact that $\| \Gb_y\|=O_{\prec}(1).$ The other terms can be bounded similarly so that
\begin{equation*}
\mathsf{R}_1=O_{\prec}\left( (\lambda/p)^2 \right). 
\end{equation*}
Similarly, we can control $\mathsf{R}_2.$ This shows that 
\begin{equation*}
(\Pb_x-\Pb_y) \circ \Pb_x^{\circ 2}=\mathsf{T}^{\circ 3}+O_{\prec}((\lambda/p)^2). 
\end{equation*}
Analogously, we can analyze  the other two terms in (\ref{eq_expansionorder}) and obtain that
\begin{equation}\label{eq_eee}
\Pb_x^{\circ 3}-\Pb_{y}^{\circ 3} =\mathsf{T}^{\circ 3}+\mathsf{T}^{\circ 2}+O_{\prec}(\lambda/p). 
\end{equation}
Moreover, note that $\Pb_{y}^{\circ 3}=\Pb_{y}\circ(\Pb_{y}^{\circ 2}-\mathsf{Q}/p)+\Pb_{y}\circ \mathsf{Q}/p )$, we have by Lemma \ref{lem_hardamardproductbound} that
\begin{equation*}
\left\|\Pb_{y}\circ(\Pb_{y}^{\circ 2}-\mathsf{Q}/p)\right\| \leq \max_{i,j} |\Pb_{y}(i,j)|  \left\|\Pb_{y}^{\circ 2}-\mathsf{Q}/p\right\| \prec \lambda/p\,.
\end{equation*}
On the other hand, since $\Pb_{y}\circ \mathsf{Q}/p=\frac{1}{p}\Pb_{y}=O_{\prec}(1/p)$, we have $\|\Pb_{y}^{\circ 3}\|\prec  \lambda/p$.
Consequently, we have that 
\begin{equation*}
\Pb_x^{\circ 3}=\mathsf{T}^{\circ 3}+\mathsf{T}^{\circ 2}+O_{\prec}(\lambda/p). 
\end{equation*}
We mention that since $\mathsf{T}$ is at most rank two so that $\mathsf{T}^{\circ 3}+\mathsf{T}^{\circ 2}$ is at most $2^3+2^2 \leq 2^4. $ 
Similarly, using the above discussion for general $k>3,$ we can show that 
\begin{equation*}
\Pb_x^{\circ k}=\sum_{j=2}^k \mathsf{T}^{\circ j}+O_{\prec}(\lambda/p). 
\end{equation*}

Second, when $l=k$, we only have the term $\Ob_x^{\circ k}$. 
When $k=2,$ using (\ref{eq_oxexpansion}) and the fact that $\mathsf{F}_2$ is diagonal, we have that 
\begin{equation*}
\Ob_x^{\circ 2}=\mathsf{F}_1^{\circ 2}+ (\mathsf{F}_2 )^2+2 \mathsf{F}_2 \operatorname{diag}(\mathsf{F}_1) . 
\end{equation*}
By Lemma \ref{lem_concentrationslowlydivergent} (aka (\ref{eq_boundbound})), we have that 
\begin{equation*}
\Ob_x^{\circ 2}=\mathsf{F}_1^{\circ 2}+ O_{\prec}((\lambda/p)^2). 
\end{equation*}
Similarly, for general $k \geq 2,$ we have that 
\begin{equation*}
\Ob_x^{\circ k}=\mathsf{F}_1^{\circ k}+ O_{\prec}((\lambda/p)^k). 
\end{equation*}

Third, when $k \neq l $ and $l \neq 0$, we discuss a typical case when $k=4$ and $l=2$, i.e., $\Ob_x^{\circ 2} \circ \Pb_x^{\circ 2}.$ We prepare some bounds. Recall that
\begin{equation*}
\Pb_x=\Gb_x-\frac{1}{p}\operatorname{diag}\{\| \xb_1 \|_2^2, \cdots, \| \xb_n \|_2^2\}\,,
\end{equation*}
where we have
\[
\Gb_x=\Gb_y+\frac{1}{p}(\bm{z} \bm{y}^\top+\bm{y} \bm{z}^\top+\bm{z} \bm{z}^\top)=:\Gb_y+\mathsf{T}\,. 
\]
By the definition of $\mathsf{T}$, we immediately have 
\begin{equation}\label{Properties of T operator}
\operatorname{rank}(\mathsf{T}) \leq 2\ \ \mbox{ and }\ \ \max_{i,j}|\mathsf{T}(i,j)| \prec \lambda/p\,,
\end{equation} 
where the bound for $\max_{i,j}|\mathsf{T}(i,j)|$ holds by the tail bound of the maximum of a finite set of sub-Gaussian random variables. Similarly, we have the bound 
 \begin{equation}\label{Properties of O operator alpha<1}
 \max_{i,j}|\Ob_x(i,j)|\prec \lambda/p
 \end{equation}
 by \eqref{lem_concentrationslowlydivergent}. By an expansion, we have
\begin{align}
\Pb_x^{\circ 2}-\mathsf{T}^{\circ 2}
=\,&\left(\frac{1}{p}\operatorname{diag}\{\| \xb_1 \|_2^2, \cdots, \| \xb_n \|_2^2\} \right)^2+\Gb_y \circ \Gb_y \nonumber\\
&-2\Gb_y\circ \frac{1}{p}\operatorname{diag}\{\| \xb_1 \|_2^2, \cdots, \| \xb_n \|_2^2\} +2\Gb_y \circ \mathsf{T} \nonumber\\
&-2\mathsf{T}\circ \frac{1}{p}\operatorname{diag}\{\| \xb_1 \|_2^2, \cdots, \| \xb_n \|_2^2\} \,,\label{Pbx2-Tbx2 expansion alpha<1}
\end{align}
which leads to 
\begin{equation}\label{difference Pbx Tbx norm alpha <1}
\|\Pb_x^{\circ 2}-\mathsf{T}^{\circ 2}\|\prec 1
\end{equation}
by Lemma \ref{lem_hardamardproductbound} with the fact $\|\Gb_y\|\prec 1$, $\|\frac{1}{p}\operatorname{diag}\{\| \xb_1 \|_2^2, \cdots, \| \xb_n \|_2^2\} \|\prec 1+\lambda/p$ by Lemma \ref{lem_concentrationslowlydivergent}, \eqref{Properties of T operator} and $\max_{i,j}|\frac{1}{p} \yb_i^\top \yb_j| \prec 1$.
So we have the bound 
\begin{align*}
\|\Ob_x^{\circ 2} \circ \Pb_x^{\circ 2}-\Ob_x^{\circ 2} \circ \mathsf{T}^{\circ 2}\|& \leq \max_{ij}|\Ob_x^{\circ 2}(i,j)|\|\Pb_x^{\circ 2}- \mathsf{T}^{\circ 2}\| \\
&=O_{\prec}((\lambda/p)^2)\,,
\end{align*}
where the first bound comes from Lemma \ref{lem_hardamardproductbound} and the second bound comes from \eqref{Properties of O operator alpha<1}.
Together with $\Ob_x-\mathsf{F}_1=\mathsf{F}_2=-2\operatorname{diag}\{\phi_1, \cdots, \phi_n\}$, {by the same argument} we have that  
\begin{equation*}
\|\Ob_x^{\circ 2} \circ \Pb_x^{\circ 2}- \mathsf{F}_1^{\circ 2} \circ \mathsf{T}^{\circ 2}\|=O_{\prec}(\lambda/p).
\end{equation*}
Using the simple estimate that $\operatorname{rank}(A \circ B) \leq \operatorname{rank}(A) \operatorname{rank}(B)$, we find that 
\begin{equation*}
\operatorname{rank}\left( \mathsf{F}_1^{\circ 2} \circ \mathsf{T}^{\circ 2} \right) \leq 2^4. 
\end{equation*}
The other $k$ and $l$ can be handled in the same way. 
{Precisely, when $l>0$, $\Ob_x^{\circ l} \circ \Pb_x^{\circ (k-l)}$ can be approximated by $\mathsf{F}_1^{\circ l} \circ \mathsf{T}^{\circ (k-l)}$ with the norm difference of order $O_\prec(\lambda/p)$, where $\operatorname{rank}\left(  \mathsf{F}_1^{\circ l} \circ \mathsf{T}^{\circ (k-l)} \right) \leq 2^k$.}
Therefore, up to an error of $O_{\prec}(\lambda/p)$,  all the terms in $\Lb_x^{\circ k}$, except $\Pb_x^{\circ k}$ that will be absorbed into the first order expansion, can be well approximated using a matrix at most of rank $C'k2^{2k}$, where $0<C'\leq 1$.

Consequently, we can find a matrix {$\Mb_f$} of rank at most $C2^{2\fd}$, where $C>0$, to approximate $\widetilde{\Wb}$ so that 
\[
\|\widetilde{\Wb}-\Mb_f \|\prec \lambda/p\,,
\] 
This indicates that $\widetilde{\Wb}$ will generate at most $C2^{2\fd}$ outlying eigenvalues.  Therefore, by replacing $\sum_{k=0}^{2}\frac{f^{(k)}(\tau)}{k!}\Lb_x^{\circ k}$ using \eqref{eq_boundbound} and the facts that $\Ob_x^{\circ 2}$ can be replaced by $\mathsf{F}_1$ and $\Ob_x\circ \Pb_x$ can be approximated by  $\mathsf{F}_1 \circ \mathsf{T}$ with rank bounded than $4$ by the same argument above, 
 we have 
\begin{align*}
&\bigg\|\Wb+f'(\tau)\Pb_x-\frac{f^{(2)}(\tau)}{2}\Pb_x \circ \Pb_x-\tilde\Mb_f\bigg\|
\prec \max\{n^{\mathcal{B}(\alpha)} \,,\lambda/p\}\,,
\end{align*} 
where 
\[
\tilde\Mb_f=f(\tau)\mathsf{Q}+\Big(f'(\tau)+\frac{f^{(2)}(\tau)}{2}\Big)\mathsf{F}_1-f^{(2)}(\tau)\mathsf{F}_1 \circ \mathsf{T}+\Mb_f
\] 
is a low rank matrix of rank at most $C2^{2\fd}+6$.

To finish the spectral analysis of $\Wb$ in \eqref{eq_wbx051expansion}, it remains to deal with the second order Taylor approximation, {$\Pb_x\circ \Pb_x$}, to control the non-outlying eigenvalues follow MP law (i.e., the first order expansion $\Pb_x$ involves $\Xb^\top \Xb.$). The discussion is similar to (\ref{eq_expansionorder}). First, we extend the Hadamard product result in Lemma \ref{lem_hardmard for Gx}. We have the following expansion:{
\begin{align*}
&\Pb_{x} \circ \Pb_{x} - \Pb_{y}  \circ \Pb_{y}\\
=\,&(\Pb_{x}+ \Pb_{y} ) \circ (\Pb_{x} - \Pb_{y}) \\
=\,&[\Pb_x+\Pb_y] \circ \mathsf{T}\circ\mathsf{Q}  \\
= \,&\mathsf{T}^{\circ 2} \circ\mathsf{Q} +2\Pb_y\circ \mathsf{T}\circ \mathsf{Q} \\
=\,&\frac{2}{p} \text{diag}(\bm{z})\Pb_y\text{diag}(\bm{z})+\frac{2}{p} \text{diag}(\bm{z})\Pb_y \text{diag}(\bm{y})  \\
&\qquad +\frac{2}{p} \text{diag}(\bm{y})\Pb_y\text{diag}(\bm{z})+\mathsf{T}^{\circ 2} \circ\mathsf{Q}\,,
\end{align*} 
where the first three terms in the right hand side could be controlled by $O_\prec(\lambda/p)$ by a discussion similar to (\ref{eq_eee}). As a result, we have
As a result, we have
\begin{align}\label{eq_hardmardgetoutbound}
 \Pb_x \circ \Pb_x-\Pb_y \circ \Pb_y=\mathsf{T}^{\circ 2}\circ\mathsf{Q}  +O_\prec(\lambda/p). 
\end{align} 
By the bound $\operatorname{rank}(\mathsf{T}^{\circ 2}\circ\mathsf{Q}) \leq \operatorname{rank}(\mathsf{T}^{\circ 2}) \operatorname{rank}(\mathsf{Q})$ and $\operatorname{rank}(\mathsf{Q})=1$, we  have}
 that  when $0.5 \leq \alpha <1$, $\Pb_x \circ \Pb_x$ {differs from $\Pb_y \circ \Pb_y$ with} at most four extra outlying eigenvalue. 
The rest four outlying eigenvalues will be from the first and second order terms in the Taylor expansion as in (\ref{eq_spikedpartofmatrix}). The discussion is similar to the equations from (\ref{eq_wx005expansionfinal}) to (\ref{eq_spikedpartofmatrix}). We omit the details here.  Finally, we emphasize that the discussion of (\ref{eq_hardmardgetoutbound}) is nearly optimal in light of the obtained bound. 
 
\end{proof}

\subsection{Proof of Theorem \ref{thm_informativeregion}}  \label{sec_sub_fast}
 
While the kernel is $f(x)=\exp(-\upsilon x)$, for notational simplicity, we keep using the symbol $f(x)$ as the kernel function and use the notation \eqref{notationalconvention}.

\begin{proof}[\bf Proof of (1) when $1\leq \alpha<2$]

Let 
\begin{equation*}
\Cb_0=f(2)\mathbf{1} \mathbf{1}^\top+(1-f(2))\mathbf{I}. 
\end{equation*}
In light of (\ref{eq_wba1}) and $\Wb_1(i,i)=1$, we have
\begin{equation*}
\Wb_{a_1}=\Cb_0 \circ \Wb_1.
\end{equation*}
Recall (\ref{eq_defncrossingterm}).  By the definition of $\Wb$, we have that
\begin{equation}\label{eq_generaltermkronecker}
\Wb=\Wb_c \circ \Wb_y \circ \Wb_1.
\end{equation}
Consequently, we have that 
\begin{align}\label{eq_exactseparationexponential}
\Wb-\Wb_{a_1}& =\left[ (\Wb_c-\mathbf{1} \mathbf{1}^\top) \circ \Wb_y \circ \Wb_1  \right]  +\left[ (\Wb_y-\Cb_0) \circ \Wb_1 \right] \nonumber \\
& :=\mathcal{E}_1+\mathcal{E}_2,
\end{align}
where we used the relation $\Wb_{a_1}=(\mathbf{1} \mathbf{1}^\top) \circ \Wb_{a_1}.$ For $\mathcal{E}_2$, 
by Lemma \ref{lem_hardamardproductbound} and the first order Taylor approximation of $\Wb_y$, we see that 
\begin{equation}\label{eq_harmardfirstexpansion}
\| \mathcal{E}_2 \| \prec \frac{1}{\sqrt{n}} \lambda_1(\Wb_1)\,,
\end{equation}
where we used (\ref{eq_finalcontrolimplication1}) and \eqref{eq_finalcontrolimplication2}. To control $\lambda_1(\Wb_1)$, we apply Lemma \ref{Lemma: W1 bound}, where
in order to make the first term of the right-hand side of (\ref{proof bound W_1 moderate fast}) negligible, we take 
\begin{equation}\label{eq_deltachoice}
\delta>\max\left\{0,\,\frac{3-\alpha}{2}\right\}, 
\end{equation} 
and hence with high probability,
\begin{equation}
 \| \Wb_1 \| =O(n^{\max\left\{0,\,\frac{3-\alpha}{2}\right\}})\,. \label{proof bound W_1 moderate fast2}
 \end{equation}
By (\ref{eq_harmardfirstexpansion}), when $\alpha \geq 1$, we find that 
 \begin{equation}\label{eq_finalresult}
 \left\| \frac{1}{n}\mathcal{E}_2 \right\| \prec n^{-3/2}+n^{-\alpha/2}\,.
 \end{equation}
Next, we discuss the first term $\mathcal{E}_1.$ By using Lemma \ref{lem_hardamardproductbound} twice and (\ref{proof bound W_1 moderate fast2}), we see that
\begin{equation}\label{eq_mathcale1decomposition}
\| \mathcal{E}_1 \| \prec  \max_{i,j} \left| \Wb_c(i,j)-1\right| \max_{i,j} \Wb_y(i,j)  \| \Wb_1\|.
\end{equation}
By definition, we have  
\begin{equation*}
\max_{i,j} \Wb_y(i,j)\leq 1. 
\end{equation*}
Moreover, using the definition (\ref{eq_defncrossingterm}) and $h=p$, we have
\begin{equation}\label{eq_crosstermbound}
\max_{i,j} \left| \Wb_c(i,j)-1\right| \prec n^{\alpha/2-1}
\end{equation} 
by the bound $|\bz_i\top \by_j|=|z_i\by_{j1}|\prec \sqrt{\lambda}$ and the Taylor expansion of $f(x)$ around $0$.
Together with (\ref{proof bound W_1 moderate fast2}) and (\ref{eq_mathcale1decomposition}), we readily obtain that
\begin{equation*}
\left\| \frac{1}{n} \mathcal{E}_1 \right\| \prec n^{-1/2}.
\end{equation*}
This completes the proof of (\ref{eq_informativeequationone}).

The proof of (\ref{eq_informativeequationtwo}) is more involved. By the definition of $\Wb_{a_1}$, it suffices to study $\Wb_1.$ 
By Mehler's formula (\ref{eq_melherformula}), we find that when $i \neq j$, 
\begin{align}\label{eq_mehlerexpansion}
\Wb_1(i,j)=\sqrt{1-t_0^2} & \exp\left(\frac{3t_0^2-2}{2(1-t_0^2)}(z_{i}^2+z_{j}^2)\right) \sum_{m=0}^{\infty} \frac{t_0^m}{m!} \widetilde{H}_m(z_{i}) \widetilde{H}_m(z_{j}),
\end{align}
where $t_0$ is defined 
\begin{equation*}
0<t_0:=\frac{-\upsilon+\sqrt{\upsilon^2+16 \beta^2/\upsilon^2}}{4(\beta/\upsilon)}<1\quad\text{and}\quad \beta=\lambda/p\,.
\end{equation*}
By a direct calculation, we have 
\begin{align*}
1-t^2_0\,&=\frac{\upsilon^3 \sqrt{\upsilon^2+16 (\beta^2/\upsilon^2)}-\upsilon^4}{8 \beta^2}, \\
1-t_0\,&=\frac{2\upsilon}{\sqrt{\upsilon^2+16 \beta^2/\upsilon^2}+\upsilon+4(\beta/\upsilon)}.
\end{align*}
Note that when $\alpha>1$, $1-t^2_0\asymp \frac{1}{\beta}$ as $p \to \infty$.
Therefore, with (\ref{eq_mehlerexpansion}), we find that $\Wb_1$ can be written as an infinite summation of rank-one matrices such that
\begin{equation}\label{eq_wbdecompositionexponential}
\Wb_1=\sqrt{1-t_0^2} \sum_{m=0}^{\infty} \frac{t_0^m}{m!} {\Hb}_m {\Hb}_m^\top, 
\end{equation} 
where ${\Hb}_m\in \mathbb{R}^n$ is defined as
\begin{equation*}
{\Hb}_m=\wb \circ \widetilde{\Hb}_m\,, \quad \widetilde{\Hb}_m(i)=\widetilde{H}_m(z_i)\,,
\end{equation*}
and $\wb=(\wb_1,\cdots, \wb_n)^\top \in \mathbb{R}^n$ is defined as
\begin{equation*}
\wb_i=\exp\left(\frac{3t_0^2-2}{2(1-t_0^2)}z_{i}^2\right), \ 1 \leq i \leq n.  
\end{equation*} 
Since both $z_{i}$ and $z_{j}$ are sub-Gaussian random variables, for some large constants $C, D>0$,
\begin{equation}\label{eq_subgaussianbound}
\mathbb{P}(|z_i|^2 \leq C \log n) \geq 1-O(n^{-D}). 
\end{equation}
Therefore, we conclude that with high probability, for some constant $C_1>0$,
\begin{align} \label{eq_caseibelow}
\exp\left(\frac{3t_0^2-2}{2(1-t_0^2)}(z_i^2+z_j^2)\right)&\leq (\exp(C \log n) )^{\frac{3t_0^2-2}{2(1-t_0^2)}} \nonumber \\
& \leq n^{C_1 \beta}.
\end{align}
To finish the proof, we control $\| \Hb_m \Hb_m^\top \|=\| \Hb_m \|_2^2$ case by case.
\newline
\vspace{-5pt}
\newline
\noindent{\bf Case I:} $\alpha=1.$ In this case, {since $\lambda \asymp p$,} we have $\beta \asymp 1$ and $t_0\in (0,1)$ is a constant away from $1$. Since the degree of $\widetilde{H}_m(x)$ is $m\in \mathbb{N}$, we have $|\widetilde{H}_m(x)| \asymp x^m$ when $x \rightarrow \infty$.
Together with (\ref{eq_subgaussianbound}), we find that with high probability, for some constant $C>0$
\begin{equation}\label{eq_controlvectorone}
\|{\Hb}_m\|_2^2 \leq n^{C_1 \beta+1} (C \log n)^{m}.
\end{equation}
Consequently, we have that for some constants $C_2,C_3>0$,
\begin{align}\label{eq_boundeachterm}
\frac{t_0^m}{m!} \|{\Hb}_m \|_2^2 &\leq 
C_3n^{C_2} m^{-1/2} \left( \frac{e t_0C \log n}{m} \right)^m,
\end{align}
where we use the Stirling's formula. 
For notational convenience, we set 
\begin{equation}\label{eq_w1decomposition}
\Wb_1=\Wb_{11}+\Wb_{12},
\end{equation} 
where
\begin{equation}\label{Definition W11 expansion}
\Wb_{11}=\sqrt{1-t_0^2} \sum_{m=0}^{C_0 \log n} \frac{t_0^m}{m!} {\Hb}_m {\Hb}_m^\top
\end{equation}
and $C_0$ will be chosen later. 
Now we control $\Wb_{11}$ and $\Wb_{12}$. Choose a fixed large constant $C_0>0$ so that $C_0>et_0C$, we set $m_0=C_0 \log n.$ When $n$ is large enough, by (\ref{eq_boundeachterm}), we have that for some constants $C_4>0$ and $0<\mathfrak{a}<1$,
\begin{align}\label{eq_allsummation}
\sum_{m=C_0 \log n}^{\infty}\frac{t_0^m}{m!} \|{\Hb}_m\|_2^2 
\leq&\,  C_3 n^{C_2} \sum_{m=C_0 \log n}^{\infty} \frac{1}{\sqrt{m}} \left( \frac{et_0  \log n}{m} \right)^m \nonumber \\
\leq& \, C_4 n^{C_2} \int_{C_0 \log n}^{\infty} \mathfrak{a}^x \mathrm{d} x \nonumber \\
=& \,C_4 n^{C_2} \frac{1}{\log \mathfrak{a}} \mathfrak{a}^{C_0 \log n}\,.
\end{align}
This yields that for some constants $C_5>0$, with high probability
\begin{align}\label{eq_controloneerror}
\sum_{m=C_0 \log n}^{\infty} \frac{t_0^m}{m!} \|{\Hb}_m\|_2^2 & \leq C_4 \mathfrak{a}^{C_0 \log n} n^{C_2}  \asymp n^{-C_5\log n+C_2}.
\end{align} 
Thus, for a big constant $D>2$, when $n$ is sufficiently large, with high probability we have
\begin{align}\label{eq_w12bound}
\|\Wb_{12}\|&=\left\|\Wb_1-\sqrt{1-t_0^2} \sum_{m=0}^{C_0 \log n} \frac{t_0^m}{m!} {\Hb}_m {\Hb}_m^\top\right\|\leq n^{-D}\,.
\end{align}
This completes the proof for the case $\alpha=1$ using (\ref{eq_wba1}) and (\ref{eq_w1decomposition}) since the rank of $\Wb_{11}$ is bounded by $C_0 \log n$. 
\newline
\vspace{-5pt}
\newline
\noindent{\bf Case II: $1 < \alpha <2$. } It is easy to see that in this case, (\ref{eq_controlvectorone}) still holds true with high probability. Since $\beta $ diverges in this case, we find that 
\begin{align*}
\frac{t_0^m}{m!} \|\widetilde{\Hb}_m(\yb)\|_2^2 & \leq C n^{C_1 \beta+1} (m)^{-m-1/2} e^m (C \log n)^{m} \\
&=Cn^{C_1 \beta+1} m^{-1/2} \left( \frac{e C \log n}{m} \right)^m.
\end{align*}
Now for some fixed large constant $C_0>0$, we set $m=C_0 n^{\alpha-1}.$ {By a discussion similar to (\ref{eq_controloneerror}),} we have that for some constants $C, C_1>0$, with high probability 
\begin{equation*}
\sum_{m=m_0}^\infty \frac{t_0^m}{m!} \|\widetilde{\Hb}_m(\zb)\|_2^2 \leq  C n^{C_1(1-\alpha)n^{\alpha-1}}. 
\end{equation*} 
The rest of the proof is similar to the case $\alpha=1$ except that we can follow (\ref{proof bound W_1 moderate fast}) to show that $\| \Wb_1 \|\prec n^{\delta}$ for $\delta>(3-\alpha)/2.$   This concludes the proof for the case $1<\alpha<2$ and hence completes the proof of (\ref{eq_informativeequationtwo}).

\end{proof}

\begin{proof}[\bf Proof of (2) when $\alpha\geq 2$] We first prove (\ref{eq_informativeequationoneone}). By (\ref{eq_generaltermkronecker}), using the definition (\ref{eq_defntildewa1}), we find that 
\begin{equation*}
\Wb-\widetilde{\Wb}_{a_1}=(\Cb_0-\Wb_y) \circ \Wb_c \circ \Wb_1. 
\end{equation*}
Then the proof follows from a discussion similar to (\ref{eq_finalresult}). For the rest of the proof, due to similarity, we only prove (\ref{eq_c1exp}). 
Note that since $z_i$'s are continuous random variables by assumption, it assures that $z_i's$ are not identical with high probability. Rewrite 
\begin{equation*}
\Wb(i,j)=\exp\left(-\upsilon \frac{\lambda}{p} \frac{\| \xb_i-\xb_j \|_2^2}{\lambda}\right).
\end{equation*}
Clearly, $\frac{\lambda}{p}\asymp n^{\alpha-1}$. 
We then show that  for the given constant $t\in (0,1)$, with probability at least $1-O(n^{-\delta})$, where $\delta>1$, we have 
\begin{equation*}
\frac{\| \xb_i-\xb_j \|_2^2}{\lambda} \geq \left(\frac{p}{\lambda}\right)^t,
\end{equation*}
when $i \neq j$.
By a direct expansion, 
\begin{align}\label{eq_expansionlambdadown}
\frac{\| \xb_i-\xb_j \|_2^2}{\lambda}= & \left( z_{i}-z_{j} \right)^2+\frac{\| \yb_i-\yb_j \|_2^2}{\lambda} +\frac{2(z_i-z_j)(\yb_{i1}-\yb_{j1})}{\lambda} \,.  
\end{align}
By Lemma \ref{lem_concentrationinequality}, we find that with high probability,
\begin{equation*}
\frac{\| \yb_i-\yb_j \|_2^2}{\lambda}  \prec  \frac{p}{\lambda}.
\end{equation*}
Similarly, we have that
\begin{equation*}
\left| \frac{2(z_i-z_j)(\yb_{i1}-\yb_{j1})}{\lambda} \right| \prec \frac{1}{\lambda}. 
\end{equation*}
Using the above result, we find that for some constant $C>0$, for $i \neq j$, 
\begin{align*}
\mathbb{P} \left(\frac{\| \xb_i-\xb_j \|_2^2}{\lambda}  \leq \left(\frac{p}{\lambda}\right)^t \right)& \leq \mathbb{P} \left( (z_i-z_{j})^2 \leq \left(\frac{p}{\lambda}\right)^t-C \frac{p}{\lambda} \right)  \leq C \left( \frac{p}{\lambda} \right)^{t/2},
\end{align*}
where the final inequality comes from a discussion similar to \eqref{eq_finallbound} since $\alpha \geq 2$. 
This leads to 
\begin{equation*}
\mathbb{P} \left(\frac{\| \xb_i-\xb_j \|_2^2}{\lambda}  \geq \left(\frac{p}{\lambda}\right)^t \right) \geq 1-C \left( \frac{p}{\lambda} \right)^{t/2}.
\end{equation*}
Note that under the condition \eqref{eq_conditionexponential}, we have $t(\alpha-1)/2>1$, and hence $\left( \frac{p}{\lambda} \right)^{t/2}=o(n^{-1})$.
Therefore, by a direct union bound, for each fixed $i$, we have that 
 \begin{align}\label{eq_boolefrechet}
 \max_{j} \mathbb{P} &\left( \frac{\| \xb_i-\xb_j \|_2^2}{\lambda} \geq \left( \frac{p}{\lambda} \right)^t \middle \vert j \neq i \right)  \geq 1-Cn \left( \frac{p}{\lambda} \right)^{t/2}.
 \end{align}
This implies that with probability at least $1-O(n^{1-t(\alpha-1)/2})$, for some constant $C>0$, we have that 
\begin{align}\label{eq_offdiagnalsummation}
\left| \sum_{j \neq i} \Wb(i,j)\right| &\leq Cn\exp \left( -\upsilon \frac{\lambda}{p} \left( \frac{p}{\lambda} \right)^t    \right) =C n\exp\left(-\upsilon \left( \frac{\lambda}{p} \right)^{1-t} \right).
\end{align} 
Consequently, by the Gershgorin circle theorem, we conclude that with probability $1-O(n^{1-t(\alpha-1)/2})$
\begin{equation}\label{eq_exponentialcircle}
\|\lambda_i(\Wb)-1\| \leq Cn \exp\left(-\upsilon \left( \frac{\lambda}{p} \right)^{1-t} \right). 
\end{equation}
This completes our proof.

\end{proof}

\subsection{Proof of Corollary \ref{thm_normailizedaffinitymatrix}}\label{sec_sub_normalized}

In this subsection, we prove the results for the transition matrix $\Ab.$

\begin{proof}[Proof of Corollary \ref{thm_normailizedaffinitymatrix}]
When $\alpha=0$, it follows from the fact that
\begin{equation}\label{eq_dapproximate}
\|(n^{-1} \Db)^{-1}-f(2)^{-1} \mathbf{I}_n \| \prec n^{-1/2}. 
\end{equation}
The proof can be found in \cite[Lemma 4.5]{DW1}, and we omit it. Together with (\ref{eq_wxdecompositionslowlydivergent}) and (\ref{eq_spikedpartofmatrix}), we have that with probability $1-O(n^{-1/2})$,
\begin{align*}
n\Db^{-1/2}&\Wb\Db^{-1/2}  =n\Db^{-1/2} \mathsf{O} \Db^{-1/2} \\
&+ n \Db^{-1/2} \left(-\frac{2f'(\tau)}{p} \Xb^\top \Xb+ \varsigma(\lambda) \mathbf{I}_n \right) \Db^{-1/2} \\
&+O(n^{-1/4}+n^{\epsilon-1/2})\,. 
\end{align*}
By definition, the first term is a matrix with rank at most three since $\mathsf{O}$ is of rank at most three. For the second term, from (\ref{eq_dapproximate}) and Lemma \ref{lem_gramsummary}, we have 
\begin{align*}
&\Big\|  n \Db^{-1/2} \left( -\frac{2f'(\tau)}{p} \Xb^\top \Xb+\varsigma(\lambda) \Ib_n  \right)\Db^{-1/2} \\
&\qquad-\frac{1}{f(2)}\left( -\frac{2f'(\tau)}{p} \Xb^\top \Xb+\varsigma(\lambda) \Ib_n  \right) \Big\| \prec n^{-1/2}.
\end{align*}
{As a result, since the spectra of $n\Ab$ and $n\Db^{-1/2}\Wb\Db^{-1/2}$ are the same, and with probability $1-O(n^{-1/2})$, $n\Db^{-1/2}\Wb\Db^{-1/2}$ can be approximated by $\frac{1}{f(2)} \left( -\frac{2f'(\tau)}{p} \Xb^\top \Xb+\varsigma(\lambda) \Ib_n  \right)$ and a perturbation of rank at most $3$ with an error bound $O(n^{-1/4}$, we obtain the claim by the Weyl's lemma.}

When $0<\alpha<1$, by a discussion similar to (\ref{eq_dapproximate}) using Lemma \ref{lem_concentrationslowlydivergent}, we find that 
 \begin{equation*}
\|(n^{-1} \Db)^{-1}-f(\tau)^{-1} \mathbf{I}_n \| \prec n^{-1/2}+\frac{\lambda}{p}\, .
\end{equation*}
and by the same argument as that for $\alpha=0$, we conclude the proof.

{Next, we} handle the case $1 \leq \alpha<2.$ Note that 
\begin{align}\label{eq_keydecomposition}
\| \Ab-\Ab_{a_1} \| 
 \leq\,& \| \Db^{-1}(\Wb-\Wb_{a_1}) \|+\| \Db^{-1} (\Db_{a_1}-\Db)\Db_{a_1}^{-1}\Wb_{a_1} \| \nonumber \\
\leq \,& \| \Db^{-1} \| \| \Wb-\Wb_{a_1} \| +\| \Db^{-1} ( \Db-\Db_{a_1})\|\|\Db_{a_1}^{-1}\Wb_{a_1} \|\,. 
\end{align}
By a direct expansion, we have 
\begin{align}
&\Db(i,i) -\Db_{a_1}(i,i)\label{eq_ddcontrol} \\
=\,&\sum_{j \neq i} \exp \left(-\upsilon \frac{\|\xb_i-\xb_j \|_2^2}{p} \right)-\sum_{j \neq i} \exp(-2\upsilon) \exp\left(-\upsilon \frac{\|\zb_{i}-\zb_{j}\|_2^2}{p} \right) \nonumber \\ 
=\,&\sum_{j \neq i} \exp\left( -\upsilon \frac{\|\zb_{i}-\zb_{j}\|_2^2}{p} \right)\nonumber\\
&\times \Big[ \exp \left(-\upsilon \frac{\|\yb_i-\yb_j \|_2^2}{p} \right) \exp\left(-2\upsilon \frac{(\zb_i-\zb_j)^\top (\yb_i-\yb_j)}{p} \right)-\exp(-2\upsilon) \Big].\nonumber
\end{align}
To control $\Db(i,i) -\Db_{a_1}(i,i)$, we bound terms in the right hand side. First, since $\exp(-2\upsilon)$ is the zero-th order Taylor expansion for $\exp(-2\|\yb_i-\yb_j \|/p)$ at $2$, together with (\ref{eq_crosstermbound}) and the fact that $\max_{ij}\Big|\exp \left(-\upsilon \frac{\|\yb_i-\yb_j \|_2^2}{p} \right) - \exp(-2\upsilon)\Big|\prec n^{-1/2}$, we have 
\begin{align}\label{eq_concentrationfirstorder}
& \max_{i,j} \Big| \exp \left(-\upsilon \frac{\|\yb_i-\yb_j \|_2^2}{p} \right) \exp\left(-2\upsilon \frac{(\zb_i-\zb_j)^\top (\yb_i-\yb_j)}{p} \right) \nonumber \\
&\qquad-\exp(-2\upsilon) \Big| \prec n^{\alpha/2-1}\,. 
\end{align}
Since $\exp\left( -\upsilon \frac{\|\zb_{i}-\zb_{j}\|_2^2}{p} \right)>0$, this implies that
\begin{align*}
&\Db(i,i)-\Db_{a_1}(i,i)
=O_{\prec}\left(n^{\alpha/2-1} \sum_{j\neq i}\exp\left( -\upsilon \frac{\|\zb_{i}-\zb_{j}\|_2^2}{p} \right)\right)\,.
\end{align*}
Denote $\Delta:=\Db^{-1} ( \Db-\Db_{a_1}).$ Since $\Db(i,i)$ and $\Db_{a_1}(i,i)$ are positive  and $\Db_{a_1}(i,i)=\sum_{j\neq i}\exp\left( -\upsilon \frac{\|\zb_{i}-\zb_{j}\|_2^2}{p} \right)+1$, we have that
\begin{align*}
\max_i |\Delta(i,i)|& =\frac{|\Db(i,i)-\Db_{a_1}(i,i)|}{\Db(i,i)} \\
& \prec \frac{n^{\alpha/2-1} \sum_{j\neq i}\exp\left( -\upsilon \frac{\|\zb_{i}-\zb_{j}\|_2^2}{p} \right) }{\sum_{j\neq i}\exp\left( -\upsilon \frac{\|\zb_{i}-\zb_{j}\|_2^2}{p} \right)+1} \prec n^{\alpha/2-1}.  
\end{align*}
As a result, we have $\| \Db^{-1} ( \Db-\Db_{a_1})\|\|\Db_{a_1}^{-1}\Wb_{a_1} \|\prec n^{\alpha/2-1}$ since $\|\Db_{a_1}^{-1}\Wb_{a_1} \|\leq 1$.

Next, we control $\| \Db^{-1} \| \| \Wb-\Wb_{a_1} \| $. Using the same argument leading to (\ref{eq_highprobabilityevent}), we find that for $0<\delta<1$, with high probability, there exists $O(n^{\delta})$ amount of $z_{i}$'s such that $|z_{i}|\leq n^{-(1-\delta)}.$ Note that  (\ref{proof bound W_1 moderate fast}) requires that $\delta \geq \frac{3-\alpha}{2}.$ Choosing $\delta=\frac{3-\alpha}{2},$
we obtain that with high probability, for some constant $C>0$, 
\begin{equation*}
\max_{i,i} \Db(i,i) \geq C \sum_{j \neq i} \exp\left(-\upsilon \frac{\|\zb_{i}-\zb_{j}\|_2^2}{p} \right)\geq Cn^{\frac{3-\alpha}{2}}. 
\end{equation*} 
This implies that $\| \Db^{-1} \| \prec n^{(\alpha-3)/2}.$
Combining \eqref{eq_informativeequationone}, we find that $\| \Db^{-1} \| \| \Wb-\Wb_{a_1} \| \prec n^{\alpha/2-1}$, and hence
\begin{equation*}
\| \Ab-\Ab_{a_1} \| \prec n^{\alpha/2-1}\,.
\end{equation*} 
This concludes our proof.

Then we prove the case when $\alpha \geq 2.$ The counterpart of (\ref{eq_informativeequationoneone}) follows from a discussion similar to the case $1\leq \alpha<2$ except that (\ref{eq_ddcontrol}) should be replaced by
\begin{align*}
\Db(i,i) -\widetilde{\Db}_{a_1}(i,i) =\,&\sum_{j \neq i} \exp\left( -\upsilon \frac{\|\zb_{i}-\zb_{j}\|_2^2}{p} \right) \exp\left(-2\upsilon \frac{(\zb_i-\zb_j)^\top (\yb_i-\yb_j)}{p} \right)  \\
&\times \left[ \exp \left(-\upsilon \frac{\|\yb_i-\yb_j \|_2^2}{p} \right) -\exp(-2\upsilon) \right],\nonumber
\end{align*}
so that 
\begin{equation*}
\max_i|\Db(i,i) -\widetilde{\Db}_{a_1}(i,i)| \prec n^{-1/2}. 
\end{equation*} 
We omit the details due to similarity. Finally, we prove the results when $\alpha$ is larger in the sense of (\ref{eq_conditionexponential}).
By (\ref{eq_offdiagnalsummation}) and the definition of $\Wb$ (i.e., $\Wb(i,i)=1$), we find that  with probability at least $1-O(n^{1-t(\alpha-1)/2})$, for some constant $C>0$,
\begin{equation*}
\| \Db-\mathbf{I}_n \| \leq Cn \exp \left( -\upsilon \left( \frac{\lambda}{p} \right)^{1-t} \right).
\end{equation*}
This bound together with \eqref{eq_c1exp} lead to
\begin{align*}
\|\Wb-\Db^{-1/2}\Wb\Db^{-1/2}\| 
 \leq &\, (\|\Db^{-1/2}\|+1)\|\Wb\|\| \Db^{-1/2}-\mathbf{I}_n \| \\
\leq &\, 2Cn \exp \left( -\upsilon \left( \frac{\lambda}{p} \right)^{1-t} \right).
\end{align*}
Since $\Ab$ is similar to $\Db^{-1/2}\Wb\Db^{-1/2}$, by Weyl's lemma, we conclude the claim. This concludes the proof for the third part when (\ref{eq_conditionexponential}) holds.

\end{proof}

\section{Proof of the results in Section \ref{section_result_2ndchoice}}\label{sec_proofs3}
In this section, we prove the main results in Section \ref{section_result_2ndchoice}.

\subsection{Proof of Theorem \ref{thm_adaptivechoiceofc}} In this subsection, we prove Theorem \ref{thm_adaptivechoiceofc} for $h=\lambda+p$. We only prove $\Wb$ and omit the details of $\Ab$ since the proof is similar to that of Corollary \ref{thm_normailizedaffinitymatrix}. Throughout the proof, we write $f(x)=\exp(-\upsilon x)$ for the ease of statements.  
\begin{proof}
For part (1), denote
 \begin{equation*}
f\left(\frac{\|\xb_i-\xb_j \|_2^2}{h} \right)=g\left(\frac{\|\xb_i-\xb_j \|_2^2}{p} \right), \ h=p+\lambda,
\end{equation*}
where $g(x):=f(px/h ).$ Since $0 \leq \alpha<1$, we have that 
\begin{equation*}
\frac{p}{p+\lambda} \asymp 1. 
\end{equation*}
Then we can apply all proofs of Theorems \ref{thm_affinity matrix} and \ref{lem_affinity_slowly} to the kernel function $g(x)$ to conclude the proof. The only difference is that we will get an extra factor $p/h$ for the derivative and we omit the details. 

For part (2), since $\beta=\lambda/(\lambda+p) \asymp 1$, (\ref{eq_closesecondbandwidth}) and (\ref{eq_closebandwidthtwo}) follow from the proof of (\ref{eq_informativeequationone}) and  Case (I) (below equation (\ref{eq_caseibelow})) of the proof of (\ref{eq_informativeequationtwo}), respectively.  For (\ref{eq_largealpharesultadap}), the proof follows from (\ref{eq_closesecondbandwidth}) and the bound 
\begin{align*}
\left\| \Wb_{a_2}-\Wb_1 \right\|
=&\, \|(\exp(-2p\upsilon/(p+\lambda))-1)\Wb_1+(1-\exp(-2p\upsilon/(p+\lambda)))\Ib_n\|\\
\leq &\, n^{1-\alpha}\|\Wb_1\|+n^{1-\alpha} \prec  n^{1-\alpha}+ n^{2-\alpha}\,,
\end{align*} 
where the first inequality comes from $1-\exp(-\frac{2p \upsilon}{p+\lambda}) \asymp n^{1-\alpha}$ when $\alpha>1$ and the last bound follows from the facts that $\| \Wb_1 \| \prec n$.
\end{proof}

\subsection{Proof of Corollary \ref{coro_adaptivechoiceofc}}
In this subsection, we prove Corollary \ref{coro_adaptivechoiceofc}. 
\begin{proof}
First, consider the first statement when $0\leq \alpha<1$. 
Recall that the square of sub-Gaussian is sub-exponential and the summation of sub-exponential random variables is also sub-exponential. Since $\xb_i, \xb_j, i \neq j$, are independent, the random variable $\|\xb_i-\xb_j\|_2^2$ follows a sub-exponential distribution, and by 
Lemma \ref{lem_concentrationslowlydivergent}, we have 
\[
\|\xb_i-\xb_j\|_2^2=2(p+\lambda)+O_\prec(p^\alpha+\sqrt{p})\,.
\]
Since $\alpha<1$, with high probability, when $p$ is large enough, $\|\xb_i-\xb_j\|_2^2$ is concentrated around $2p$. Thus, for any $\omega\in (0,1)$, 
\begin{equation}\label{eq_claimhp}
h \asymp p
\end{equation}
holds with high probability. We can thus rewrite
\begin{equation*}
f\left(\frac{\|\xb_i-\xb_j \|_2^2}{h} \right)=g\left(\frac{\|\xb_i-\xb_j \|_2^2}{p} \right),
\end{equation*}
where $g(x):=f(px/h ).$ Then we can apply all the proof of Theorems \ref{thm_affinity matrix} and \ref{lem_affinity_slowly} to the kernel function $g(x)$ to conclude the proof. The only difference is that we will get an extra factor $p/h$ for the derivative and we omit the details.

We next prove the second statement when $\alpha\geq 1$. We first show the counterpart of (\ref{eq_closesecondbandwidth}).  
In this case, we claim that for some given sufficiently small $\epsilon>0$, we have that with high probability, for some constants $C_1, C_2>0$,
\begin{equation}\label{eq_claimlambadasecond}
C_1(\lambda \log^{-1}n+p) \leq h \leq C_2 \lambda \log^2 n.
\end{equation}
To see this claim, we follow the notation \eqref{notationalconvention}. Note that 
\begin{align*}
\| \xb_i-\xb_j \|_2^2= \,& \|\yb_{i}-\yb_{j}\|^2+\lambda(z_{i}-z_{j})^2+2\sqrt{\lambda}(\yb_{i1}-\yb_{j1})(z_{i}-z_{j})\,.
\end{align*}
By Lemma \ref{lem_concentrationinequality}, $\|\yb_{i}-\yb_{j}\|^2=2p+O_{\prec}(\sqrt{p})$. Also, since $\yb_{i1}-\yb_{j1}$ and $z_i-z_j$ are sub-Gaussian, $|(\yb_{i1}-\yb_{j1})(z_{i}-z_{j})|=O_\prec(\log(n))$. It thus remains to handle $(z_{i}-z_{j})^2$.
On one hand, since $z_{i}-z_{j}$ is sub-Gaussian, we have that for some constant $C>0$, $(z_{i}-z_{j})^2 \leq C \log^2 n$ with high probability; on the other hand, using a discussion similar to (\ref{eq_highprobabilityevent}) by setting $\epsilon=0.5 (\log \log n/\log n)$, with high probability, we have that there are at least $C n/\sqrt{\log n}$ $(z_{i}-z_{j})^2$'s such that $(z_{i}-z_{j})^2 \geq  \log^{-1}n$. Since $\omega\in (0,1)$, when $n$ is large enough, we have $\omega>C/\sqrt{\log n}$. This proves (\ref{eq_claimlambadasecond}).  
Using the bandwidth $h,$ we now have $$\Wb_y(i,j)=f\left(\frac{\| \yb_i-\yb_j\|_2^2}{h}\right)=f\left(\frac{p}{h} \frac{\| \yb_i-\yb_j \|_2^2}{p}\right).$$ Denote $\widetilde{\Cb}_0=f(2p/h) \mathbf{1} \mathbf{1}^\top+(1-f(2p/h))\mathbf{I}.$ Consider the same expansion like \eqref{eq_exactseparationexponential} using the bandwidth $h$; that is,
$\Wb-\Wb_{a_2}=[ (\Wb_c-\boldsymbol{1}\boldsymbol{1}^\top) \circ \Wb_y \circ \Wb_1  ]+[ (\Wb_y- \widetilde{\Cb}_0) \circ \Wb_1 ]$.
By Lemma \ref{lem_hardamardproductbound}, we have 
\begin{align*}
\| \Wb-\Wb_{a_2} \| \leq \,&  (\max_{i,j} \left| \Wb_c(i,j)-1\right| \max_{i,j} \Wb_y(i,j) \\
& +\max_{i,j}|\Wb_y(i,j)- \widetilde{\Cb}_0(i,j)|) \| \Wb_1\|\,.
\end{align*}
Due to \eqref{eq_claimlambadasecond}, we have $|(z_i-z_j)(\yb_{i1}-\yb_{j1})|/h=O_\prec(\lambda^{-1/2}) $ and hence $\max_{i,j}|\Wb_c(i,j)-1| \prec \sqrt{\lambda}/h=O(\lambda^{-1/2})=O(n^{-\alpha/2})$.
On the other hand, we have $\max_{i,j} \Wb_y(i,j)\leq 1$ by definition and $\max_{i,j}|\Wb_y(i,j)- \widetilde{\Cb}_0(i,j)|\prec n^{-1/2}$ by Lemma \ref{lem_concentrationinequality}. 
As a result, we obtain
\begin{equation}\label{eq_bbbbb}
\| \Wb-\Wb_{a_2} \| \prec n^{-1/2} \| \Wb_1 \|.
\end{equation}

To finish the proof of the counterpart of (\ref{eq_closesecondbandwidth}), we need to control $\|\Wb_1\|$ with $h$ satisfying (\ref{eq_claimlambadasecond}). We use a trivial bound $\|\Wb_1\|=O(n)$ by Gershgorin circle theorem. As a result, we have
\[
\left\| \frac{1}{n}\Wb-\frac{1}{n}\Wb_{a_2} \right\| \prec n^{-1/2} \,.
\]

The argument for the couterpart of \eqref{eq_closebandwidthtwo} is analogous to that in Case ($\mathbf{I}$) in the proof of (\ref{eq_informativeequationtwo}) with the following necessary modifications. For $\beta:=\lambda/h$, we have 
$\frac{1}{C_2 \log^2 n} \leq \beta \leq \frac{1}{C_1(\log^{-1} n+p/\lambda)}$ by (\ref{eq_claimlambadasecond}). 
First, when $\beta\asymp 1$, the discussion reduces to Case ($\mathbf{I}$) (i.e., the arguments below (\ref{eq_caseibelow})) of the proof of (\ref{eq_informativeequationtwo}).  Second, when $\beta$ diverges and $\beta \leq C \min\{\log n, n^{\alpha-1}\}$ for some constant $C>0$, the discussion reduces to Case ($\mathbf{I}$) again. Finally, we discuss $\beta=o(1)$ satisfying $\beta \geq 1/(C_2 \log^2 n).$ In this case, we still have $t_0<1.$ Recall (\ref{eq_wbdecompositionexponential}), where we have $1-t_0^2 \asymp \frac{1}{\beta} \leq C \log^2 n$ for some constant $C>0.$ However, compared to the error bound in (\ref{eq_controloneerror}), the extra factor $\log^2 n$ is negligible. Then the case still reduces to Case ($\mathbf{I}$). This completes the proof.

\end{proof}

\section{Verification of Remark \ref{rmk_partone}}\label{sec_rmkjust} In this section, we justify (\ref{eq_stieltjestransformlimit}) of Remark \ref{rmk_partone}.

\begin{proof}[\bf Justification of Remark \ref{rmk_partone}] 
We focus on the case $\alpha=1$, and the same claim holds for $1<\alpha<2$. Recall (\ref{eq_generaltermkronecker}). Using a discussion similar to (\ref{eq_exactseparationexponential}), we have that
\begin{align}\label{eq_expansioncontrol}
\Wb-\Wb_{b_1} =\,& (\Wb_c-\mathbf{1} \mathbf{1}^\top) \circ \Wb_y \circ \Wb_1 +\Wb_1 \circ \Rb_1+\Wb_1 \circ \left(\exp(-2 \upsilon) \mathbf{1} \mathbf{1}^\top \right) \nonumber \\
:=\,&\Eb_0+\Eb_1+\Eb_2,
\end{align}
where $\Rb_1$ is the error of the first order entrywise expansion of $\Wb_y$ defined as
\begin{align}\label{eq_R1R1R1DFN}
\Rb_1 &:=\Wb_y- \Big[\exp(-2 \upsilon) \mathbf{1} \mathbf{1}^\top+\frac{2\upsilon\exp(-2\upsilon)}{p} \Yb^\top \Yb +2 \upsilon \exp(-4 \upsilon)\mathbf{I} \Big] \nonumber \\
& :=\Wb_y-\Ub_y \,.
\end{align}
Take the same decomposition in \eqref{Definition W11 expansion} with some fixed large constant $C_0>0$ and $m=C_0 \log n$. 
By (\ref{eq_controloneerror}), we have that with high probability
\begin{equation}\label{eq_wdecomposeerror}
\operatorname{rank}(\Wb_{1,1}) \leq m, \quad  \max_{i,j} |\Wb_{1,2}(i,j)| \leq n^{-D},
\end{equation}
for some large constant $D>2.$ We start with the discussion of $\Eb_2$ in (\ref{eq_expansioncontrol}).
Using (\ref{eq_w1decomposition}) and the fact that $\Wb_1 \circ \left(\exp(-2 \upsilon) \mathbf{1} \mathbf{1}^\top \right)=\exp(-2 \upsilon)\Wb_1$, we obtain that
\begin{align}\label{eq_e2decomposition}
\Eb_2&=\exp(-2\upsilon)\Wb_{1,1} +\exp(-2\upsilon)\Wb_{1,2} =:\Eb_{2,1}+\Eb_{2,2} \,. 
\end{align}
By (\ref{eq_wdecomposeerror}) and the Gershgorin circle theorem,  with high probability, for some constant $C>0,$ we have  
\begin{equation}\label{eq_e2boundone}
\operatorname{rank}\left( \Eb_{2,1} \right) \leq m. 
\ \ \mbox{ and }\ \ 
\left \|\Eb_{2,2}  \right\| \leq C n^{-D+1} 
\end{equation}
We control $m_{\Wb_{b_1}+\Eb_1+\Eb_2}(z)-m_{\Wb_{b_1}+\Eb_1}(z)$ by the triangle inequality,
\begin{align*}
& |m_{\Wb_{b_1}+\Eb_1+\Eb_2}(z)-m_{\Wb_{b_1}+\Eb_1}(z)| \\
 \leq\,& |m_{\Wb_{b_1}+\Eb_1+\Eb_2}(z)-m_{\Wb_{b_1}+\Eb_1+\Eb_{21}}(z) |+|m_{\Wb_{b_1}+\Eb_1+\Eb_{21}}(z)-m_{\Wb_{b_1}+\Eb_1}(z) |\,.
\end{align*}
By Lemma \ref{lem_collectioninequality}, we have
\begin{align}
&|m_{\Wb_{b_1}+\Eb_1+\Eb_2}(z)-m_{\Wb_{b_1}+\Eb_1+\Eb_{2,1}}(z)| \nonumber\\
\leq \,& \frac{\operatorname{rank}(\Eb_{2,2})}{n}\min\left\{\frac{2}{\eta},\,\frac{\|\Eb_{2,2}\|}{\eta^2}\right\}\leq \frac{C}{n^{D-1}\eta^2}
\label{eq_finalcontrole2}\,,
\end{align}
where we use the fact that the rank of $\Eb_{2,2}$ may be full and $D$ is large.
Similarly, we have
\begin{align}
&|m_{\Wb_{b_1}+\Eb_{2,1}+\Eb_1}(z)-m_{\Wb_{b_1}+\Eb_1}(z)|\nonumber\\
 \leq \,& \frac{\operatorname{rank}(\Eb_{2,1})}{n}\min\left\{\frac{2}{\eta},\,\frac{\|\Eb_{2,1}\|}{\eta^2}\right\}\leq \frac{2C_0\log n}{n\eta}
\label{eq_finalcontrole2q}\,,
\end{align}
where we use the fact that $\|\Eb_{2,1}\|\leq \|\exp(-2\upsilon)\Wb_1\|\leq \exp(-2\upsilon) n^{-\delta}$ for some $0<\delta<1/2$ in the argument leading to
\eqref{proof bound W_1 moderate fast}.
In conclusion, since $D$ is a big constant, we control the $\Eb_2$ term in the Stieltjes transform, and obtain 
\begin{align}
|m_{\Wb_{b_1}+\Eb_1+\Eb_2}(z)-m_{\Wb_{b_1}+\Eb_1}(z)|\leq \frac{3C_0\log n}{n\eta}\,
\end{align}
for $z \in \mathcal{D}$ defined in (\ref{defn_domaind}). 
It is easy to see that the above control is negligible compared to the rate $n^{-1/2+\epsilon} \eta^{-2}$ for any arbitrarily small $\epsilon>0.$ Consequently, it suffices to focus on $\Eb_1$ in (\ref{eq_expansioncontrol}). 
To control $\Eb_1$ in (\ref{eq_expansioncontrol}), note that 
\begin{equation}\label{eq_e1decomposition}
\Eb_1=\Wb_{1,1} \circ \Rb_1+\Wb_{1,2} \circ \Rb_1
\end{equation}
and we control
\begin{align*}
& |m_{\Wb_{b_1}+\Eb_1}(z)-m_{\Wb_{b_1}}(z)| \\
 \leq&\, |m_{\Wb_{b_1}+\Wb_{1,1} \circ \Rb_1+\Wb_{1,2} \circ \Rb_1}(z)-m_{\Wb_{b_1}+\Wb_{1,1} \circ \Rb_1}(z) |\\
&+|m_{\Wb_{b_1}+\Wb_{1,1} \circ \Rb_1}(z)-m_{\Wb_{b_1}}(z) |\,.
\end{align*}
By a discussion similar to (\ref{eq_wx005}) with for the cloud points $\{\yb_i\}$ (or see the proof of Lemma 4.2 of \cite{DW1}), we have
 \begin{equation}\label{eq_r1bound}
 \max_{i,j} |\Rb_1(i,j)| \prec n^{-1/2}.
 \end{equation}
Together with (\ref{eq_wdecomposeerror}), by the Gershgorin circle theorem, we conclude that 
\begin{equation}\label{eq_rb1w12control}
\norm{  \Wb_{1,2}\circ \Rb_1} \prec n^{-D+1/2}.
\end{equation}
This error is negligible since $D>2$ by the same argument as \eqref{eq_finalcontrole2}, thus we have the control for $|m_{\Wb_{b_1}+\Wb_{1,1} \circ \Rb_1+\Wb_{1,2} \circ \Rb_1}(z)-m_{\Wb_{b_1}+\Wb_{1,1} \circ \Rb_1}(z) |$. The rest of the proof is controling $|m_{\Wb_{b_1}+\Wb_{1,1} \circ \Rb_1}(z)-m_{\Wb_{b_1}}(z) |$. Set $\mathcal{E}_1:=\Wb_{1,1} \circ \Rb_1$ for convenience. Note that with high probability, for some constant $C>0,$ we have that 
\begin{equation}\label{eq_boundedw11}
\max_{i,j} |\Wb_{1,1} (i,j)| \leq C. 
\end{equation}
By the definition of Stieltjes transform, for any even integer $q,$ we have 
\begin{align*} 
&|m_{\Wb_{b_1}+\mathcal{E}_1}(z)-m_{\Wb_{b_1}+\Rb_1}(z)|^q  \\
& \leq  \left( \frac{1}{n} \sum_{i=1}^n \frac{|\lambda_i(\Wb_{b_1}+\mathcal{E}_1)-\lambda_i(\Wb_{b_1}+\Rb_1)|}{|\lambda_i(\Wb_{b_1}+\mathcal{E}_1)-z||\lambda_i(\Wb_{b_1}+\Rb_1)-z|} \right)^{q}\,, \nonumber 
\end{align*}
which, by the Cauchy-Schwarz inequality, is bounded by 
\begin{align}
&\, \frac{1}{n^q} \bigg[ \left(\sum_{i=1}^n |\lambda_i(\Wb_{b_1}+\mathcal{E}_1)-\lambda_i(\Wb_{b_1}+\Rb_1)|^2 \right)\nonumber\\
&\qquad\qquad\times \left( \sum_{i=1}^n \frac{1}{|\lambda_i(\Wb_{b_1}+\mathcal{E}_1)-z|^2|\lambda_i(\Wb_{b_1}+\Rb_1)-z|^2} \right) \bigg]^{q/2} \nonumber \\
\leq &\, \frac{1}{n^{q/2} \eta^{2q}} ( \text{tr} \{(\mathcal{E}_1-\Rb_1)^2\} )^{q/2}, \label{eq_finalcontrolbound}
\end{align}
where the last inequality comes from the Hoffman-Wielandt inequality (Lemma \ref{lem_collectioninequality}) and the trivial bound $|\lambda-z|^{-1}\leq \eta^{-1}$ for any $\lambda\in \mathbb{R}$. The rest of the proof leaves to control the right-hand side of (\ref{eq_finalcontrolbound}). By (\ref{eq_boundedw11}) and the fact that $\mathcal{E}_1-\Rb_1=(\Wb_{1,1}-\mathbf{1}\mathbf{1}^\top)\circ \Rb_1$.
Hence, since $\mathcal{E}$ is symmetric, for some constant $C>0,$ we can replace (\ref{eq_finalcontrolbound}) with
\begin{equation}
\frac{C}{n^{q/2} \eta^{2q}} [ \text{tr} (\Rb_1^2)]^{q/2}\,.\label{eq_finalcontrolboundqq}
\end{equation} 
Thus, we only need to consider the $\Rb_1$ part,
and we claim that \eqref{eq_finalcontrolboundqq} can be controlled by a term of order $(\log n)^q n^{-q/2} \eta^{-2q}$, where the proof can be found in \cite[eq. (C.8)]{DW1}. 
To finish the proof, by the same argument of controlling $|m_{\Wb_{b_1}+\mathcal{E}_1}(z)-m_{\Wb_{b_1}+\Rb_1}(z)|^q$, we get $|m_{\Wb_{b_1}+\Rb_1}(z)-m_{\Wb_{b_1}}(z)|^q$ controlled by a term of order $(\log n)^q n^{-q/2} \eta^{-2q}$. By collecting the above controls, including $|m_{\Wb_{b_1}+\Eb_1+\Eb_2}(z)-m_{\Wb_{b_1}+\Eb_1}(z)|$ and $|m_{\Wb_{b_1}+\Eb_1}(z)-m_{\Wb_{b_1}}(z)|$, we conclude the bound for $|m_{\Wb_{b_1}+\Eb_1+\Eb_2}(z)-m_{\Wb_{b_1}}(z)|$. To finish the proof, we control $|m_{\Wb}(z)-m_{\Wb_{b_1}+\Eb_1+\Eb_2}(z)|$, which depends on controlling $\Eb_0$. We can use an analogous argument to control $\Eb_0$ by decomposing $\Wb_1$ as in (\ref{eq_w1decomposition}) and using (\ref{eq_crosstermbound}) and Lemma \ref{lem_collectioninequality} and a discussion similar to
(\ref{eq_finalcontrolbound}). We only sketch the proof. 
Recall (\ref{eq_R1R1R1DFN}). Note that we have 
\begin{align*}
\Eb_0 =&\,(\Wb_c-\mathbf{1}\mathbf{1}^\top) \circ \Ub_y \circ \Wb_{1,1}\\
&+(\Wb_c-\mathbf{1}\mathbf{1}^\top) \circ \Ub_y \circ \Wb_{1,2} \\
&+(\Wb_c-\mathbf{1}\mathbf{1}^\top) \circ \Rb_1 \circ \Wb_{1,1}\\
&+(\Wb_c-\mathbf{1}\mathbf{1}^\top) \circ \Rb_1 \circ \Wb_{1,2}. 
\end{align*}
We explain how to control the first term, which is the leading order term.  Observe that 
\begin{align*}
&(\Wb_c-\mathbf{1}\mathbf{1}^\top) \circ \Ub_y \circ \Wb_{1,1}\\
=\,&\exp(-2 \upsilon)(\Wb_c-\mathbf{1}\mathbf{1}^\top) \circ \Wb_{1,1} \\
& +2 \upsilon \exp(-4 \upsilon) [(\Wb_c-\mathbf{1}\mathbf{1}^\top)\circ \Ib_n] [\Wb_{1,1}\circ \Ib_n]  \\
&+(\Wb_c-\mathbf{1}\mathbf{1}^\top) \circ \frac{2 \upsilon \exp(-2\upsilon)}{p} \Yb^\top \Yb \circ \Wb_{1,1}.
\end{align*}
First, by (\ref{eq_boundedw11}) and (\ref{eq_crosstermbound}), the operator norm of the second term can be bounded by $n^{\alpha/2-1}$ and hence the differences of the Stieltjes transform can be bounded similarly as in (\ref{eq_finalcontrole2q}). Second, the first and third terms can be bounded using a discussion similar to (\ref{eq_finalcontrolbound}) with the counterpart of (\ref{eq_r1bound}), which read as 
\begin{align*}
\max_{i,j} \left| \left[\exp(-2 \upsilon)(\Wb_c-\mathbf{1}\mathbf{1}^\top) \circ \Wb_{1,1} \right](i,j) \right| \prec n^{\alpha/2-1}
\end{align*}
and
\begin{align*}
&  \max_{i,j} \left| \left[(\Wb_c-\mathbf{1}\mathbf{1}^\top) \circ \frac{2 \upsilon \exp(-2\upsilon)}{p} \Yb^\top \Yb \circ \Wb_{1,1} \right](i,j) \right| \prec n^{\alpha/2-1}.
\end{align*}
This completes our proof. 

\end{proof}

\bibliographystyle{abbrv}
\bibliography{sensornonnull}

\end{document}